\documentclass[12pt,twoside]{amsart}
\usepackage{mathrsfs}
\usepackage{txfonts}
\usepackage{bbm}
\usepackage{amsfonts}
\usepackage{amsmath}
\numberwithin{equation}{section}
\usepackage{esint}
\usepackage{amssymb}
\usepackage{wasysym}
\usepackage{psfrag}
\usepackage{graphics,epsfig,amsmath}
\usepackage{comment}
\usepackage{enumerate}
\usepackage{times}
\usepackage{fancyhdr}
\usepackage{tikz-cd}
\usepackage{cite}
\newbox\xratbelow
\newbox\xratabove

\newtheorem{Theorem}{Theorem}[section]
\newtheorem{Lemma}[Theorem]{Lemma}

\newtheorem{remark}[Theorem]{Remark}

\usepackage{metalogo}
\selectfont
\usepackage{mathtools}
\mathtoolsset{showonlyrefs}
\renewcommand{\epsilon}{\varepsilon}

\newcommand{\R}{\mathbb{R}}

\renewcommand{\le}{\leqslant}

\renewcommand{\ge}{\geqslant}

\allowdisplaybreaks

\usepackage{geometry}
\title[Mixed order concentration phenomena]{Concentration phenomena \\for a mixed local/nonlocal  Schr\"{o}dinger equation \\with Dirichlet datum}

\author[S. Dipierro]{Serena Dipierro}
\address{Department of Mathematics and Statistics, University of Western Australia, 35 Stirling Highway, WA 6009 Crawley, Australia}
\email{serena.dipierro@uwa.edu.au}

\author[X. Su]{Xifeng Su}
\address{School of Mathematical Sciences, Laboratory of Mathematics and Complex Systems (Ministry of Education)\\
	Beijing Normal University,
	No. 19, XinJieKouWai St., HaiDian District, Beijing 100875, P. R. China}
\email{xfsu@bnu.edu.cn, billy3492@gmail.com}

\author[E. Valdinoci]{Enrico Valdinoci}
\address{Department of Mathematics and Statistics, University of Western Australia, 35 Stirling Highway, WA 6009 Crawley, Australia}
\email{enrico.valdinoci@uwa.edu.au}

\author[J. Zhang]{Jiwen Zhang}
\address{School of Mathematical Sciences, Laboratory of Mathematics and Complex Systems (Ministry of Education)\\
	Beijing Normal University,
	No. 19, XinJieKouWai St., HaiDian District, Beijing 100875, P. R. China}
\email{jwzhang826@mail.bnu.edu.cn, jiwen.zhang@uwa.edu.au}


\begin{document}
	\maketitle
	
	\begin{abstract}
	We consider the mixed local/nonlocal semilinear equation 
	\begin{equation*}
			-\epsilon^{2}\Delta u +\epsilon^{2s}(-\Delta)^s u +u=u^p\qquad  \text{in } \Omega
	\end{equation*}
with zero Dirichlet datum, where $\epsilon>0$ is a small parameter, $s\in(0,1)$, $p\in(1,\frac{n+2}{n-2})$ and $\Omega$ is a smooth, bounded domain.  We construct a family of solutions that concentrate,  as $\epsilon\rightarrow 0$,  at an interior point of $\Omega$ having uniform distance to $\partial\Omega$ (this point can also be characterized as a local
minimum of a nonlocal functional).

In spite of the presence of the Laplace operator, the leading order of the relevant reduced energy functional in the Lyapunov-Schmidt procedure is polynomial rather than exponential in the distance to the boundary, in light of the nonlocal effect at infinity.  A delicate analysis is required to establish some uniform estimates with respect to~$\epsilon$, due to the difficulty caused by 
the different scales coming from 
the mixed operator.
	\end{abstract}
	
	\tableofcontents
	
	\section{Introduction}
Given~$n\ge2$, $s\in(0,1)$ and a bounded, smooth domain $\Omega\subset\mathbb{R}^n$, we are concerned with the following mixed local/nonlocal equation 
\begin{equation}\label{vfveffd}
\begin{cases}
	-\epsilon^{2}\Delta u +\epsilon^{2s}(-\Delta)^s u +u={u}^p\qquad  &\text{in } \Omega,\\
	u=0& \text{in } \mathbb{R}^n\setminus\Omega,
\end{cases}
\end{equation}
where $\epsilon>0$ is a small parameter.
Also, here above and in what follows, we assume that~$p\in (1,+\infty)$ if~$n=2$ and~$p\in (1, \frac{n+2}{n-2})$ if $n\geqslant 3$. 

Moreover, for any~$s\in(0,1)$, the fractional Laplacian is defined as
$$ (-\Delta)^s u(x):={\rm{P.V.}}\int_{\R^n}\frac{u(x)-u(y)}{|x-y|^{n+2s}}\,dy,$$
where P.V. means that the integral is taken in the Cauchy principal value sense.

The purpose of this paper is to construct a family of solutions $u_\epsilon$ to~\eqref{vfveffd} that concentrate at a point away from the boundary  $\partial\Omega$ for $\epsilon>0$ sufficiently small, namely,
\begin{equation*}
	u_\epsilon(x)\approx w\left(\frac{x-\xi_\epsilon}{\epsilon}\right).
\end{equation*}
Here, $\xi_\epsilon$ is some interior point of $\Omega$ such that dist$({\xi}_\epsilon,\partial\Omega)\geqslant c$ for some constant $c>0$ independent of $\xi_\epsilon$, and $w$ is the unique radial  ground state (least energy solution)  of \begin{equation}\label{vsdvdsvsd}
	\begin{cases}
		-\Delta w +(-\Delta)^s w +w=w^p \quad \text{in }\mathbb{R}^n, \\ 
	w> 0,\quad 	w\in H^1(\mathbb{R}^n).
	\end{cases}
\end{equation}
See~\cite{DSVZ24} for the existence of such a solution and some basic properties,
which are summarized below:
\begin{itemize}
	\item[(i)] $w\in C^2(\mathbb{R}^n)$, 
	$w$ is radial (hence, with a slight abuse of notations, we write
	$ w(x)=w(|x|)$ for every $ x\in\mathbb{R}^n$), and $w$ is decreasing in $r=|x|$;
	\item[(ii)] there exists a  positive constant $C$ such that, for all~$ x\in\mathbb{R}^n$,  \begin{equation}\label{bvsvs}
		\frac{C}{(1+|x|)^{n+2s}}\leqslant w(x)\leqslant \frac{1}{C(1+|x|)^{n+2s}} .
	\end{equation}
\end{itemize}

A key assumption in our setting, that will be assumed from now on, is that the linearized operator at~$w$ is nondegenerate, i.e., its kernel is exhausted by the derivatives of~$w$ and
their linear combinations, namely
\begin{equation}\label{NOBS-DEFG} \operatorname{Ker} \Big(-\Delta+(-\Delta)^s+1-pw^{p-1}\Big) = \operatorname{span} \big\{
\partial_i w,\; i\in\{1,\dots,n\}\big\}.\end{equation}
  
We now state the main result of this paper.
\begin{Theorem}\label{th main theorem}
	If  $\epsilon>0$ is sufficiently small, there exist a point ${\xi}_\epsilon\in\Omega$ with dist$({\xi}_\epsilon,\partial\Omega)\geqslant c$, 
	 and  a solution ${u}_\epsilon$ of problem~\eqref{vfveffd} such that 
\begin{equation}\label{fvvsf}
		\left|u_\epsilon(x)-w\left(\frac{x-{\xi}_\epsilon}{\epsilon}\right)\right|\leqslant C\epsilon^{\gamma_1},
	\end{equation}
for a suitable constant $\gamma_1>0$ depending only on~$n$, $s$ and $p$.   

Here, $c$ and $C$ are positive constants
depending only on~$n$, $s$, $p$ and~$\Omega$.
%
\end{Theorem}

\begin{remark}{\rm
The constant $\gamma_1$ in Theorem~\ref{th main theorem} can also be
described in further detail (see the forthcoming
Theorem~\ref{th:nonlinear problem }). In particular, one can show that~$\gamma_1>\frac{n}{2}+2s$, that~$\gamma_1\rightarrow \frac{n}{2}+2s$
as~$p\to 1 $, and that~$\gamma_1={n}+2s$ if~$p\geqslant 2 $.}\end{remark}

\begin{remark}{\rm Regarding assumption~\eqref{NOBS-DEFG},
the nondegeneracy of the ground state~$w$ of problem~\eqref{vsdvdsvsd} plays a pivotal role in proving Theorem~\ref{th main theorem}. Presently, in view of  \cite[Theorems~1.3 and~1.4]{SZZ24}, we know that  this  nondegeneracy property (together with
a uniqueness result) holds when~$s$ is close to either~$0$ or~$1$. 

In particular, Theorem~\ref{th main theorem} holds true if, instead of
assumption~\eqref{NOBS-DEFG}, one considers fractional exponents~$s$
sufficiently close to either~$0$ or~$1$.

It remains an interesting  open question whether  the nondegeneracy and uniqueness results for the ground states of problem~\eqref{vsdvdsvsd} hold true for any $s\in(0,1)$. 

Also, for the way we formulated our main result, if one proved that some
fractional exponent~$s_\star$
satisfies the nondegeneracy property~\eqref{NOBS-DEFG},
then the concentration result in Theorem~\ref{th main theorem} would automatically hold true for~$s_\star$.}\end{remark}

Concentration phenomena have been addressed 
in the literature 
under different perspectives,
see for instance \cite{MR3393677,MR1931757,MR1471107,MR2270164,MR3350616}, and the
references therein. The construction of concentrating solutions is essentially carried out by two methods. The first approach is to use critical point theory or topological techniques. A second alternative is to reduce the problem to a finite-dimensional one by means of a Lyapunov-Schmidt technique. In this paper, we follow the spirit of \cite{MR3393677} by adopting the second approach to search for the  solutions to problem~\eqref{vfveffd}.
\medskip

Let us now briefly describe the strategy behind proving the main result stated in Theorem~\ref{th main theorem}. Setting  $\Omega_\epsilon:={\Omega}/{\epsilon}=\left\{{x}/{\epsilon}:\, x\in\Omega\right\}$ and $\tilde{u}(x):=u(\epsilon x)$, problem~\eqref{vfveffd} becomes
\begin{equation}\label{vsdvdscsc}
	\begin{cases}
		-\Delta \tilde{u} +(-\Delta)^s \tilde{u}+ \tilde{u}=\tilde{u}^p\qquad  &\text{in } \Omega_\epsilon,\\
\tilde{u}=0& \text{in } \mathbb{R}^n\setminus\Omega_\epsilon.
	\end{cases}
\end{equation}
Furthermore, for a given $\xi\in\Omega_\epsilon$, we set $w_\xi(x):=w(x-\xi)$
for all~$x\in\mathbb{R}^n$. We will seek the solution of~\eqref{vsdvdscsc} near an approximate solution $\bar{u}_\xi$, for a suitable point $\xi=\xi_\epsilon$, solving the linear problem 
\begin{equation}\label{sdvbsd}
	\begin{cases}
		-\Delta \bar{u}_\xi +(-\Delta)^s \bar{u}_\xi +\bar{u}_\xi=w_\xi^p\qquad  &\text{in } \Omega_\epsilon,\\
		\bar{u}_\xi=0& \text{in } \mathbb{R}^n\setminus\Omega_\epsilon.
	\end{cases}
\end{equation}
It is known that $\bar{u}_\xi\in C^{2,\alpha}_{\rm loc}(\Omega_\epsilon)\cap C(\mathbb{R}^n)$ and $\bar{u}_\xi\geqslant 0$ in $\mathbb{R}^n$ (see for instance \cite{SVWZ23}).

 We observe that problem~\eqref{vsdvdscsc} is variational,  and we consider the following space
\begin{equation*}
H_0^1(\Omega_\epsilon):=\left\{u\in H^1(\mathbb{R}^n)\, :\, u=0 \text{ a.e. in } \mathbb{R}^n\setminus \Omega_\epsilon \right\}
\end{equation*}
equipped with the norm 
\begin{equation*}
	\|u\|^2_{H^1_0}:=\int_{\Omega_\epsilon}\left(|\nabla u(x)|^2+u^2(x)\right)\, dx+\int_{\mathbb{R}^n}\int_{\mathbb{R}^n}\frac{|u(x)-u(y)|^2}{|x-y|^{n+2s}}\, dx\,dy.
\end{equation*}
Moreover, problem~\eqref{vsdvdscsc} corresponds to the Euler-Lagrange equation for the functional
\begin{equation}\label{functional}
	I_\epsilon(u):=\frac{1}{2}\|u\|^2_{H^1_0}-\frac{1}{p+1}\int_{\Omega_\epsilon}u^{p+1}(x)\, dx.
\end{equation}

We will provide an expansion of $I_\epsilon(\bar{u}_\xi)$ in Section~\ref{sec:Energy estimates and functional expansion}: more precisely,
given~$\delta>0$,
when~$\operatorname{dist}(\xi,\partial\Omega_\epsilon)\geqslant \delta/\epsilon$
we have that
\begin{equation}\label{sdv}
	I_\epsilon(\bar{u}_\xi)=I(w)+\frac{1}{2}\mathcal{H}_\epsilon(\xi)+o(\epsilon^{n+4s})
\end{equation}
where $I$ is the energy of $w$ given in~\eqref{dsvsbs},  and $\mathcal{H}_\epsilon(\xi)$ is defined by
\begin{equation}\label{bjbca}
	\mathcal{H}_\epsilon(\xi):=	\int_{\Omega_\epsilon}w^p_\xi(x)\left(w_\xi(x)-\bar{u}_\xi(x)\right)\, dx.
\end{equation}
Additionally, we will show that,  for every $\xi\in\Omega_\epsilon$ with $\operatorname{dist}(\xi,\partial\Omega_\epsilon)\in [2,\rho/\epsilon]$ (the constant $\rho>0$ being
appropriately small),  
 there exist two positive constants  $\bar{c}_1$ and $\bar{c}_2$ such that 
\begin{equation}\label{jkn}
	\frac{\bar{c}_2}{\operatorname{dist}(\xi,\partial\Omega_\epsilon)^{n+4s}}\leqslant \mathcal{H}_\epsilon(\xi)\leqslant \frac{\bar{c}_1}{\operatorname{dist}(\xi,\partial\Omega_\epsilon)^{n+4s}}.
\end{equation}

We remark that, by taking into account the fact that $w$ is the least energy solution,  one deduces from~\eqref{sdv} and~\eqref{jkn} that, for $\epsilon>0$ sufficiently small, there exists a local minimizer $\xi_\epsilon$ for 
 the functional $I_\epsilon(\bar{u}_\xi)$, which is located in a compact set with uniform distance  to $\partial\Omega$. 
 
{F}rom this,  one is able to reduce  the  full problem,  by means of a Lyapunov-Schmidt technique, to that of finding a critical point $\xi_\epsilon$ of a functional which has a similar expansion to~\eqref{sdv} (see Section~\ref{sec:The Lyapunov-Schmidt reduction} for the details).

In the local case of the Laplace operator (see~\cite{MR1737546,MR1639159,MR1342381,MR1404386}) one can
construct solutions of the classical singularly perturbed Dirichlet problem that concentrate, as $\epsilon\rightarrow 0$,  near points that have maximum distance to the boundary
(this  problem with zero
Neumann boundary condition was also studied in  \cite{MR2765513,MR2275329,MR1923818,MR1219814}). In the nonlocal
Dirichlet case, however, the localization of
the concentration point is more delicate and depends on the local minimization
of a nonlocal functional (and this feature persists in the case
treated in this paper).

In the nonlocal case, in~\cite{MR3121716}, the existence of multiple spike solutions
  of a fractional Schr\"{o}dinger equation in~$\mathbb{R}^n$ was considered. Also, in~\cite{MR3393677} a family of solutions  that concentrate
   at an interior point of the domain are constructed in the form of a scaling of a global ground state. See additionally~\cite{MR3350616,MR3059423,MR4018100,MR3522330} and the
   references therein for related problems about 
   concentration phenomena
   and fractional Schr\"{o}dinger equations.

In contrast to  the both situations above, in which classical or nonlocal
operators of a given order are considered,
the study of concentration phenomena in a mixed local/nonlocal framework remained
open, and this paper aims at filling this gap in the literature. In this situation,  the presence of the nonlocal term~$(-\Delta)^s$, though of lower order with respect to the classical Laplacian, causes the leading order of the relevant energy expansion to be polynomial  instead of exponential, making the information about the concentration point less explicit than in the classical case
and the analysis dependent on both the scales of the two elliptic operators involved in the problem.

Also, in the Lyapunov-Schmidt reduction, the regularity requirement for the solutions is higher than that of the fractional case, 
thanks to the effect of the Laplacian operator (see Sections~\ref{sec:Decay of the ground state}-~\ref{sec:The Lyapunov-Schmidt reduction} for the details), therefore the analysis needed
to deal with mixed operators of  different orders
is more involved  than the one addressing separately the cases 
of the Laplacian and of the fractional Laplacian.
 
The rest of this paper is organized as follows. In Section~\ref{sec:Energy estimates and functional expansion}, we first estimate the function~$\mathcal{H}_\epsilon$, thus obtaining an approximation of the energy expansion on $\bar{u}_\xi$. Sections~\ref{sec:Decay of the ground state} and~\ref{sec:Some regularity estimates} are devoted to some decay  and regularity estimates for computing the small error terms in the expansions of the actual solution. 

In
Sections~\ref{sec:The Lyapunov-Schmidt reduction} and ~\ref{sec:Proof of Theorem 1.1} we find a suitable perturbation in the vicinity of the approximate solution via a Lyapunov-Schmidt reduction method, and complete the proof of Theorem~\ref{th main theorem}.

\section{Energy estimates and functional expansion at~$\bar{u}_\xi$}\label{sec:Energy estimates and functional expansion}

In this section, we focus on the asymptotic expansion of $I_\epsilon(\bar{u}_\xi)$.
Let us first consider the  functional associated with problem~\eqref{vsdvdsvsd}, namely
\begin{equation}\label{dsvsbs}
	I(u):=\frac{1}{2}\int_{\mathbb{R}^n}\left(|\nabla u(x)|^2+u^2(x)\right)\,dx+\frac{1}{2}\int_{\mathbb{R}^n}\int_{\mathbb{R}^n}\frac{|u(x)-u(y)|^2}{|x-y|^{n+2s}}\, dxdy-\frac{1}{p+1}\int_{\mathbb{R}^n}u^{p+1}(x)\, dx,
\end{equation}
for every $u\in H^1(\mathbb{R}^n)$.

We will prove that, for dist$(\xi,\partial\Omega_\epsilon)\geqslant \delta/\epsilon$ with $\delta>0$ fixed,
\begin{equation*}
	I_\epsilon(\bar{u}_\xi)=I(w)+\frac{1}{2}\mathcal{H}_\epsilon(\xi)+o(\epsilon^{n+4s})
\end{equation*}
where $\mathcal{H}_\epsilon(\xi)$ is defined by~\eqref{bjbca}.

To this end, we notice that
the expansion of $I_\epsilon(\bar{u}_\xi)$ involves 
the fundamental solution  $\mathcal{K}$ for the operator $-\Delta+(-\Delta)^s+1$ in $\mathbb{R}^n$, namely the solution of
\begin{equation}\label{mathcal{K}}
	\Big(-\Delta+(-\Delta)^s+1\Big)\mathcal{K}=\delta_0\qquad \text{in } \mathbb{R}^n.
\end{equation}
We recall (see for instance \cite{DSVZ24}) that the fundamental solution $\mathcal{K}$ satisfies the following basic properties:
\begin{itemize}
	\item[(i)] 
	$\mathcal{K}$ is radially symmetric, positive, non-increasing in $r=|x|$, and $\mathcal{K}\in C^\infty(\mathbb{R}^n\setminus\left\{0\right\})$;
	\item[(ii)]  there exist two positive constants  $C_1,C_2>0$ such that, for all~$ |x|>1$,
	\begin{equation}\label{decaying}
		\frac{C_1}{|x|^{n+2s}}\leqslant \mathcal{K}(x)\leqslant
		\frac{1}{C_1|x|^{n+2s}},
	\end{equation}
	\begin{equation}\label{dfdsfs}
		|\nabla \mathcal{K}(x)|\leqslant \frac{C_2}{|x|^{n+2s+1}}\quad \text{and} \quad 	|\nabla^2 \mathcal{K}(x)|\leqslant \frac{C_2}{|x|^{n+2s+2}} .
	\end{equation}
\end{itemize}

We next  make use of  the fundamental solution $\mathcal{K}$ to construct an auxiliary function to obtain some estimates on $\mathcal{H}_\epsilon$. 

\subsection{Estimates on $\mathcal{H}_\epsilon(\xi)$}
We recall that~$w_\xi(x):=w(x-\xi)$ and we denote 
\begin{equation}\label{dbsfbs}
	v_\xi(x):=w_\xi(x)-\bar{u}_\xi(x).
\end{equation}
Owing to~\eqref{vsdvdsvsd} and~\eqref{sdvbsd}, we see that $	v_\xi\in C^{2,\alpha}_{\rm loc}(\Omega_\epsilon)\cap C(\mathbb{R}^n) $ solves
\begin{equation}\label{vdsvd}
	\begin{cases}
		-\Delta v_\xi +(-\Delta)^s v_\xi +v_\xi=0 \qquad  &\text{in } \Omega_\epsilon,\\
		v_\xi=w_\xi& \text{in } \mathbb{R}^n\setminus\Omega_\epsilon.
	\end{cases}
\end{equation}

Moreover, given $\xi\in\Omega_\epsilon$ and $x\in\mathbb{R}^n$, we define 
\begin{equation*}
	h_\xi(x):=\int_{\mathbb{R}^n\setminus\Omega_\epsilon}\mathcal{K}(x-z)\mathcal{K}(\xi-z)\, dz.
\end{equation*}
We observe that, for any $x\in \Omega_\epsilon$ and $z\in\mathbb{R}^n\setminus\Omega_\epsilon$,
\begin{equation*}
	\left(-\Delta+(-\Delta)^s+1\right)\mathcal{K}(x-z)=\delta_0(x-z)=0.
\end{equation*}
Thus, one has that
\begin{equation}\label{vsdv ds}
	\left(-\Delta+(-\Delta)^s+1\right)h_\xi=0\quad \text{in } \Omega_\epsilon.
\end{equation}

Our purpose is to employ $h_\xi$  as a barrier, from above and below, for the function $v_\xi$,  using~\eqref{vdsvd},~\eqref{vsdv ds} and the Maximum Principle (see e.g.\cite{SVWZ23}). More precisely, we estimate the behavior of $v_\xi$ in $\mathbb{R}^n$ as follows:

\begin{Lemma}\label{lemma hepsilon}
	There exists a constant $c_1>0$ such that 
	\begin{equation}\label{fdf}
		h_\xi(x)\leqslant c_1(1+|x-\xi|)^{-(n+2s)}
	\end{equation}
for any $x\in\mathbb{R}^n$ and $\xi\in\Omega_\epsilon$ with $d:=\operatorname{dist}(\xi,\partial\Omega_\epsilon)\geqslant 1.$
\end{Lemma}

\begin{proof}
Since $\mathcal{K}\in L^1(\mathbb{R}^n)$ and~$|\xi-z|>1$ for any $z\in\mathbb{R}^n\setminus\Omega_\epsilon$, owing to~\eqref{decaying},
	we deduce that, if~$|x-\xi|>1$,
	\begin{equation*}
		\begin{split}
			&h_\xi(x)
			\leqslant C\left(
			\int_{B_{ |x-\xi|/2}(x)}\mathcal{K}(x-z)|z-\xi|^{-(n+2s)}\, dz
			+\int_{\R^n\setminus B_{ |x-\xi|/2}(x)}\mathcal{K}(\xi-z)|x-z|^{-(n+2s)}\, dz\right)\\
			&\qquad\qquad \leqslant C |x-\xi|^{-(n+2s)}\int_{\mathbb{R}^n}\mathcal{K}(z)\, dz\leqslant C |x-\xi|^{-(n+2s)}
		\end{split}
	\end{equation*}
	up to renaming $C>0$. 
	
	Furthermore, for every $x\in\mathbb{R}^n$, 
	\begin{equation}\label{fvfdv}
		\begin{split}
				h_\xi(x)\leqslant C\int_{\R^n\setminus B_d(\xi)} \mathcal{K}(x-z)|z-\xi|^{-(n+2s)}\, dz\leqslant C d^{-(n+2s)}\int_{\mathbb{R}^n}\mathcal{K}(z)\, dz\leqslant Cd^{-(n+2s)}
		\end{split}
	\end{equation}
	up to renaming $C>0$, due to the fact that $\mathcal{K}\in L^1(\mathbb{R}^n)$. 
	
	The above two formulas entail the desired result in~\eqref{fdf}.
\end{proof}

\begin{Lemma}\label{lemma comparable for h}
	There exists $\tilde{c}\in(0,1)$ such that, for any $x\in\mathbb{R}^n$ and $\xi\in\Omega_\epsilon$ satisfying
\begin{equation}\label{vdsvsdvs}
	\operatorname{dist}(\xi,\partial\Omega_\epsilon)\geqslant 1,
\end{equation}
we have that
	\begin{equation}\label{gsdfds}
		\tilde{c}h_\xi(x)\leqslant v_\xi(x)\leqslant \tilde{c}^{-1}h_\xi(x).
	\end{equation}
\end{Lemma}

\begin{proof}
	{F}rom~\cite[formula~(2.5)]{MR3393677}, we know that, for any $x\in\mathbb{R}^n\setminus\Omega_\epsilon $,
	\begin{equation}\label{dvsd}
		\left|B_{1/2}(x)\setminus\Omega_\epsilon\right|\geqslant c_\star,
	\end{equation}
for a suitable $c_\star>0$ independent of $x$ and $\epsilon$.

Now, given $x\in\mathbb{R}^n\setminus\Omega_\epsilon $, recalling~\eqref{bvsvs} and~\eqref{decaying}, by combining~\eqref{vdsvsdvs} and~\eqref{dvsd}, we obtain that
\begin{equation}\label{vdsvsddsv}
	\begin{split}
	&h_\xi(x)\geqslant \int_{B_{1/2}(x)\setminus\Omega_\epsilon}\mathcal{K}(x-z)\mathcal{K}(\xi-z)\, dz\geqslant C \int_{B_{1/2}(x)\setminus\Omega_\epsilon}|\xi-z|^{-(n+2s)}\mathcal{K}(x-z)\, dz\\
	&\quad\geqslant C \int_{B_{1/2}(x)\setminus\Omega_\epsilon}|\xi-x|^{-(n+2s)}\mathcal{K}(x-z)\, dz\geqslant Cw_\xi(x) \int_{B_{1/2}(x)\setminus\Omega_\epsilon}\mathcal{K}(x-z)\, dz\\
	&\quad\geqslant Cw_\xi(x)\inf\limits_{y\in B_{1/2}}\mathcal{K}(y)\left|B_{1/2}(x)\setminus\Omega_\epsilon\right|\geqslant c_\star Cw_\xi(x)\inf\limits_{y\in B_{1/2}}\mathcal{K}(y)
	\end{split}
\end{equation}
up to renaming $C$ from line to line, independently of $x,\epsilon$ and $\xi$.

Additionally, using~\eqref{bvsvs},~\eqref{decaying} and~\eqref{vdsvsdvs}, given $x\in\mathbb{R}^n\setminus\Omega_\epsilon $, one has that
\begin{equation}\label{sdvsd}
	\begin{split}
		h_\xi(x)&\leqslant\int_{B_{|x-\xi|/2}(x)\setminus\Omega_\epsilon}\mathcal{K}(x-z)\mathcal{K}(\xi-z)\, dz+\int_{\mathbb{R}^n\setminus B_{|x-\xi|/2}(x)}\mathcal{K}(x-z)\mathcal{K}(\xi-z)\, dz\\
		&\leqslant  C\left(\int_{B_{|x-\xi|/2}(x)\setminus\Omega_\epsilon}\frac{\mathcal{K}(x-z)}{|\xi-z|^{n+2s}}\, dz+\int_{\mathbb{R}^n\setminus B_{|x-\xi|/2}(x)}\frac{\mathcal{K}(\xi-z)}{|x-z|^{n+2s}}\, dz\right)\\
		&\leqslant C2^{n+2s}\left(\int_{B_{|x-\xi|/2}(x)\setminus\Omega_\epsilon}\frac{\mathcal{K}(x-z)}{|\xi-x|^{n+2s}}\, dz+\int_{\mathbb{R}^n\setminus B_{|x-\xi|/2}(x)}\frac{\mathcal{K}(\xi-z)}{|x-\xi|^{n+2s}}\, dz\right)\\
		&\leqslant C w_\xi(x)\int_{\mathbb{R}^n}\mathcal{K}(x)\,dx,
	\end{split}
\end{equation}
up to renaming $C$ independently of $x,\epsilon$ and $\xi$, from line to line. 

Moreover, by combining~\eqref{vdsvsddsv} with~\eqref{sdvsd}, we obtain that, for any $x\in\mathbb{R}^n\setminus\Omega_\epsilon$,
\begin{equation*}
	ch_\xi(x)\leqslant w_\xi(x)\leqslant c^{-1}h_\xi(x)
\end{equation*} 
for some constant $c$ independent of~$x,\epsilon$ and $\xi$. {F}rom this, in virtue of~\eqref{vdsvd}, \eqref{vsdv ds} and the Maximum Principle (see e.g. \cite{SVWZ23}), we obtain the desired result in~\eqref{gsdfds}.
\end{proof}

To show that the reduced functional $I_\epsilon(\bar{u}_\xi)$ has  a local minimum, we will prove that $\mathcal{H}_\epsilon$
attains, in a compact subset of $\Omega_\epsilon$, values smaller than the ones attained at the boundary.
This is described in the forthcoming Lemma~\ref{lemma compact set }, 
and the necessary upper and lower bounds for~$\mathcal{H}_\epsilon$ will be given in Lemmata~\ref{lemma bound of H1} and~\ref{lemma bound of H2}.

\begin{Lemma}\label{lemma bound of H1}
	Let  $\xi\in\Omega_\epsilon$ with 
	\begin{equation*}
		d:=\operatorname{dist}(\xi,\partial\Omega_\epsilon)>2.
	\end{equation*}
Then,  there exists $\bar{c}_1>0$ such that 
\begin{equation*}
 \mathcal{H}_\epsilon(\xi)\leqslant \frac{\bar{c}_1}{d^{n+4s}},
\end{equation*}
where $\mathcal{H}_\epsilon(\xi)$ is defined in~\eqref{bjbca}.
\end{Lemma}

\begin{proof}
{F}rom~\eqref{bjbca} and~\eqref{dbsfbs},  it follows that 
\begin{equation*}
	\mathcal{H}_\epsilon(\xi)=	\int_{\Omega_\epsilon}w^p_\xi(x)\left(w_\xi(x)-\bar{u}_\xi(x)\right)\, dx=\int_{\Omega_\epsilon}w^p_\xi(x)v_\xi(x)\, dx.
\end{equation*}
In view of
Lemma~\ref{lemma comparable for h}, one has that 
\begin{equation}\label{cghv}
	\mathcal{H}_\epsilon(\xi)\leqslant\tilde{c}^{-1}\int_{\Omega_\epsilon}w^p_\xi(x)h_\xi(x)\, dx.
\end{equation}

Moreover, we observe that, for any $x\in B_{d/2}(\xi)$ and $z\in \mathbb{R}^n\setminus\Omega_\epsilon$, \begin{equation*}
	|x-z|\geqslant |z-\xi|-|x-\xi|\geqslant \frac{|z-\xi|}2>
\frac{d}{2}.
\end{equation*} {F}rom this and~\eqref{decaying}, it follows that, for all~$x\in B_{d/2}(\xi)$,
\begin{equation*}
	\begin{split}
		&h_\xi(x)= \int_{\mathbb{R}^n\setminus\Omega_\epsilon} \mathcal{K}(x-z)\mathcal{K}(\xi-z)\, dz\leqslant \int_{\mathbb{R}^n\setminus B_d(\xi)} \mathcal{K}(x-z)\mathcal{K}(\xi-z)\, dz\\
		&\qquad\leqslant C\int_{\mathbb{R}^n\setminus B_d(\xi)} |z-\xi|^{-(2n+4s)}\, dz\leqslant \frac{C}{d^{n+4s}}, 
	\end{split}
\end{equation*} up to renaming~$C>0$.

As a consequence of this and~\eqref{bvsvs}, using~\eqref{decaying} again, we have that
\begin{equation}\label{chjhvh}
	\begin{split}
		&\int_{\Omega_\epsilon}w^p_\xi(x)h_\xi(x)\, dx\leqslant \int_{B_{d/2}(\xi)}w^p_\xi(x)h_\xi(x)\, dx+ \int_{\mathbb{R}^n\setminus B_{d/2}(\xi)}w^p_\xi(x)h_\xi(x)\, dx\\
		&\qquad\leqslant \frac{C}{d^{n+4s}}\int_{\mathbb{R}^n} w^p(x)\, dx+C\int_{\mathbb{R}^n\setminus B_{d/2}(\xi)}|x-\xi|^{-p(n+2s)}\left(d^{-(n+2s)}\int_{\mathbb{R}^n}\mathcal{K}(z)\, dz\right)\, dx\\
		&\qquad\leqslant C\left(\frac{1}{d^{n+4s}}+\frac{1}{d^{pn+2s(p+1)}}\right)\leqslant \frac{C}{d^{n+4s}},
	\end{split}
\end{equation}
where $C>0$ is a constant independent of~$\epsilon$ and~$\xi$,
possibly changing from line to line.

By combining~\eqref{cghv} with~\eqref{chjhvh}, we complete the proof of Lemma~\ref{lemma bound of H1}.
\end{proof}

\begin{Lemma}\label{lemma bound of H2}
	Let $\rho\in(0,1)$ and $\xi\in\Omega_\epsilon$ with 
	\begin{equation*}
		d:=\operatorname{dist}(\xi,\partial\Omega_\epsilon)\in \left[2,\frac{\rho}{\epsilon}\right].
	\end{equation*}
	
	Then, if $\rho>0$ is sufficiently small, only in dependence of~$\Omega$,
	there exists~$\bar{c}_2>0$ such that 
	\begin{equation}\label{sdcsdvsd}
	 \mathcal{H}_\epsilon(\xi)\geqslant 	\frac{\bar{c}_2}{d^{n+4s}},
	\end{equation}
	where $\mathcal{H}_\epsilon(\xi)$ is defined in~\eqref{bjbca}.
\end{Lemma}

\begin{proof}
	{F}rom
	Lemma~\ref{lemma comparable for h}, one has that 
	\begin{equation}\label{vdsfds}
	\mathcal{H}_\epsilon(\xi)
	=\int_{\Omega_\epsilon}w^p_\xi(x)v_\xi(x)\, dx
	\geqslant 	\tilde{c}\int_{\Omega_\epsilon}w^p_\xi(x)h_\xi(x)\, dx.
	\end{equation}

	We now claim that, for any $x\in B_{d/2}(\xi)$,
	\begin{equation}\label{sddsfsd}
		 h_\xi(x)\geqslant\frac{C}{d^{n+4s}},
	\end{equation}
	for some constant $C>0$ independent of~$x$, $\epsilon$ and~$\xi$.
	
	Indeed, since $\Omega$ has a smooth boundary, there exists $\rho_0>0$ such that every point of $\partial\Omega$ can be touched from outside by a ball of radius~$\rho_0$. Then,
	by scaling, we can touch~$\Omega_\epsilon$ from the outside with balls of radius~$\rho_0\epsilon^{-1}$, and so of radius~$d$ (indeed, if~$\rho$ is small enough, then~$d\leqslant \rho\epsilon^{-1}\leqslant \rho_0\epsilon^{-1}$).  
	
	Let $\eta\in\partial\Omega_\epsilon$ be
	such that $|\xi-\eta|=d$. We take a ball $\mathcal{B}$ of radius $d$ which touches $\Omega_\epsilon$ from outside at $\eta$ and we use~\eqref{decaying} to find that
	\begin{equation*}
		\begin{split}&
			h_\xi(x)= \int_{\mathbb{R}^n\setminus\Omega_\epsilon} \mathcal{K}(x-z)\mathcal{K}(\xi-z)\, dz\geqslant C \int_{\mathbb{R}^n\setminus \Omega_\epsilon} |z-\xi|^{-(2n+4s)}\, dz\\
			&\qquad\qquad \geqslant C\int_{\mathcal{B}} d^{-(2n+4s)}\, dz\geqslant \frac{C}{d^{n+4s}},
		\end{split}
	\end{equation*} up to renaming~$C>0$.
	This shows the validity of the claim in~\eqref{sddsfsd}.
	
	Owing to~\eqref{bvsvs} and~\eqref{sddsfsd}, by taking into account the fact that $w(x)=w(|x|)$ is non-increasing in $r=|x| $, 
 one has that
	\begin{equation*}
		\int_{\Omega_\epsilon}w^p_\xi(x)h_\xi(x)\, dx \geqslant \int_{B_{d/2}(\xi)}w^p_\xi(x)h_\xi(x)\, dx\geqslant \frac{C}{d^{n+4s}}\int_{B_{1}(\xi)} w^p(|x-\xi|)\, dx\geqslant \frac{C}{d^{n+4s}},
	\end{equation*}
	where $C>0$ is a constant independent of $\epsilon$ and $\xi$.
	As a consequence of this and~\eqref{vdsfds},  we obtain the desired result in~\eqref{sdcsdvsd}.
\end{proof}

We remark that the constant  $\rho$ given in Lemma~\ref{lemma bound of H2} only depends on~$\Omega$.

Now, we summarize the consequences of Lemmata~\ref{lemma bound of H1} and~\ref{lemma bound of H2} as follows:

\begin{Lemma}\label{lemma compact set }
Let $\delta\in(0,1)$ and 
\begin{equation}\label{ngdjs}
	\Omega_{\epsilon,\delta}:=\big\{x\in\Omega_\epsilon\,:\, \operatorname{dist}(x,\partial\Omega_\epsilon)>\delta/\epsilon\big\}.
\end{equation}

Then, if~$\delta$ is sufficiently small, only in dependence of~$\Omega$,
there exist~$c_1$, $c_2>0$ such that 
\begin{equation*}
	\min\limits_{\Omega_{\epsilon,\delta}}\mathcal{H}_\epsilon\leqslant c_1\epsilon^{n+4s}< c_2\left(\frac{\epsilon}{\delta}\right)^{n+4s}\leqslant \min\limits_{\partial\Omega_{\epsilon,\delta}}\mathcal{H}_\epsilon
.\end{equation*}

In particular,
$\mathcal{H}_\epsilon$ attains an interior minimum in $\Omega_{\epsilon,\delta}$.
\end{Lemma}

\begin{proof}
	Let $ P_0\in\Omega$ be   such that \begin{equation*}
		d_0:=\operatorname{dist}(P_0,\partial\Omega)=\max\limits_{P\in\Omega}\operatorname{dist}(P,\partial\Omega). 
	\end{equation*}
By scaling, the maximal distance that a point of $\Omega_\epsilon$ may attain from the boundary of $\Omega_\epsilon$ is $d_0/\epsilon$. Let $P$ be such that 
\begin{equation*}
	\operatorname{dist}(P,\partial\Omega_\epsilon)=\frac{d_0}{\epsilon}.
\end{equation*}
For $\delta>0$ small enough, we have that $P\in\Omega_{\epsilon,\delta}$. {F}rom Lemma~\ref{lemma bound of H1}, it follows that 
\begin{equation}\label{bdfgsdg}
	\min\limits_{\Omega_{\epsilon,\delta}}\mathcal{H}_\epsilon\leqslant \mathcal{H}_\epsilon(P)\leqslant\frac{\bar{c}_1\epsilon^{n+4s}}{d_0^{n+4s}}.
\end{equation}

In addition, owing to Lemma~\ref{lemma bound of H2}, if~$\delta$
is sufficiently small, we see that 
\begin{equation}\label{dghdf}
	\min\limits_{\partial\Omega_{\epsilon,\delta}}\mathcal{H}_\epsilon\geqslant \frac{\bar{c}_2\epsilon^{n+4s}}{\delta^{n+4s}}.
\end{equation}
The desired result follows from~\eqref{bdfgsdg} and~\eqref{dghdf} for $\delta$ appropriately small.
\end{proof}

\subsection{Functional expansion}

Based on the above analysis, this section is devoted to obtaining an
asymptotic expansion of~$I_\epsilon(\bar{u}_\xi)$.

\begin{Theorem}\label{th energy estimates}
	Let $\delta\in(0,1)$ and $\xi\in\Omega_\epsilon$  with $d:=$ dist$(\xi,\partial\Omega_\epsilon)\geqslant \delta/\epsilon>2 $.
	
	Then, 
	\begin{equation}\label{vdsvsd}
		I_\epsilon(\bar{u}_\xi)=I(w)+\frac{1}{2}\mathcal{H}_\epsilon(\xi)+o(\epsilon^{n+4s})
	\end{equation}
	as $\epsilon\rightarrow 0$. Here $I$ is given in~\eqref{dsvsbs},
	$w$ is the solution of~\eqref{vsdvdsvsd} and $\mathcal{H}_\epsilon(\xi)$ is as defined in~\eqref{bjbca}.
\end{Theorem}

\begin{proof}
	{F}rom~\eqref{sdvbsd} and~\eqref{functional}, we see that 
	\begin{equation*}
		\begin{split}
			I_\epsilon(\bar{u}_\xi)&=\frac{1}{2}\|\bar{u}_\xi\|_{H^1_0}^2-\frac{1}{p+1}\int_{\Omega_\epsilon}\bar{u}_\xi^{p+1}(x)\, dx\\
			&=\frac{1}{2}\int_{\Omega_\epsilon}w^p_\xi(x)\bar{u}_\xi(x)\, dx-\frac{1}{p+1}\int_{\Omega_\epsilon}\bar{u}_\xi^{p+1}(x)\, dx.
		\end{split}
	\end{equation*}
	Since $w$ is a solution of~\eqref{vsdvdsvsd}, one  has that 
	\begin{equation*}
		I(w)=\left(\frac{1}{2}-\frac{1}{p+1}\right)\int_{\mathbb{R}^n}w_\xi^{p+1}(x)\, dx.
	\end{equation*}
This gives that
	\begin{equation}\label{dsvsdvds}
		\begin{split}
		I_\epsilon(\bar{u}_\xi)
		&=I(w)-\left(\frac{1}{2}-\frac{1}{p+1}\right)\int_{\mathbb{R}^n\setminus\Omega_\epsilon}w_\xi^{p+1}(x)\, dx-\frac{1}{2}\int_{\Omega_\epsilon}w^p_\xi(x)\left(w_\xi(x)-\bar{u}_\xi(x)\right)\, dx\\
		&\qquad \qquad\qquad +\frac{1}{p+1}\int_{\Omega_\epsilon}\left(w^{p+1}_\xi(x)-\bar{u}_\xi^{p+1}(x)\right)\, dx\\
		&=: I(w)-\left(\frac{1}{2}-\frac{1}{p+1}\right)A_1-\frac{1}{2}\mathcal{H}_\epsilon(\xi)+\frac{1}{p+1}A_2,
		\end{split}
	\end{equation}
where $\mathcal{H}_\epsilon(\xi)$ is given by~\eqref{bjbca}, 
\begin{equation*}
	A_1:=\int_{\mathbb{R}^n\setminus\Omega_\epsilon}w_\xi^{p+1}(x)\, dx\quad \text{and} \quad A_2:=\int_{\Omega_\epsilon}\left(w^{p+1}_\xi(x)-\bar{u}_\xi^{p+1}(x)\right)\, dx.
\end{equation*}

Let us first estimate $A_1$. Owing to~\eqref{bvsvs}, one deduces that
\begin{equation}\label{sdvsdv}
		A_1\leqslant \int_{\R^n\setminus B_d(\xi) }w_\xi^{p+1}(x)\, dx
		\leqslant \int_{\R^n\setminus B_d(\xi)}\frac{C}{|x-\xi|^{(n+2s)(p+1)}}\, dx\leqslant \frac{C}{d^{np+2s(p+1)}}\leqslant\frac{C\epsilon^{np+2s(p+1)}}{\delta^{np+2s(p+1)}}.
\end{equation} 

As for $A_2$, in the light of Lemma~\ref{lemma comparable for h},
we see that $ \bar{u}_\xi(x)< w_\xi(x)$ in $\mathbb{R}^n$. Thus, we expand~$w_\xi^{p+1}(x)$ in the following way
\begin{equation*}
	w_\xi^{p+1}(x)=\bar{u}^{p+1}_\xi(x)+(p+1)w_\xi^{p}(x)(w_\xi(x)-\bar{u}_\xi(x))+c_p\alpha_\xi^{p-1}(x)(w_\xi(x)-\bar{u}_\xi(x))^2
,\end{equation*}
where $0\leqslant\bar{u}_\xi\leqslant \alpha_\xi\leqslant w_\xi$ and $c_p$ is a positive constant depending on $p$. 

Therefore, we have that
\begin{equation}\label{bsdfds}
	A_2=(p+1)\mathcal{H}_\epsilon(\xi)+c_p\int_{\Omega_\epsilon}\alpha_\xi^{p-1}(x) (w_\xi(x)-\bar{u}_\xi(x))^2\, dx.
\end{equation}

Moreover, since $\alpha_\xi\leqslant w_\xi$ and $\bar{u}_\xi<w_\xi $, employing Lemma~\ref{lemma comparable for h}
and the decay estimates in~\eqref{bvsvs} and~\eqref{decaying}, one has that, when $1<p\leqslant 2$,
\begin{equation*}\begin{split}&
c_p\int_{\Omega_\epsilon}\alpha_\xi^{p-1}(x) (w_\xi(x)-\bar{u}_\xi(x))^2\, dx
		\leqslant  c_p\int_{\Omega_\epsilon}w_\xi(x) v^p_\xi(x)\, dx\\&\qquad \qquad
		\leqslant c_p \tilde{c}^{-p}
\int_{\Omega_\epsilon}w_\xi(x) h^p_\xi(x)\, dx
\leqslant Cd^{-p(n+4s)}\leqslant C\frac{\epsilon^{p(n+4s)}}{\delta^{p(n+4s)}}
\end{split}
\end{equation*}
and, when $p\geqslant 2$,
\begin{equation*}
	\begin{split}
		&c_p\int_{\Omega_\epsilon}\alpha_\xi^{p-1}(x) (w_\xi(x)-\bar{u}_\xi(x))^2\, dx
		\leqslant  c_p\int_{\Omega_\epsilon}w^{p-1}_\xi(x) v^2_\xi(x)\, dx\\
		&\qquad \qquad\leqslant c_p\tilde{c}^{-2}\int_{\Omega_\epsilon}w^{p-1}_\xi(x) h^2_\xi(x)\, dx\leqslant Cd^{-2(n+2s)}\leqslant C\frac{\epsilon^{2(n+2s)}}{\delta^{2(n+2s)}},
	\end{split}
\end{equation*}
for some constant $C>0$ independent of $\epsilon$ and $ \xi$.

{F}rom the above two formulas, combining~\eqref{dsvsdvds},~\eqref{sdvsdv}  and~\eqref{bsdfds}, we obtain the desired result in~\eqref{vdsvsd}.
\end{proof}

\section{Decay estimates for the ground state $w$}\label{sec:Decay of the ground state}

In this section, our goal is to give some basic decay properties of the ground state and of its derivatives.
For this, fix $\xi\in\Omega_\epsilon$ 
and recall that~$w_\xi$ is the ground state solution centered at $\xi$. 
Define, for any~$i=1,\cdots, n$,
\begin{equation}\label{dvsdvsd}
	Z_i:=\frac{\partial w_\xi}{\partial x_i} .
\end{equation}

We first establish the following decay estimate for~$Z_i$.

\begin{Lemma}\label{lemma decay of w1}
	There exists a positive constant $C$ such that, for any $i=1,\cdots, n$,
	\begin{equation*}
		|Z_i(x)|\leqslant \frac{C}{|x-\xi|^{n+2s}} \qquad \text{for any } 
		x\in\R^n\setminus B_1(\xi).
	\end{equation*}
\end{Lemma}

\begin{proof}
	We observe that, for a given $\xi\in\Omega_\epsilon$, the function~$w_\xi$ solves
	\begin{equation*}
		-\Delta	w_\xi+(-\Delta)^s w_\xi +w_\xi= w_\xi^p\qquad \text{in }\mathbb{R}^n.
	\end{equation*}
	In view of \cite[Theorem~1.1]{SZZ24}, we know that $w_\xi\in H^2(\mathbb{R}^n)\cap C^2(\mathbb{R}^n)$. In particular, 
	\begin{equation}\label{hih}
		\|w_\xi\|_{L^\infty(\mathbb{R}^n)}+	\|w_\xi\|_{H^2(\mathbb{R}^n)}\leqslant C(\|w_\xi\|_{H^1(\mathbb{R}^n)},n,s,p)\leqslant C(\|w\|_{H^1(\mathbb{R}^n)},n,s,p).
	\end{equation}
	Also, its derivatives $Z_i=\frac{\partial w_\xi}{\partial x_i}\in H^1(\mathbb{R}^n)\cap C^1(\mathbb{R}^n)$ are solutions of the linearized
	equation
	\begin{equation}\label{fdbd}
		-\Delta Z_i+(-\Delta)^s Z_i+Z_i=pw_\xi^{p-1}Z_i\qquad \text{in } \mathbb{R}^n.
	\end{equation}
	Moreover, owing to~\eqref{hih}, by closely following the proof of Lemmata~A.1 and~C.1 in \cite{SZZ24}, we infer that $Z_i\in L^\infty(\mathbb{R}^n)\cap H^2(\mathbb{R}^n)$ and 
	\begin{equation*}
		\|Z_i\|_{H^2(\mathbb{R}^n)}+	\|Z_i\|_{L^\infty(\mathbb{R}^n)}\leqslant C(\|Z_i\|_{H^1(\mathbb{R}^n)},\|w_\xi\|_{L^\infty(\mathbb{R}^n)},n,s,p)\leqslant C(\|w\|_{H^1(\mathbb{R}^n)},n,s,p).
	\end{equation*}
	
	Furthermore, for any $f\in H^{2}(\mathbb{R}^n)$, we have the general Kato-type inequality 
	\begin{equation}\label{vhj}
		(-\Delta)|f|+(-\Delta)^s|f|\leqslant ({\rm sgn }\, f)\left(-\Delta f+(-\Delta)^s f\right)\quad \text{a.e. on } \mathbb{R}^n 
	\end{equation} 
	where 
	\begin{equation*}
		{\rm sgn }\, f=\begin{cases}
			1\qquad &{\mbox{if }}f> 0,\\
			-1 &{\mbox{if }}f<0,\\
			0&{\mbox{if }}f=0.
		\end{cases}
	\end{equation*}
	
Additionally, recalling the decay behaviour of $w_\xi$ in~\eqref{bvsvs}, one finds that there exists~$R>0$ independent of $\xi$ such that 
	\begin{equation*}
		pw_\xi^{p-1}(x)\leqslant \frac{1}{2}\qquad \text{for all }x\in \mathbb{R}^n\setminus B_R(\xi).
	\end{equation*}
Hence, by using~\eqref{vhj}, we see that, a.e. in~$\mathbb{R}^n\setminus B_R(\xi)$,
	\begin{equation}\label{gege}\begin{split}
		&(-\Delta)|Z_i|+(-\Delta)^s|Z_i|+\frac{1}{2}|Z_i|
		\leqslant ({\rm sgn }\, Z_i)\left(-\Delta Z_i+(-\Delta)^s Z_i\right)
		\\&\qquad=-({\rm sgn }\, Z_i)Z_i \left(1-p w_\xi^{p-1}\right)
		\leqslant 0 .
\end{split}	\end{equation}
	
	Now we claim that 
	\begin{equation}\label{bhjkj}
		|Z_i|(x)\leqslant  \frac{C(n,s,R,\|Z_i\|_{L^\infty(\mathbb{R}^n)})}{|x-\xi|^{n+2s}} \quad  \text{on } \mathbb{R}^n\setminus B_1(\xi),
	\end{equation}
	which will give the desired result of Lemma~\ref{lemma decay of w1}.
	
The estimates in~\eqref{bhjkj} follows from a comparison argument,
whose details we now present.  
	We denote by~$\mathcal{K}_{1/2}$ the fundamental function for the operator~$-\Delta +(-\Delta)^s+{1}/{2}$, namely
\begin{equation}\label{poiuytre09876543}
\big(-\Delta +(-\Delta)^s+{1}/{2} \big)\mathcal{K}_{1/2}(x-\xi)=\delta_\xi(x) \qquad {\mbox{for all }}x\in \mathbb{R}^n.\end{equation}
{F}rom \cite[Lemma~4.4]{DSVZ24}, we know that~$\mathcal{K}_{1/2}(x-\xi)\geqslant c>0$ for any~$x\in B_R(\xi)$, for some suitable constant $c=c(n,s,R)>0$.
	 
	 Let us take $C_0:=\|Z_i\|_{L^\infty(\mathbb{R}^n)}c^{-1}>0$ and observe that
	 \begin{equation}\label{2.6NIS}C_0 \mathcal{K}_{1/2}(x-\xi)\geqslant |Z_i|(x) \qquad{\mbox{for all }}x\in B_R(\xi). 
\end{equation}	 
	 We define the function 
	\begin{equation*}
		\omega(x):=C_0 \mathcal{K}_{1/2}(x-\xi) -|Z_i|(x)
	\end{equation*}
	and we notice that, by~\eqref{gege} and~\eqref{poiuytre09876543}, 
	\begin{equation}\label{vsdd}
		(-\Delta)\omega+(-\Delta)^s\omega+\frac{1}{2}\omega\geqslant 0 \qquad \text{a.e. in } \mathbb{R}^n\setminus B_R(\xi).
	\end{equation}
	
	We now prove that $\omega\geqslant 0$ in $ \mathbb{R}^n$. For this, we write~$\omega=\omega_+-\omega_-$ with~$\omega_-:=-\min\left\{0,\omega\right\}$ and~$\omega_+:=\max\left\{0,\omega\right\}$. Since~$\omega\geqslant 0$ in $B_R(\xi)$,
	thanks to~\eqref{2.6NIS}, one has that~$\omega_-=0$ in~$B_R(\xi)$.  Hence, we multiply equation~\eqref{vsdd} by $\omega_-$ and integrate over~$\mathbb{R}^n\setminus B_R(\xi)$, obtaining that
	\begin{equation*}
		\begin{split}
			0&\leqslant -\int_{\mathbb{R}^n \setminus B_R(\xi)}|\nabla \omega_-|^2\, dx+\int_{\mathbb{R}^n \setminus B_R(\xi)}\left(\int_{\mathbb{R}^n}\omega_-(x)\frac{\omega(x)-\omega(y)}{|x-y|^{n+2s}}\, dy\right)\, dx
			-\frac{1}{2}\int_{\mathbb{R}^n \setminus B_R(\xi)}\omega_-^2\, dx\\
			&=-\int_{\mathbb{R}^n  }|\nabla \omega_-|^2\, dx+\frac{1}{2}\int_{\mathbb{R}^n}\int_{\mathbb{R}^n}\frac{(\omega_-(x)-\omega_-(y))(\omega(x)-\omega(y))}{|x-y|^{n+2s}}\, dy\, dx
			-\frac{1}{2}\int_{\mathbb{R}^n }\omega_-^2\, dx\\
			&\leqslant -\frac{1}{2}\int_{\mathbb{R}^n}\int_{\mathbb{R}^n}\frac{(\omega_-(x)-\omega_-(y))^2}{|x-y|^{n+2s}}\, dy\, dx-\frac{1}{2}\int_{\mathbb{R}^n}\int_{\mathbb{R}^n}\frac{\omega_-(x)\omega_+(y)+\omega_+(x)\omega_-(y)}{|x-y|^{n+2s}}\, dy\, dx\\&\leqslant 0.
		\end{split}
	\end{equation*}
	This yields that $\omega_-=0$ a.e. in $\mathbb{R}^n$. 
	
	As a result, taking into account the fact that $\omega$ is continuous away from the origin, we obtain that~$\omega\geqslant 0$ in~$\mathbb{R}^n$, which gives that  
	\begin{equation*}
		|Z_i|(x)\leqslant C_0\mathcal{K}_{1/2}(x-\xi) \qquad \text{in } \mathbb{R}^n. 
	\end{equation*}
	Therefore, the desired result in~\eqref{bhjkj} follows from~\cite[Lemma~4.4]{DSVZ24}.
\end{proof}

One can improve the decay estimate in Lemma~\ref{lemma decay of w1}, according
to the following statement.

\begin{Lemma}\label{lemma decay of w}
	There exists a positive constant $C$ such that, for any $i=1,\cdots, n$,
	\begin{equation*}
		|Z_i(x)|\leqslant \frac{C}{|x-\xi|^{n+2s+1}} \qquad \text{for any } x\in \R^n\setminus B_1(\xi).
	\end{equation*}
\end{Lemma}

\begin{proof}
	We recall that 
\begin{equation*}
	Z_i(x)=\frac{\partial w(x-\xi)}{\partial x_i}=w'\left(|x-\xi|\right)\frac{x_i-\xi_i}{|x-\xi|}
\end{equation*} and,
as a result,	\begin{equation*}
	|\nabla w_\xi(x)|=	\left(\sum_{i=1}^{n}|Z_i(x)|^2\right)^{1/2}=|w'\left(|x-\xi|\right)|.
	\end{equation*}
	
	For any $x\in\mathbb{R}^n$,
we denote by~$x_\xi:=(|x-\xi|,0,\cdots,0)\in\mathbb{R}^n$. Then, 
for every $i\in\left\{2,\cdots,n\right\}$, we obtain that
$\frac{\partial w}{\partial x_i}(x_\xi)=0$  and 
	\begin{equation*}
		|\nabla w_\xi(x)|=\left|\frac{\partial w}{\partial x_1}(x_\xi)\right|.
	\end{equation*}

We also recall that $\mathcal{K}$ is the fundamental function for the operator $-\Delta+(-\Delta)^s+1$ in $\mathbb{R}^n$,	and,
since $\mathcal{K}$  has a singularity at the origin
that we wish to remove, we choose a radial function 
\begin{equation}\label{defkappauno}
{\mbox{$ \mathcal{K}_{1}\in C^\infty(\mathbb{R}^n)$ with $0\leqslant \mathcal{K}_{1}\leqslant \mathcal{K}$ in $\mathbb{R}^n$,
$\mathcal{K}_{1}=\mathcal{K}$ outside $ B_{1/4}$, and $\mathcal{K}_{1}=0$ in  $ B_{1/6}$.}}\end{equation} 
Also, we define 
\begin{equation}\label{defkappadue}\mathcal{K}_{2}:=\mathcal{K}-\mathcal{K}_{1}.\end{equation} {F}rom~\eqref{mathcal{K}}, it follows that 
	\begin{equation}\label{dsfsfe}
		w=\mathcal{K}_{1}\ast w^p+\mathcal{K}_{2}\ast w^p.
	\end{equation}
 We differentiate~\eqref{dsfsfe} with respect to~$x_1$ and obtain that 
\begin{equation*}
	\frac{\partial w}{\partial x_1}=\frac{\partial \mathcal{K}_{1}}{\partial x_1}\ast w^p+\mathcal{K}_{2}\ast \left(pw^{p-1}\frac{\partial w}{\partial x_1}\right).
\end{equation*}

	Note that, in virtue of the properties of~$\mathcal{K}_1$
	and~\eqref{dfdsfs}, one has that  
\begin{equation*}
	\begin{split}
		&\left|\left(\frac{\partial \mathcal{K}_{1}}{\partial x_1}\ast w^p\right)(x_\xi)\right|\\&= \left|\int_{\R^n\setminus B_{1/8} }  \frac{\partial \mathcal{K}_{1}}{\partial x_1}(y)\, w^p(x_\xi-y)\, dy\right|\\
		&\leqslant \left|\int_{\{1/8\leqslant|y|\leqslant |x-\xi|/2\}}  \frac{\partial \mathcal{K}_{1}}{\partial x_1}(y)\, w^p(x_\xi-y)\, dy\right|
		+\left|\int_{\{|y|\geqslant |x-\xi|/2\}}  \frac{\partial \mathcal{K}_{1}}{\partial x_1}(y)\, w^p(x_\xi-y)\, dy\right|\\
			&\leqslant \left|\int_{\{1/8\leqslant|y|\leqslant |x-\xi|/2\}}  
			\frac{\partial \mathcal{K}_{1}}{\partial x_1}(y)\, w^p(x_\xi-y)\, dy\right|+\frac{c}{|x-\xi|^{n+2s+1}}\int_{\mathbb{R}^n}   w^p(y)\, dy.
	\end{split}
\end{equation*}

Let now  ${y}:=(y_1, y_2,\cdots, y_n)$ and $\bar{y}:=(-y_1, y_2,\cdots, y_n)$. Since $\mathcal{K}_1$ is radial, we have that
\begin{equation*}
	\begin{split}
	&\left|\int_{\{1/8\leqslant|y|\leqslant |x-\xi|/2\}}  
	\frac{\partial \mathcal{K}_{1}}{\partial x_1}(y)\, w^p(x_\xi-y)\, dy\right|= \left|\int_{\{1/8\leqslant|y|\leqslant |x-\xi|/2\}} \mathcal{K}_{1}'(|y|) \frac{y_1 }{|y|}\, w^p(x_\xi-y)\, dy\right|\\
	& =\left|\int_{\{1/8\leqslant|y|\leqslant |x-\xi|/2\}\cap \{y_1>0\} } \mathcal{K}_{1}'(|y|) \frac{y_1 }{|y|}\, \left(w^p(x_\xi-y)-w^p(x_\xi-\bar{y})\right)\, dy\right|\\
	&=\left|\int_{\{1/8\leqslant|y|\leqslant |x-\xi|/2\}\cap \{y_1>0 \}}
	\left[\int_{0}^{1} \mathcal{K}_{1}'(|y|) \frac{y_1 }{|y|}\, (pw^{p-1}\nabla w)\left(t(x_\xi-y)+(1-t)(x_\xi-\bar{y})\right)\cdot (y-\bar{y})\, dt\right]\, dy\right|
	.\end{split}
\end{equation*}
We observe that
\begin{eqnarray*}&& |t(x_\xi-y)+(1-t)(x_\xi-\bar{y})
\ge |x-\xi|-| ty+(1-t)\bar{y}|\ge 
|x-\xi|-| y|\geqslant |x-\xi|/2>1/2.
\end{eqnarray*}
Thus, owing to the decay of~$w$ in~\eqref{bvsvs} and the decay of~$\nabla w$ in Lemma~\ref{lemma decay of w1}, 
we see that 
$$\left|\int_{\{1/8\leqslant|y|\leqslant |x-\xi|/2\}}  
	\frac{\partial \mathcal{K}_{1}}{\partial x_1}(y)\, w^p(x_\xi-y)\, dy\right|
 \leqslant  \frac{c}{|x-\xi|^{p(n+2s)}} \left|\int_{\R^n\setminus B_{1/8}} \mathcal{K}_{1}'(|y|)|y| \, dy\right|\leqslant  \frac{c}{|x-\xi|^{p(n+2s)}},$$
for some constant $c>0$ independent of $\xi$ and $x$. Here, the fact that~$\mathcal{K}_1=\mathcal{K}$ outside~$B_{1/4}$
and the decay of~$\nabla\mathcal{K}$ in~\eqref{dfdsfs} have been employed.

Moreover, we observe that, for any~$y\in B_{1/4} $,
$$|x_\xi-y|\geqslant |x-\xi|-|y|\ge |x-\xi|-\frac14\ge \frac{|x-\xi|}2>\frac12.$$
Therefore, using again the  decay of~$w$ in~\eqref{bvsvs} and the decay of~$\nabla w$ in Lemma~\ref{dsdcds1}, we find that 
\begin{equation*}
	\begin{split}
		\left|\left(\mathcal{K}_{2}\ast pw^{p-1}\frac{\partial w}{\partial x_1}\right)(x_\xi)\right|&=\left|\int_{B_{1/4}}\mathcal{K}_2(y)\left(pw^{p-1}\frac{\partial w}{\partial x_1}\right)(x_\xi-y)\, dy\right|\\
		&\leqslant \frac{c}{|x-\xi|^{p(n+2s)}}\int_{\mathbb{R}^n}\mathcal{K}(y)\, dy.
	\end{split}
\end{equation*}

Gathering the above facts, we conclude that, for every~$x\in \R^n\setminus B_1(\xi)$,
\begin{equation}\label{9876543hdsulkeds}
	|\nabla w_\xi(x)|=\left|\frac{\partial w}{\partial x_1}(x_\xi)\right|\leqslant C\left(\frac{1}{|x-\xi|^{p(n+2s)}}+\frac{1}{|x-\xi|^{n+2s+1}}\right).
\end{equation}
If $p(n+2s)\ge n+2s+1$, this gives the desired estimate and completes the proof of Lemma~\ref{lemma decay of w}. 

If instead~$p(n+2s)<n+2s+1$, we deduce from~\eqref{9876543hdsulkeds} that
\begin{equation*}
	|\nabla w_\xi(x)|\leqslant \frac{C}{|x-\xi|^{p(n+2s)}}.
\end{equation*}
Using this improved estimate and repeating the computations above~$m$ times, we find that
\begin{equation*}
	|\nabla w_\xi(x)|\leqslant C\left(\frac{1}{|x-\xi|^{(p+m(p-1))(n+2s)}}+\frac{1}{|x-\xi|^{n+2s+1}}\right).
\end{equation*}
Since there exists $m\in\mathbb{N}$ such that $n+2s+1<(p+m(p-1))(n+2s)$, this entails the desired estimate in Lemma~\ref{lemma decay of w}.
\end{proof}

We now provide estimates for the derivatives of~$Z_1$.

\begin{Lemma}\label{lemma gradient decay}
	There exists a positive constant $C$ such that, for any $i=1,\cdots, n$,
	\begin{equation*}
		|\nabla Z_i(x)|\leqslant \frac{C}{|x-\xi|^{n+2s}} \qquad \text{for any } x\in\R^n\setminus B_1(\xi) .
	\end{equation*}
\end{Lemma}

\begin{proof}
	Let $\mathcal{K}_{1} $ and $\mathcal{K}_{2}$ be as defined in~\eqref{defkappauno} and~\eqref{defkappadue}. Then,
	\begin{equation}\label{dsfsfes}
		w_\xi=\mathcal{K}_{1}\ast w_\xi^p+\mathcal{K}_{2}\ast w_\xi^p.
	\end{equation}
	
	We will show by an inductive argument that, for any $k\in\mathbb{N}$, 
	\begin{equation}\label{sdfsdc}
		\left|	\frac{\partial^2 w_\xi}{\partial x_i \partial x_j}(x)\right|\leqslant C_k(1+|x-\xi|)^{-\nu(k)}\qquad {\mbox{for all }}x\in\mathbb{R}^n
	\end{equation}
	where $C_k>0$ and $$\nu(k):=\min\big\{(n+2s+2),p(n+2s),k(p-1)(n+2s)\big\}.$$
	Indeed, Lemma~\ref{lemma gradient decay} follows from~\eqref{sdfsdc} by taking the smallest $k$ for which $k(p-1)>p$.
	
	We first notice that~\eqref{sdfsdc} holds true
	when $k=0$, thanks to the fact that $\|w_\xi\|_{C^2(\mathbb{R}^n)}\leqslant C\left(\|w_\xi\|_{L^\infty(\mathbb{R}^n)},n,s,p\right) $ (see for instance~\cite[Theorem~1.4]{DSVZ24}).
	
	Now we suppose that~\eqref{sdfsdc} holds true for  some $k$, and we prove it for~$k+1$.
	For this,  we differentiate~\eqref{dsfsfes} and obtain that, for any~$i$, $j\in\left\{1,\cdots, n\right\}$,
		\begin{equation}\label{hdsjbvd}
		\partial^2_{ij}w_\xi:=\frac{\partial^2 w_\xi}{\partial x_i \partial x_j}=\frac{\partial^2 \mathcal{K}_{1}}{\partial x_i\partial x_j}\ast w_\xi^p+\mathcal{K}_{2}\ast \left(p(p-1)w_\xi^{p-2}\frac{\partial w_\xi}{\partial x_i}\frac{\partial w_\xi}{\partial x_j}+pw_\xi^{p-1}\frac{\partial^2 w_\xi}{\partial x_i\partial x_j}\right).
	\end{equation}
	
	Note that, since $\mathcal{K}_{1}\in C^\infty(\mathbb{R}^n)$ and $\mathcal{K}_{1}=\mathcal{K} $ outside $B_{1/4}$,
	in virtue of~\eqref{dfdsfs} one has  that  
	\begin{equation*}
		\left|\frac{\partial^2 \mathcal{K}_{1}}{\partial x_i\partial x_j}(x)\right|\leqslant \frac{C}{(1+|x|)^{n+2s+2}} .
	\end{equation*}
	Hence, using~\eqref{bvsvs} and~\cite[Lemma~5.1]{MR3393677}, we see that, for every $x\in\mathbb{R}^n$, 
	\begin{equation}\label{fwev}
		\left|\frac{\partial^2 \mathcal{K}_{1}}{\partial x_i\partial x_j}\ast w_\xi^p(x)\right|
		\leqslant \frac{C}{(1+|x-\xi|)^{\nu_1}},
	\end{equation}
	up to renaming $C>0$, where $\nu_1:=\min\left\{n+2s+2,p(n+2s)\right\}$.
	
	Moreover, from Lemma~\ref{lemma decay of w} and~\eqref{bvsvs}, \begin{equation}\label{dsvdv}
	{w_\xi(x)^{-1}}|{\partial_i w_\xi}(x)|\leqslant c \qquad {\mbox{for all }}x\in\mathbb{R}^n
	\end{equation}
	for some constant~$c>0$. 
	
	Therefore, for every $x\in\mathbb{R}^n$,
	\begin{equation*}
		\begin{split}
			&\left|p(p-1)\left(w_\xi^{p-2}\frac{\partial w_\xi}{\partial x_i}\frac{\partial w_\xi}{\partial x_j}\right)(x)+p\left(w_\xi^{p-1}\frac{\partial^2 w_\xi}{\partial x_i\partial x_j}\right)(x)\right|\\
			&\qquad\leqslant C\left(\frac{1}{(1+|x-\xi|)^{p(n+2s)}}+ \frac{1}{(1+|x-\xi|)^{(p-1)(n+2s)+\nu(k)}}\right) 
			\leqslant  \frac{C}{(1+|x-\xi|)^{\nu(k+1)}},
		\end{split}
	\end{equation*}
	up to renaming $C>0$. 
	
	Furthermore, we notice that if $|y|<1/4$ and $|x-\xi|<1$, then $1+|x-\xi|<2(1+|x-\xi-y|)$. If instead $|y|<1/4$ and $|x-\xi|\geqslant 1$, then $|x-\xi-y|\geqslant |x-\xi|/2$.
	{F}rom these facts, since $\mathcal{K}_{2}$ vanishes outside $B_{1/4}$, we have that
	\begin{equation*}
		\begin{split}
			&\left|\mathcal{K}_{2}\ast \left(p(p-1)w_\xi^{p-2}\frac{\partial w_\xi}{\partial x_i}\frac{\partial w_\xi}{\partial x_j}+pw_\xi^{p-1}\frac{\partial^2 w_\xi}{\partial x_i\partial x_j}\right)(x)\right|\\
			&\qquad\leqslant C\int_{B_{1/4}} \frac{\mathcal{K}_{2}(y)}{(1+|x-\xi-y|)^{\nu(k+1)}}\, dy\leqslant \frac{C}{(1+|x-\xi|)^{\nu(k+1)}}\int_{\mathbb{R}^n} \mathcal{K}(y)\,dy.
		\end{split}
	\end{equation*}
	This and~\eqref{fwev} establish~\eqref{sdfsdc} for $k+1$, thus completing the inductive argument.
\end{proof}

\begin{Lemma}\label{lemma second gradient decay}
	There exists a positive constant $C$ such that, for any $i=1,\cdots, n$,
	\begin{equation*}
		|\nabla^2 Z_i|\leqslant \frac{C}{|x-\xi|^{n+2s}} \qquad{\mbox{for all }}x\in\R^n\setminus B_1(\xi).
	\end{equation*}
\end{Lemma}

\begin{proof}
Since $Z_i\in H^1(\mathbb{R}^n)\cap L^\infty(\mathbb{R}^n)$ solves problem~\eqref{fdbd}, recalling~\eqref{dsvdv} and taking into account the fact that $w_\xi\in C^2(\mathbb{R}^n)$ (see e.g.\cite[Theorem~1.4]{DSVZ24}), one deduces that
\[ \|(Z_i-pw_\xi^{p-1}Z_i)\|_{C^1(\mathbb{R}^n)} \leqslant C\left(\|w_\xi\|_{C^2(\mathbb{R}^n)},n,s,p\right)\leqslant C\left(\|w_\xi\|_{L^\infty(\mathbb{R}^n)},n,s,p\right). \]
Hence, by closely following the proof of~\cite[Theorem~1.4]{DSVZ24}, we infer that $Z_i\in C^2(\mathbb{R}^n)$ and  
\begin{equation}\label{fsfds}
	\|Z_i\|_{C^2(\mathbb{R}^n)} \leqslant C\left(\|Z_i\|_{L^\infty(\mathbb{R}^n)},\|w_\xi\|_{C^2(\mathbb{R}^n)}, n,s,p\right)
	\leqslant C\left(\|w_\xi\|_{L^\infty(\mathbb{R}^n)}, n,s,p\right).
\end{equation}

We now claim that, for any $k\in\mathbb{N}$, $i,j,l \in\left\{1,\cdots, n\right\}$
and~$x\in\mathbb{R}^n$,
\begin{equation}\label{fdgfdf}
	\left|	\frac{\partial^3 w_\xi}{\partial x_i \partial x_j \partial x_l}(x) \right|\leqslant C_k(1+|x-\xi|)^{-\nu(k)},
\end{equation}
where $C_k>0$ and $$\nu(k):=\min\big\{(n+2s+2),p(n+2s),k(p-1)(n+2s)\big\}.$$

Indeed, \eqref{fsfds} states that~\eqref{fdgfdf} holds when $k=0$.
 Now we suppose that~\eqref{fdgfdf} holds true for some~$k$, and we prove it for~$k+1$.
 For this,  we differentiate~\eqref{hdsjbvd}, obtaining that, for any~$i$, $j$, $l \in\left\{1,\cdots, n\right\}$,
	\begin{equation*}
	\begin{split}
	&\partial^3_{ijl}w_\xi:=\frac{\partial^3 w_\xi}{\partial x_i \partial x_j \partial x_l}=\frac{\partial^2 \mathcal{K}_{1}}{\partial x_i\partial x_j}\ast \left(pw_\xi^{p-1}Z_l\right)\\
	& \qquad+\mathcal{K}_{2}\ast \left(p(p-1)(p-2)w_\xi^{p-3}Z_iZ_jZ_l+p(p-1)w_\xi^{p-2}\left(\partial_l(Z_iZ_j)+Z_l\partial_j(Z_i)\right)+pw_\xi^{p-1}\partial^2_{jl}Z_i\right).
	\end{split}
\end{equation*}
Hence,  by combining Lemmata~\ref{lemma decay of w} and~\ref{lemma gradient decay}, in virtue of~\eqref{dsvdv}, one infers that
\begin{equation*}
	|\partial^3_{ijk}w_\xi|\leqslant \frac{C}{(1+|x-\xi|)^{\nu(k+1)}},
\end{equation*}
for some constant $C>0$, which establishes the desired result~\eqref{fdgfdf}. 

Thus, Lemma~\ref{lemma second gradient decay} follows from~\eqref{fdgfdf} by taking the smallest $k$ for which $k(p-1)>p$.
\end{proof}

In view of Lemmata~\ref{lemma gradient decay} and~\ref{lemma second gradient decay}, using a similar argument to the one in the proof of Lemma~\ref{lemma decay of w}, it is immediate to check the following result.

\begin{Lemma}
	There exists a positive constant $C$ such that, for any $i=1,\cdots,n$,
		\begin{equation*}
		|\nabla Z_i|\leqslant \frac{C}{|x-\xi|^{n+2s+2}} \quad \text{and} \quad |\nabla^2 Z_i|\leqslant \frac{C}{|x-\xi|^{n+2s+3}}
		\qquad{\mbox{for all }}x\in\R^n\setminus B_1(\xi).
	\end{equation*}
\end{Lemma}

Since the ground state $w$ is radial, in virtue of Lemma~\ref{lemma decay of w} and \cite[Lemma~5.5]{MR3393677}, we have:  

\begin{Lemma}\label{lemma Zi}
	Let
	$$\alpha:=\int_{\mathbb{R}^n} Z_1^2\, dx.$$ Then:
	\begin{itemize}
		\item[(i)] For any  $i=1,\cdots,n$,
		$$	\int_{\mathbb{R}^n} Z_i^2\, dx=\alpha  .$$
		\item[(ii)] For any  $i,j=1,\cdots,n$,
		$$\int_{\mathbb{R}^n} Z_i Z_j\, dx=\alpha\delta_{ij} .$$
		\item[(iii)] For any  $i,j=1,\cdots,n$,
		$$\int_{\Omega_\epsilon} Z_i Z_j\, dx=\alpha\delta_{ij}+O(\epsilon^{n+4s}).$$
		\item [(iv)] If $\tau_0\in L^\infty([0,+\infty))$, $\tau(x):=\tau_0(|x-\xi|)$ for any $x\in\mathbb{R}^n$ and $\tilde{Z_i}:=\tau Z_i$, we have that $$\int_{\mathbb{R}^n} \tilde{Z}_i Z_j\, dx=\tilde{\alpha}\delta_{ij} ,$$ with
		$$\tilde{\alpha}:=\int_{\mathbb{R}^n} \tilde{Z}_1Z_1\, dx ,$$
		and~$\tilde{\alpha}$ is independent of $\xi$.
	\end{itemize}
\end{Lemma}

\section{Some regularity estimates}\label{sec:Some regularity estimates}

We devote this section to establish several  uniform estimates on the solutions of our general linear equations. For this, 
let us first introduce some notations. Let $\xi\in\Omega_\epsilon$ with
\begin{equation}\label{fdgd}
	\operatorname{dist}(\xi,\partial \Omega_\epsilon)\geqslant \frac{c}{\epsilon}\qquad \text{for some } c\in(0,1).
\end{equation}
We also define a $L^\infty$ weighted norm as follows: 
\begin{equation}\label{vdsvdsv}
	\|\psi\|_{\infty,\xi}:=\|\rho_\xi\psi\|_{L^\infty(\mathbb{R}^n)},
\end{equation}
where $\rho_\xi(x):=(1+|x-\xi|)^\mu$  for any $x\in\mathbb{R}^n$, with 
\begin{equation}\label{fiuewfjke768578599}
\frac{n}{2}<\mu<\min\left\{p(n+2s)-\frac{n}2-2s,n+2s\right\}.\end{equation}
We remark that the upper bound on $\mu$ is indeed associated with~$n$, $s$ and~$p$ (see Theorem~\ref{th expansion} below).

\begin{Lemma}\label{lemma regularity}
Let $g\in L^2(\mathbb{R}^n)\cap L^\infty(\mathbb{R}^n)$ and  $\psi\in H^1(\mathbb{R}^n)$ be a solution of
\begin{equation}\label{vsdjvb}
	\begin{cases}
		-\Delta \psi+ (-\Delta )^s\psi+\psi =g\qquad &\text{in } \Omega_\epsilon,\\
		\psi=0 &\text{in } \mathbb{R}^n\setminus\Omega_\epsilon.
	\end{cases}
\end{equation}

Then, for every $\alpha\in(0,\min\left\{1,2-2s\right\})$
there exists a positive constant $C$, depending only
on $n$, $s$, $\alpha$ and~$\Omega$, such that 
\begin{equation*}
	\|\psi\|_{C^{1,\alpha}(\overline{\Omega_\epsilon})}\leqslant C\left(\|g\|_{L^\infty(\mathbb{R}^n)}+\|g\|_{L^2(\mathbb{R}^n)}\right).
\end{equation*}
\end{Lemma}

\begin{proof}
  First,  employing a De Giorgi-type
  argument (see e.g. the proof of \cite[Theorem~8.2]{MR3259559}), we see that
  \begin{equation}\label{vdfsd}
  	\|\psi\|_{L^\infty(\mathbb{R}^n)}\leqslant C(n,s)\left(\|g\|_{L^\infty(\mathbb{R}^n)}+\|\psi\|_{L^2(\mathbb{R}^n)}\right).
  \end{equation}
    The full details of the proof of~\eqref{vdfsd} are provided in Appendix~\ref{sec:vfdvdfv} for the reader's convenience.

 We now show that there exists a constant $C>0$ independent of $\epsilon$ such that,  for every $\alpha\in(0,\min\left\{1,2-2s\right\})$,
\begin{equation}\label{bkm }
	\|\psi\|_{C^{1,\alpha}(\overline{\Omega_\epsilon})}\leqslant C\left(\|g\|_{L^\infty(\mathbb{R}^n)}+\|\psi\|_{L^\infty(\mathbb{R}^n)}\right).
\end{equation}
Indeed,  for a given $\epsilon>0$,  since both $\psi$ and $g$ are bounded,  from the regularity results in \cite{SVWZ23},
we know that $\psi\in C_{\rm loc}^{2,\alpha}(\Omega_\epsilon)\cap C^{1,\alpha}(\overline{\Omega_\epsilon})$ for any $\alpha\in(0,\min\left\{1,2-2s\right\})$. 

Let us denote by
$$\mathcal{O}:=\left\{x\in\Omega_\epsilon:\, \operatorname{dist}(x,\partial\Omega_\epsilon)>\frac{1}{2}\right\}.$$
In view of \cite[Theorem~1.4]{SVWZ23}, we have that, for every $\alpha\in(0,1)$,
\begin{equation}\label{bkjbj}
		\|\psi\|_{C^{1,\alpha}(\overline{\mathcal{O}})}\leqslant \Lambda\left(\|g\|_{L^\infty(\mathbb{R}^n)}+\|\psi\|_{L^\infty(\mathbb{R}^n)}\right),
\end{equation} where the constant  $\Lambda>0$ depends only on $n$, $s$ and~$\alpha$.

It remains to prove that $\psi$ has uniform $C^{1,\alpha}$-norm near the boundary of $\Omega$. For this, we notice that, since $\Omega$ is a bounded domain with smooth boundary, there exist positive constants~$\rho$ and~$M$ such that, for every $x_0\in\partial\Omega$, we find
a $C^2$ function $\gamma_{x_0}: \mathbb{R}^{n-1}\rightarrow\mathbb{R}$ such that, up to relabeling and reorienting the coordinate axes if necessary, we have that
\begin{eqnarray*}
&&	\Omega\cap B_{\rho}(x_0)=\left\{y\in B_{\rho}(x_0):\, y_n>\gamma_{x_0}(y_1,\cdots,y_{n-1}) \right\},
\\&&	\partial\Omega\cap B_{\rho}(x_0)=\left\{y\in B_{\rho}(x_0):\, y_n=\gamma_{x_0}(y_1,\cdots,y_{n-1}) \right\} ,
\\{\mbox{and }}&&
	\|\nabla \gamma_{x_0}\|_{L^\infty(\mathbb{R}^{n-1})}+	\|D^2\gamma_{x_0}\|_{L^\infty(\mathbb{R}^{n-1})}\leqslant M.
\end{eqnarray*}

We now assume that  $\epsilon$ sufficiently small such that $\epsilon<\min\left\{1,\rho/4\right\}$. We observe that  $\epsilon P\in\partial\Omega$ for every~$P\in\partial\Omega_\epsilon$, and  that $\epsilon x\in B_\rho(\epsilon P)$ for every~$x\in B_4(p)$. Thus, 
we set~$\gamma_P^\epsilon(y_1,
\cdots,y_{n-1}) =\epsilon^{-1}\gamma_{\epsilon P}(\epsilon y_1,
\cdots,\epsilon y_{n-1})$ and we have that~$\gamma_P^\epsilon$ is a~$C^2$ function such that 
\begin{eqnarray*}&&
	\Omega_\epsilon\cap B_{4}(P)=\left\{y\in B_{4}(P):\, y_n>\gamma_{P}^\epsilon(y_1,\cdots,y_{n-1}) \right\},
\\&&
\partial\Omega_\epsilon\cap B_{4}(P)=\left\{y\in B_{4}(P):\, y_n=\gamma_{P}^\epsilon(y_1,\cdots,y_{n-1}) \right\},\end{eqnarray*}
and
\begin{equation}\nabla_{y'} \gamma_{P}^\epsilon(y')= \nabla_{\epsilon y'} \gamma_{\epsilon P} (\epsilon y') \quad \text{and} \quad	D_{y'}^2\gamma_{P}^\epsilon(y')=\epsilon D^2_{\epsilon y'}\gamma_{\epsilon P}(\epsilon y') \label{fbdfdfbd}
\end{equation}
where $y'=(y_1,\cdots,y_{n-1})$.

Therefore,   regularizing the graph of $\partial(B_3(P)\cap\Omega_\epsilon) $  near the intersection of   $\partial\Omega_\epsilon$ and $ B_3(P)$, we can find  a $C^{2}$ domain $\mho_P$ such that  $(B_{2}(P)\cap\Omega_\epsilon )\subset \mho_P\subset (B_{4}(P)\cap\Omega_\epsilon)$ (see Figure 1), and  for every~$\epsilon<\min\left\{1,\rho/{4}\right\}$ one has that
the~$C^{1,1}$-norm of $\partial\mho_p$ is bounded uniformly, only in dependence of~$n$ and~$\Omega$.

\begin{figure}
	\centering
\tikzset{every picture/.style={line width=0.75pt}} 
\begin{tikzpicture}[x=0.4pt,y=0.4pt,yscale=-1,xscale=1]

	
	\draw   (65,134.5) .. controls (65,71.26) and (116.26,20) .. (179.5,20) .. controls (242.74,20) and (294,71.26) .. (294,134.5) .. controls (294,197.74) and (242.74,249) .. (179.5,249) .. controls (116.26,249) and (65,197.74) .. (65,134.5) -- cycle ;
	\draw   (95.5,134.5) .. controls (95.5,88.11) and (133.11,50.5) .. (179.5,50.5) .. controls (225.89,50.5) and (263.5,88.11) .. (263.5,134.5) .. controls (263.5,180.89) and (225.89,218.5) .. (179.5,218.5) .. controls (133.11,218.5) and (95.5,180.89) .. (95.5,134.5) -- cycle ;
	\draw   (128.5,134.5) .. controls (128.5,106.33) and (151.33,83.5) .. (179.5,83.5) .. controls (207.67,83.5) and (230.5,106.33) .. (230.5,134.5) .. controls (230.5,162.67) and (207.67,185.5) .. (179.5,185.5) .. controls (151.33,185.5) and (128.5,162.67) .. (128.5,134.5) -- cycle ;
	\draw  [fill={rgb, 255:red, 0; green, 0; blue, 0 }  ,fill opacity=1 ] (177,134.5) .. controls (177,133.12) and (178.12,132) .. (179.5,132) .. controls (180.88,132) and (182,133.12) .. (182,134.5) .. controls (182,135.88) and (180.88,137) .. (179.5,137) .. controls (178.12,137) and (177,135.88) .. (177,134.5) -- cycle ;
	\draw    (169,83.81) -- (179.5,134.5) ;
	\draw    (109,178.81) -- (179.5,134.5) ;
	\draw    (286,91.81) -- (179.5,134.5) ;
	\draw    (21,139.81) .. controls (52,55.81) and (344,227.81) .. (382,121.81) ;
	\draw [color={rgb, 255:red, 208; green, 2; blue, 27 }  ,draw opacity=1 ][line width=1.5]    (118,190.81) .. controls (94,164.81) and (92.88,132.98) .. (102.44,124.4) .. controls (112,115.81) and (135,125.19) .. (179.5,134.5) .. controls (224,143.81) and (251,149.81) .. (255,158.81) .. controls (259,167.81) and (246,186.81) .. (234,198.81) .. controls (222,210.81) and (194,221.81) .. (171,217.81) .. controls (148,213.81) and (134,207.81) .. (118,190.81) -- cycle ;
	
	\draw (181.5,135.4) node [anchor=north west][inner sep=0.5pt]    {$p$};
	\draw (384,113.4) node [anchor=north west][inner sep=0.5pt]  [font=\small]  {$\partial \Omega _{\varepsilon }$};
	\draw (351,185.4) node [anchor=north west][inner sep=0.5pt]  [font=\small]  {$\Omega _{\varepsilon }$};
	\draw (236,202.21) node [anchor=north west][inner sep=0.5pt]  [font=\small,color={rgb, 255:red, 208; green, 2; blue, 27 }  ,opacity=1 ]  {$\mho_{p}$};
	\draw (270,40.4) node [anchor=north west][inner sep=0.5pt]  [font=\small]  {\tiny $B_{4}( p)$};
	\draw (80,182.4) node [anchor=north west][inner sep=0.5pt]  [font=\small]  {\tiny $B_{3}( p)$};
	\draw (157,61.4) node [anchor=north west][inner sep=0.5pt]  [font=\small]  {\tiny $B_{2}( p)$};
	\end{tikzpicture}
\caption{The set $\mho_p$ in the proof of Lemma \ref{lemma regularity}}
\end{figure}

 We now consider  a cutoff function $\phi\in C_0^\infty(\mathbb{R}^n)$
 satisfying
 \begin{equation*}
 	\phi\equiv1 \quad \text{on } B_{1}, \qquad \text{supp}(\phi)\subset B_2, \qquad 
 	0\leqslant \phi\leqslant 1 \quad  \text{on } \mathbb{R}^n.
 \end{equation*}
Let us denote by~$\phi_p(x):=\phi(x-p)$ and $\varphi:=\psi\phi_p$. Thus, one has that 
\begin{equation}\label{fbdg}
	\varphi\equiv\psi \quad \text{on } B_{1}(p)\cap\Omega_\epsilon \quad \text{and} \quad \varphi\equiv 0\quad   \text{on } \mathbb{R}^n\setminus (B_{2}(p)\cap\Omega_\epsilon).
\end{equation}

Since $\psi\in C_{\rm loc}^{2,\alpha}(\Omega_\epsilon)\cap C^{1,\alpha}(\overline{\Omega_\epsilon})$ for any $\alpha\in(0,\min\left\{1,2-2s\right\})$, one has that $\varphi\in C_{\rm loc}^{2,\alpha}(\mho_p)\cap C^{1,\alpha}(\overline{\mho_p})$. Thus, by combining~\eqref{vsdjvb} with~\eqref{fbdg}, one finds that 
\begin{equation*}
	\begin{cases}
		-\Delta \varphi+ (-\Delta )^s\varphi =\Psi
		\quad &\text{in } \mho_p,\\
		\varphi=0 &\text{in } \mathbb{R}^n\setminus\mho_p,
		\end{cases}
\end{equation*}
where $$\Psi:=\phi_p(g-\psi)-2\nabla \psi\cdot\nabla \phi_p+\psi(-\Delta \phi_p+ (-\Delta )^s\phi_p)+\int_{\mathbb{R}^n}\frac{\psi(y)(\phi_p(x)-\phi_p(y))}{|x-y|^{n+2s}}\, dy.$$

In view of \cite[Theorem 4.1]{SVWZ23}, we know that 
\begin{equation}\label{fbdfb}
\begin{split}
	\|\varphi\|_{W^{2,q}(\mho_p)}&\leqslant C(n,s,q,\mho_p)\left(\|\varphi\|_{L^{q}(\mho_p)}+\|\Psi\|_{L^{q}(\mho_p)}\right)\\
	&\leqslant C(n,s,q,\mho_p)\left(\|g\|_{L^{q}(\mho_p)}+\|\psi\|_{L^{q}(\mho_p)}+\|\nabla\psi\|_{L^\infty(\mho_p)}+\|\psi\|_{L^\infty(\mathbb{R}^n)}\right)\\
	&\leqslant C(n,s,q,\mho_p)\left(\|g\|_{L^{\infty}(\mathbb{R}^n)}+\|\nabla\psi\|_{L^\infty(\Omega_\epsilon)}+\|\psi\|_{L^\infty(\mathbb{R}^n)}\right),
\end{split}
\end{equation}
up to renaming $C(n,s,q,\mho_p)$ from line to line. Here, $1<q<+\infty$ if~$s\in(0,\frac{1}{2}]$ and~$n<q<\frac{n}{2s-1} $ if~$s\in(\frac{1}{2},1)$.

Additionally, we recall that the upper bound on the volume of $\mho_p$ and that on the  $C^{1,1}$-norm of $\partial\mho_p$ only depend on $n$ and $\Omega$.  Thus, by~\eqref{fbdfb},  
\begin{equation}\label{tendf}
		\|\varphi\|_{W^{2,q}(\mho_p)}\leqslant C\left(\|g\|_{L^{\infty}(\mathbb{R}^n)}+\|\psi\|_{C^1(\overline{\Omega_\epsilon})}\right),
\end{equation}
for some constant $C$ depending only
on $n$, $s$, $q$ and~$\Omega$.

Moreover, recalling the definition of $\varphi$ and the fact that the data of $\mho_p$ depends on $\Omega$,  by combining Sobolev inequalities with~\eqref{tendf},  for every  $p\in\partial\Omega_\epsilon$ and $\alpha\in(0,\min\left\{1,2-2s\right\})$, taking $q=\frac{n}{1-\alpha}$, one has that 
\begin{equation}\label{bfgbsf}\begin{split}&
	\|\psi\|_{C^{1,\alpha}(\overline{B_1(p)\cap\Omega_\epsilon})}\leqslant \|\varphi\|_{C^{1,\alpha}(\overline{\mho_p})}\leqslant C(n,\alpha,\mho_p) \|\varphi\|_{W^{2,q}(\mho_p)}\\&\qquad\qquad
	\leqslant C(n,s,\alpha,\Omega)\left(\|g\|_{L^{\infty}(\mathbb{R}^n)}+\|\psi\|_{C^1(\overline{\Omega_\epsilon})}\right).\end{split}
\end{equation}

We next claim that, for every $\alpha\in(0,\min\left\{1,2-2s\right\})$, \begin{equation}\label{cvsfvs}
	\|\psi\|_{C^{1,\alpha}(\overline{\Omega_\epsilon})}\leqslant \bar{C}\left(\|g\|_{L^{\infty}(\mathbb{R}^n)}+\|\psi\|_{C^1(\overline{\Omega_\epsilon})}\right),
\end{equation}
where $\bar{C}$ is a positive constant depending only on $n$, $s$, $\alpha$ and~$\Omega$.

Indeed,  for every $x\in \Omega_\epsilon\setminus\mathcal{O}$  there exists $\tilde{x}
\in\partial\Omega_\epsilon$ such that $|x-\tilde{x}|=$dist$(x,\partial\Omega_\epsilon)$. Then for every $y\in\Omega_\epsilon$ with  $|x-y|<\frac{1}{2},$ from~\eqref{bfgbsf}, one has that $x$, $y\in B_1(\tilde{x})\cap\Omega_\epsilon$ and that
\begin{equation*}
\sup_{\begin{subarray}{l}
		x\in\Omega_\epsilon\setminus\mathcal{O}\\
	|x-y|<\frac{1}{2}
	\end{subarray}
	}\frac{|\nabla \psi(x)-\nabla\psi(y)|}{|x-y|^\alpha}\leqslant \sup\limits_{x\in\Omega_\epsilon\setminus\mathcal{O}}	\|\psi\|_{C^{1,\alpha}(\overline{B_1(\tilde{x})\cap\Omega_\epsilon})}\leqslant C(n,\alpha,s,\Omega)\left(\|g\|_{L^{\infty}(\mathbb{R}^n)}+\|\psi\|_{C^1(\overline{\Omega_\epsilon})}\right).
\end{equation*}
In addition, for every $y\in\Omega_\epsilon$ with  $|x-y|\geqslant\frac{1}{2},$ by combining~\eqref{bkjbj} and~\eqref{bfgbsf}, we see that 
\begin{equation*}
	\begin{split}
\sup_{\begin{subarray}{l}
		x\in\Omega_\epsilon\setminus\mathcal{O}\\
			|x-y|\geqslant \frac{1}{2}
		\end{subarray}
	}\frac{|\nabla \psi(x)-\nabla\psi(y)|}{|x-y|^\alpha}
&\leqslant 	2\left(\|\psi\|_{C^{1,\alpha}(\overline{B_1(\tilde{x})\cap\Omega_\epsilon})}+\|\psi\|_{C^{1,\alpha}(\overline{B_1(\tilde{y})\cap\Omega_\epsilon})}+\|\psi\|_{C^{1,\alpha}(\overline{\mathcal{O}})}\right)\\
&\leqslant C(n,s,\alpha,\Omega)\left(\|g\|_{L^{\infty}(\mathbb{R}^n)}+\|\psi\|_{C^1(\overline{\Omega_\epsilon})}\right).
\end{split}
\end{equation*}
Gathering  these facts, we obtain the desired result in~\eqref{cvsfvs}.

Moreover, for every $\epsilon<\min\left\{1,{\rho}/{4}\right\}$ and $p\in\partial\Omega_\epsilon$, from the proof of  Lemma~6.35 (global interpolation inequalities) in \cite{GTbook} and~\eqref{fbdfdfbd},  we know that, for any $\delta>0$, there exists a constant~$C(\delta, \Omega,\alpha)>0 $ such that
\begin{equation*}
	\|\psi\|_{C^{1}(\overline{B_1(p)\cap\Omega_\epsilon})}\leqslant \delta \|\psi\|_{C^{1,\alpha}(\overline{\Omega_\epsilon})}+C(\delta, \Omega,\alpha)\|\psi\|_{L^\infty(\mathbb{R}^n)} .
\end{equation*}
As a consequence of this and~\eqref{bkjbj}, we infer that 
\begin{equation*}
	\|\psi\|_{C^{1}(\overline{\Omega_\epsilon})}\leqslant \delta \|\psi\|_{C^{1,\alpha}(\overline{\Omega_\epsilon})}+C(\delta, \Omega,\alpha,n,s)\left(\|\psi\|_{L^\infty(\mathbb{R}^n)} + \|g\|_{L^\infty(\mathbb{R}^n)}\right).
\end{equation*}
Recalling~\eqref{cvsfvs} and taking $\delta:={1}/(2\bar{C})$, one has that, for every~$\epsilon<\min\left\{1, {\rho}/{4}\right\}$ and~$\alpha\in(0,\min\left\{1,2-2s\right\})$,
\begin{equation*}
	 \|\psi\|_{C^{1,\alpha}(\overline{\Omega_\epsilon})}\leqslant C( \Omega,\alpha,n,s)\left(\|\psi\|_{L^\infty(\mathbb{R}^n)} + \|g\|_{L^\infty(\mathbb{R}^n)}\right).
\end{equation*}
{F}rom this, we obtain~\eqref{bkm }, as desired.

Finally, we complete the proof of the Lemma~\ref{lemma regularity}.  In virtue of~\eqref{vdfsd} and~\eqref{bkm }, it suffices to show that
\begin{equation}\label{gevsvdf}
	\|\psi\|_{L^2(\mathbb{R}^n)}\leqslant \|g\|_{L^2(\mathbb{R}^n)}.
\end{equation}
Indeed, we multiply~\eqref{vsdjvb} by $\psi$ and we integrate over $\Omega_\epsilon$,  obtaining that
\begin{equation}\label{vdvsd}
	\int_{\Omega_\epsilon}|\nabla\psi|^2+\int_{\Omega_\epsilon}\psi(-\Delta)^s\psi+\int_{\Omega_\epsilon}\psi^2=\int_{\Omega_\epsilon}g\psi.
\end{equation}
Thanks to formula~(1.5) in \cite{MR3211861}, one has that 
\begin{equation*}
	\int_{\Omega_\epsilon}\psi(-\Delta)^s\psi=\int_{\Omega_\epsilon}|(-\Delta)^{\frac{s}{2}}\psi|^2\geqslant 0.
\end{equation*}
Thus, by combining~\eqref{vdvsd} with the
H\"{o}lder inequality, we obtain~\eqref{gevsvdf}. This finishes the proof of  Lemma~\ref{lemma regularity}.
\end{proof}

\begin{Lemma}\label{lemma bdjshfsd}
	Let $g\in L^2(\mathbb{R}^n)$ with $\|g\|_{\infty,\xi}<+\infty$, and let $\psi\in H^1(\mathbb{R}^n)$ be a solution of \begin{equation*}
		\begin{cases}
			-\Delta \psi+ (-\Delta )^s\psi+\psi-pw_\xi^{p-1}\psi=g\qquad &\text{in } \Omega_\epsilon,\\
			\psi=0 &\text{in } \mathbb{R}^n\setminus\Omega_\epsilon.
		\end{cases}
	\end{equation*}
	
	Then, there exist positive constants $ C$ and $\tilde{R}$ such that 
	\begin{equation}\label{hguhi}
		\|\psi\|_{\infty,\xi}\leqslant C\left(\|\psi\|_{L^\infty(B_{\tilde{R}}(\xi))}+\|g\|_{\infty,\xi}\right).
	\end{equation}
\end{Lemma}
\begin{proof}
	{F}rom the decay of $w$ at infinity, there exists $\tilde{R}>0$ such that~$1-pw_\xi^{p-1}(x)\geqslant {1}/{2}$  in~$\mathbb{R}^n\setminus B_{\tilde{R}}(\xi)$. Thus, we define the function
	\begin{equation*}
	W(x):=\begin{dcases} \frac12  &{\mbox{ if }} x\in B_{\tilde{R}}(\xi),\\
	1-pw_\xi^{p-1}(x) &{\mbox{ if }}x\in \mathbb{R}^n\setminus B_{\tilde{R}}(\xi).\end{dcases}
	\end{equation*}
	 Then, we have that
	\begin{equation}\label{chjb}
		-\Delta \psi +(-\Delta)^s\psi+W\psi=G 
	\end{equation} where
		\begin{equation*}
	G:=\begin{dcases} g-\left(\frac12-pw_\xi^{p-1}\right)\psi &{\mbox{ in }}  B_{\tilde{R}}(\xi),\\
	g&{\mbox{ in }} \Omega_\varepsilon\setminus B_{\tilde{R}}(\xi).\end{dcases}
	\end{equation*}
	{F}rom this, we infer that \begin{equation}\label{dsvds}
		\|G\|_{\infty,\xi}\leqslant C\|\psi\|_{L^\infty(B_{\tilde{R}}(\xi))}+\|g\|_{\infty,\xi}
	\end{equation}
	for some constant $C>0$ depending on $\tilde{R}$.
	
	Moreover, let us define $ \phi\in H^1(\mathbb{R}^n)$ by the convolution
	\begin{equation*}
		\phi(x):=\int_{\mathbb{R}^n}\frac{\mathcal{K}_{1/2}(y) }{(1+|x-y|)^{\mu}}\, dy
	\end{equation*} 
	where  $\mathcal{K}_{1/2}$ is the fundamental solution of the operator~$-\Delta+(-\Delta)^s)+1/2$. 
	Hence, we have that~$\phi\geqslant 0$ and
	\begin{equation}\label{dsfds}
		-\Delta\phi +(-\Delta)^s\phi+\frac{1}{2}\phi={(1+|x|)^{-\mu}} \qquad \text{in }\mathbb{R}^n.
	\end{equation}

	We now claim that there exists $C_0>0$ such that
	\begin{equation}\label{dsfdsdfs}
		\sup\limits_{x\in\mathbb{R}^n}(1+|x|)^\mu\phi(x)\leqslant C_0.
	\end{equation}
	 Indeed, we observe that, owing to the decay of $\mathcal{K}_{1/2}$  at infinity (see e.g.\cite[Lemma~4.4]{DSVZ24}), for any $|x|>2$,
	 \begin{equation}\label{lkjhgfd09876543mnbvc}
	 \begin{split}
	 \phi(x)&= \int_{\{|x-y|<|x|/2\}} \frac{\mathcal{K}_{1/2}(y) }{(1+|x-y|)^{\mu}}\, dy+ \int_{\{|x-y|\geqslant |x|/2\}} \frac{\mathcal{K}_{1/2}(y) }{(1+|x-y|)^{\mu}}\, dy\\&=
	 \int_{\{|z|<|x|/2\}} \frac{\mathcal{K}_{1/2}(x+z) }{(1+|z|)^{\mu}}\, dz+ \int_{\{|x-y|\geqslant |x|/2\}} \frac{\mathcal{K}_{1/2}(y) }{(1+|x-y|)^{\mu}}\, dy\\
	 &\le C \int_{\{|z|<|x|/2\}} \frac{dz }{(1+|z|)^{\mu}|x+z|^{n+2s}}\, dz+C (1+|x|)^{-\mu}\int_{\mathbb{R}^n}\mathcal{K}_{1/2}(y)\, dy\\
	 &\le \frac{C}{|x|^{n+2s}} \int_{\{|z|<|x|/2\}} \frac{dz }{(1+|z|)^{\mu} }\, dz+C (1+|x|)^{-\mu}\int_{\mathbb{R}^n}\mathcal{K}_{1/2}(y)\, dy\\
	  &\le \frac{C}{|x|^{n+2s}}\left(|B_1| +
	    \int_{\{1\le|z|<|x|/2\}} \frac{dz }{ |z|^{\mu} }\right)+C (1+|x|)^{-\mu} \\
	 &\leqslant \frac{C}{|x|^{n+2s}}\left(1+|x|^{n-\mu} \right)+C (1+|x|)^{-\mu}\\
	 	&\leqslant C|x|^{-\mu} \left( \frac1{|x|^{n+2s-\mu}} +\frac1{|x|^{2s}}\right)
	+C(1+|x|)^{-\mu}\\
	 	&\leqslant C(1+|x|)^{-\mu},
	 \end{split}
	 \end{equation}
	up to renaming $C>0$. 

This and the fact that~$\|\phi\|_{L^\infty(\mathbb{R}^n)}\leqslant \|\mathcal{K}_{1/2}\|_{L^1(\mathbb{R}^n)}$
entail the desired claim in~\eqref{dsfdsdfs}.

We now define $\phi_\xi(x):=\phi(x-\xi)$ and $\omega:=\|G\|_{\infty,\xi}\phi_\xi\pm \psi$. Thus, from~\eqref{chjb} and~\eqref{dsfds}, we deduce that
\begin{equation*}
		-\Delta\omega +(-\Delta)^s\omega+W\omega=\frac{\|G\|_{\infty,\xi}}{(1+|x-\xi|)^\mu}+\|G\|_{\infty,\xi}\left(W-\frac{1}{2}\right)\phi_\xi\pm G\geqslant 0\quad \text{in } \Omega_\epsilon.
\end{equation*}
Furthermore, we observe that $\omega\geqslant 0$ in $\mathbb{R}^n\setminus \Omega_\epsilon$. As a result of this and the Maximum Principle, we obtain that~$\omega\geqslant 0$ in~$\mathbb{R}^n$. 

Hence, in virtue of~\eqref{dsfdsdfs}, we have that, for any $x\in\mathbb{R}^n$,
\begin{equation*}
	|(1+|x-\xi|)^\mu \psi(x)|\leqslant (1+|x-\xi|)^\mu\|G\|_{\infty,\xi}\phi_\xi(x)\leqslant C_0\|G\|_{\infty,\xi}.
\end{equation*}
This and~\eqref{dsvds} show the desired result in~\eqref{hguhi}.
\end{proof}

In a similar way, one can prove the following estimate:

\begin{Lemma}\label{lemma dgfusdhi}
	Let $g\in L^2(\mathbb{R}^n)$ with $\|g\|_{\infty,\xi}<+\infty$, and let $\psi\in H^1(\mathbb{R}^n)$ be a solution of \begin{equation*}
		\begin{cases}
			-\Delta \psi+ (-\Delta )^s\psi+\psi=g\qquad &\text{in } \Omega_\epsilon,\\
			\psi=0 &\text{in } \mathbb{R}^n\setminus\Omega_\epsilon.
		\end{cases}
	\end{equation*}
	
	Then, there exists a positive constant $ C$ such that 
	\begin{equation*}
		\|\psi\|_{\infty,\xi}\leqslant C\|g\|_{\infty,\xi}.
	\end{equation*}
\end{Lemma}

\section{The Lyapunov-Schmidt reduction}\label{sec:The Lyapunov-Schmidt reduction}

In this section, we focus on the linear theory associated with problem~\eqref{vsdvdscsc}. 
For this, we  introduce the function space 
\begin{equation*}
	\Psi:=\left\{\psi\in H^1(\mathbb{R}^n)\,:\, \psi=0 \text{ in }\mathbb{R}^n\setminus\Omega_\epsilon\text{ and } \int_{\Omega_\epsilon} \psi Z_i\, dx=0 \text{ for any } i=1,\cdots,n  \right\}
\end{equation*}
where $Z_i$ are defined by~\eqref{dvsdvsd}.

Our aim is to find a solution for problem~\eqref{vsdvdscsc} of the form 
\begin{equation}\label{vdsfdsfew}
	u=u_\xi:=\bar{u}_\xi+\psi
\end{equation}
where $\bar{u}_\xi$ is a solution of problem~\eqref{sdvbsd} and $\psi\in\Psi$ is a ``small'' perturbation for~$\epsilon$ sufficiently small. 
Inserting $u$ into problem~\eqref{vsdvdscsc}, since $\bar{u}_\xi$ solves~\eqref{sdvbsd}, we see that~$\psi$ satisfies \begin{equation}\label{vsdfse}
	-\Delta \psi+(-\Delta)^s\psi+\psi-p{w}_\xi^{p-1}\psi=E(\psi)\qquad \text{in }\Omega_\epsilon
\end{equation}
where \begin{equation}\label{sdvds}
	E(\psi):=(\bar{u}_\xi+\psi)^p-w_\xi^p-pw_\xi^{p-1}\psi.
\end{equation}

 We deal with a projected version of this equation, instead of solving problem~\eqref{vsdfse} directly, that is, we consider the problem of finding $\psi\in \Psi$ such that, for certain constants $c_{i}\in\mathbb{R}$,
 \begin{equation}\label{vdsdsvs}
 	-\Delta \psi+(-\Delta)^s\psi+\psi-pw_\xi^{p-1}\psi=E(\psi)+\sum_{i=1}^{n}c_iZ_i \qquad \text{in }\Omega_\epsilon.
 \end{equation}
 
 We will show that problem~\eqref{vdsdsvs} has a unique solution in $\Psi$, which is ``small'' as~$\epsilon$ is sufficiently small, and then we will find a suitable $\xi$ such that $c_i=0$ for every $i\in\left\{1,\cdots,n\right\}$. This  provides us with a solution~$\psi\in\Psi$ for~\eqref{vsdfse}, and thus a solution~$u$ for problem~\eqref{vsdvdscsc}.

\subsection{Linear theory} This section establishes an existence result for the linear problem~\eqref{vdsdsvs}.
In what follows, the notation in~\eqref{vdsvdsv} for the norm~$\|\cdot\|_{\infty,\xi}$ is used.

\begin{Theorem}\label{th unique about g}
	Let $g\in L^2(\mathbb{R}^n)$ with $\|g\|_{\infty,\xi}<+\infty$. If $\epsilon>0$ is sufficiently small, there exist a unique~$\psi\in \Psi$ and numbers~$ c_i\in\mathbb{R}$ for any $i\in\left\{1,\cdots, n\right\}$ such that 
	\begin{equation}\label{sdfsde}
			\begin{cases}
			-\Delta \psi+(-\Delta)^s\psi+\psi-pw_\xi^{p-1}\psi=g+\displaystyle \sum_{i=1}^{n}c_iZ_i \qquad &\text{in }\Omega_\epsilon,\\
			\psi=0 &\text{in }\mathbb{R}^n\setminus\Omega_\epsilon.
			\end{cases}
	\end{equation}
	Moreover, there exists a constant $C>0$ such that 
	\begin{equation}\label{sdfsdv}
		\|\psi\|_{\infty,\xi}\leqslant C\|g\|_{\infty,\xi}.
	\end{equation}
\end{Theorem}

\begin{remark}{\rm
We point out that, by closely following the proof of \cite[Lemma~7.2]{MR3393677},
 the coefficients~$c_i$ in~\eqref{sdfsde} can actually be uniquely determined, in terms of~$\psi$ and~$g$.}
\end{remark}

We now provide an  a priori estimate for the solutions of problem~\eqref{sdfsde}.

\begin{Lemma}\label{lemma cdscsd}
Let $g\in L^2(\mathbb{R}^n)$ with $\|g\|_{\infty,\xi}<+\infty$.  Assume that $\psi\in \Psi$ is a solution of~\eqref{sdfsde} for some coefficients $ c_i\in\mathbb{R}$ with $i\in\left\{1,\cdots, n\right\}$. 

Then, for $ \epsilon>0$ sufficiently small,
\begin{equation}\label{dfgdfs}
	\|\psi\|_{\infty,\xi}\leqslant C\|g\|_{\infty,\xi}.
\end{equation}
\end{Lemma}

\begin{proof}
	We suppose by contradiction that there exist an infinitesimal sequence~$\epsilon_k$ as~$k\rightarrow +\infty$ such that, for any $k\in\mathbb{N}$, $\psi_k$ satisfies
	\begin{equation}\label{dsvds4}
			\begin{cases}
			-\Delta \psi_k+(-\Delta)^s\psi_k+\psi_k-pw_{\xi_k}^{p-1}\psi_k=g_k+\displaystyle
			\sum_{i=1}^{n}c_i^kZ_i^k \qquad &\text{in }\Omega_{\epsilon_k},\\
			\psi_k=0 &\text{in }\mathbb{R}^n\setminus\Omega_{\epsilon_k},\\ \displaystyle
			\int_{\Omega_{\epsilon_k}} \psi_k Z_i^k\, dx=0 &\text{for any  } i=1,\cdots, n,
		\end{cases}
	\end{equation}
for some $g_k\in L^2(\mathbb{R}^n)$ with $\|g_k\|_{\infty,\xi}<+\infty$. 
Here,  $\xi_k\in\Omega_{\epsilon_k}$ satisfies dist$(\xi_k,\partial\Omega_{\epsilon_k})\geqslant c/\epsilon_k $ for some~$c\in(0,1)$ and $$	Z_i^k:=\frac{\partial w_{\xi_k}}{\partial x_i}.$$ 
Additionally, suppose that 
\begin{equation}\label{	dsfsdf}
\|\psi_k\|_{\infty,\xi_k}=1 \quad \text{for any } k\in\mathbb{N} \qquad \text{and }\qquad \lim_{k\to+\infty} \|g_k\|_{\infty,\xi_k}= 0.
\end{equation}  

We claim that, for any given $r>0$, 
\begin{equation}\label{dsvdsddf}
	\lim_{k\to+\infty}\|\psi_k\|_{L^\infty(B_r(\xi_k))}= 0.
\end{equation}
We stress that once the desired estimate~\eqref{dsvdsddf} is established, then, for $r$ sufficiently large, from Lemmata~\ref{lemma decay of w} and~\ref{lemma bdjshfsd}, it follows that 
\begin{equation}\label{sdsdf}
	\begin{split}
	\|\psi_k\|_{\infty,\xi_k}&\leqslant C\left(\|\psi_k\|_{L^\infty(B_r(\xi_k))}+\|g_k\|_{\infty,\xi_k}+\sum_{i=1}^{n}|c_i^k|\, \|Z_i^k\|_{\infty,\xi_k}\right)\\
	&=C\left(\|\psi_k\|_{L^\infty(B_r(\xi_k))}+\|g_k\|_{\infty,\xi_k}+\sum_{i=1}^{n}|c_i^k|\, \|(1+|\cdot-\xi_k|)^\mu Z_i^k\|_{L^\infty(\mathbb{R}^n)}\right)\\
	&\leqslant C\left(\|\psi_k\|_{L^\infty(B_r(\xi_k))}+\|g_k\|_{\infty,\xi_k}+\sum_{i=1}^{n}|c_i^k|\right),
	\end{split}
\end{equation} 
up to renaming $C>0$, thanks to the fact that~$\mu<n+2s$ (recall~\eqref{fiuewfjke768578599}). 

Moreover,  Lemmata~\ref{lemma decay of w},~\ref{lemma gradient decay} and~\ref{lemma second gradient decay} allow us to use~\cite[Lemma~7.2]{MR3393677}. We thus infer that 
\begin{equation*}
	c_i^k:=\alpha^{-1} \int_{\mathbb{R}^n} g_k Z_i^k\, dx+f_i^k,
\end{equation*}
where $\alpha$ is as defined in Lemma~\ref{lemma Zi}.

Here above, we have considered a suitable $f_i^k\in\mathbb{R}$ with 
\begin{equation*}
	|f_i^k|\leqslant C_0\epsilon_k^{n/2} \left(\|\psi_k\|_{L^2(\mathbb{R}^n)}+\|g_k\|_{L^2(\mathbb{R}^n)}\right),
\end{equation*} 
for some positive constant $C_0$ independent of $k$.  

Furthermore, we point out that, for any measurable function~$\phi:\R^n\to\R$,   
\begin{equation}\label{dgerf}
	\|\phi\|_{L^2(\mathbb{R}^n)}\leqslant \|\phi\|_{\infty,\xi}\left(\int_{\mathbb{R}^n}(1+|x-\xi_k|)^{-2\mu}\, dx\right)^{1/2}\leqslant C_{n,\mu} \|\phi\|_{\infty,\xi},
\end{equation}
thanks to the fact that~$\mu>n/2$ (recall~\eqref{fiuewfjke768578599}).
Hence, using the decay of $Z_i^k$  in Lemma~\ref{lemma decay of w}, one has that
\begin{equation}\label{sddsf}
		|c_i^k|\leqslant C_1\alpha^{-1} \|g_k\|_{\infty,\xi_k}+C_2\epsilon_k^{n/2},
\end{equation}
for some constants $C_1$, $C_2>0$.

As a consequence of this,  by combining~\eqref{	dsfsdf}, \eqref{dsvdsddf} and~\eqref{sdsdf}, we obtain that
\begin{equation*}
\lim_{k\to+\infty}	\|\psi_k\|_{\infty,\xi_k}= 0,
\end{equation*}
which contradicts the fact that $ \|\psi_k\|_{\infty,\xi_k}=1$ for any $k\in\mathbb{N}$. This establishes the desired result~\eqref{dfgdfs}.


Accordingly, it only remains to check the claim in~\eqref{dsvdsddf}. We argue by contradiction and we suppose that 
there exist~$r$, $\gamma>0$ and a subsequence (still denoted by $\psi_k$)  such that  
\begin{equation}\label{sdfsddsf}
	\|\psi_{k}\|_{L^\infty(B_r(\xi_{k}))}\geqslant \gamma.
\end{equation}
We now split the proof into four steps.

\noindent{\bf Step 1.} 
We first prove that  
\begin{equation}\label{giueuiegtujst5948}
{\mbox{$\psi_k$ and $\nabla\psi_k$ are equicontinous.}}\end{equation}
Indeed, 
recall that $\psi_k=0$ outside $\Omega_{\epsilon_k}$.
We can thereby employ Lemma~\ref{lemma regularity},  owing to~\eqref{dgerf}, obtaining that
\begin{equation*}
	\begin{split}
&\sup\limits_{x\neq y}\frac{|\psi_k(x)-\psi_k(y)|}{|x-y|}+	\sup\limits_{x\neq y\, x,y\in\Omega_{\epsilon_k}}\frac{|\nabla\psi_k(x)-\nabla\psi_k(y)|}{|x-y|^\alpha}\\
&\leqslant C\left(\left|\left|g_k+pw_{\xi_k}^{p-1}\psi_k+\sum_{i=1}^{n}c_i^kZ_i^k\right|\right|_{L^\infty(\mathbb{R}^n)}+\left|\left|g_k+pw_{\xi_k}^{p-1}\psi_k+\sum_{i=1}^{n}c_i^kZ_i^k\right|\right|_{L^2(\mathbb{R}^n)}\right)\\
&\leqslant C\left(\|g_k\|_{\infty,\xi_k}+\|\psi_k\|_{\infty,\xi_k}+\sum_{i=1}^{n}|c_i^k|\right)
\end{split}
\end{equation*}
for some constant $C>0$ independent of $k$, thanks to the decay of $Z_i^k$ in Lemma~\ref{lemma decay of w} and the fact that $w_{\xi_k}^{p-1}$ is bounded.
{F}rom this, by combining~\eqref{	dsfsdf} with~\eqref{sddsf}, we deduce~\eqref{giueuiegtujst5948}.

We now define $\bar{\psi}_k(x):=\psi_k(x+\xi_k)$ and $\bar{\Omega}_{\epsilon_k}:=\left\{x\in\mathbb{R}^n\, :\,x\in\Omega_{\epsilon_k}-\xi_k \right\}$.
Due to~\eqref{dsvds4}, we see that 
\begin{equation}\label{gergr}
	\begin{cases}
		-\Delta \bar{\psi}_k+(-\Delta)^s\bar{\psi}_k+\bar{\psi}_k-pw^{p-1}\bar{\psi}_k=\bar{g}_k+\displaystyle\sum_{i=1}^{n}c_i^k\bar{Z}_i \qquad &\text{in }\bar{\Omega}_{\epsilon_k},\\
		\bar{\psi}_k=0 &\text{in }\mathbb{R}^n\setminus\bar{\Omega}_{\epsilon_k},\\
		\displaystyle
		\int_{\bar{\Omega}_{\epsilon_k}} \bar{\psi}_k \bar{Z}_i\, dx=0 &\text{for any  } i=1,\cdots, n,
	\end{cases}
\end{equation}
where $ \bar{g}_k(x):=g_k(x+\xi_k)$ and $\bar{Z}_i:=\partial w/\partial x_i $. 

Moreover, since dist$(\xi_k,\partial\Omega_{\epsilon_k})> c/\epsilon_k $, we have that~$B_{c/\epsilon_k}\subset \bar{\Omega}_{\epsilon_k}$. 
This entails that~$ \bar{\Omega}_{\epsilon_k}$ converges to~$\mathbb{R}^n$ as~$k\rightarrow\infty$.

Recalling~\eqref{	dsfsdf} and~\eqref{sdfsddsf}, we see that 
\begin{equation}\label{fvsddfs}
\|\bar{\psi}_k\|_{\infty,0}=\|\psi_k\|_{\infty,\xi_k}=1 \qquad \text{and}\qquad \|\bar{\psi}_k\|_{L^\infty(B_r)}=\|\psi_{k}\|_{L^\infty(B_r(\xi_{k}))}\geqslant \gamma.
\end{equation}

Furthermore, in light of~\eqref{giueuiegtujst5948}, one has that $\bar{\psi}_k$ and $\nabla\bar{\psi}_k$ are also equicontinuous. 
Thus, there exists a function~$ \bar{\psi}$ such that, up to subsequence,  $\bar{\psi}_k$  and $\nabla\bar{\psi}_k$ uniformly converge to $ \bar{\psi}$ and $\nabla\bar{\psi}$   on  compact sets, respectively.

As a consequence of this and~\eqref{fvsddfs}, \begin{equation}\label{dfgge}
	 \|\bar{\psi}\|_{\infty,0}\leqslant 1 \qquad \text{and}\qquad \|\bar{\psi}\|_{L^\infty(B_r)}\geqslant \gamma.
\end{equation}

\noindent{\bf Step 2.}
We now show that $\bar{\psi}\in H^1(\mathbb{R}^n)$.

Indeed, we notice that,  multiplying~\eqref{gergr} by $\bar{\psi}_k$ and  integrating over~$\bar{\Omega}_{\epsilon_k}$, since~$\bar{\psi}_k=0 $ outside~$ \bar{\Omega}_{\epsilon_k}$, one has that
\begin{equation*}
	\int_{\bar{\Omega}_{\epsilon_k}}|\nabla\bar{\psi}_k|^2+\int_{\bar{\Omega}_{\epsilon_k}}\bar{\psi}_k(-\Delta)^s\bar{\psi}_k+\int_{\bar{\Omega}_{\epsilon_k}}\bar{\psi}_k^2=\int_{\bar{\Omega}_{\epsilon_k}}pw^{p-1}\bar{\psi}_k^2+\int_{\bar{\Omega}_{\epsilon_k}}\bar{g}_k\bar{\psi}_k.
\end{equation*}
Owing to formula~(1.5) in \cite{MR3211861}, one infers that 
\begin{equation*}
	\int_{\bar{\Omega}_{\epsilon_k}}|\nabla\bar{\psi}_k|^2\leqslant \left(p\|w\|^{p-1}_{L^\infty(\mathbb{R}^n)}+\|\bar{g}_k\|_{L^2(\mathbb{R}^n)}\right)\| \bar{\psi}_k\|_{L^2(\mathbb{R}^n)}.
\end{equation*}
{F}rom this, utilizing~\eqref{dgerf} and  the decay of $Z_i$,
\begin{equation}\label{vsdvdsv}
	\int_{\bar{\Omega}_{\epsilon_k}}|\nabla\bar{\psi}_k|^2	\leqslant C\|\psi_k\|_{\infty,\xi_k}\left(1+\|g_k\|_{\infty,\xi_k}\right)\leqslant C
\end{equation}
up to renaming  $C>0$ independently of $k$, thanks to~\eqref{	dsfsdf}.

Moreover,  for any given $R>0$, there exists $K\in\mathbb{N}$ such that   $ B_R\subset \bar{\Omega}_{\epsilon_k}$ for every $k>K$, and therefore,
by Fatou Lemma and~\eqref{fvsddfs}, we see that
\begin{equation*}
	\begin{split}&
		\int_{B_R} (\bar{\psi}^2+|\nabla \bar{\psi}|^2) \, dx\leqslant \liminf\limits_{k\rightarrow\infty}\int_{\bar{\Omega}_{\epsilon_k}} (\bar{\psi}_k^2+|\nabla\bar{\psi}_k|^2)\, dx\\
		&\qquad\leqslant \liminf\limits_{k\rightarrow\infty} \left(\|\bar{\psi}_k\|_{\infty,0}\int_{\mathbb{R}^n}(1+|x|)^{-2\mu}\, dx+\int_{\bar{\Omega}_{\epsilon_k}}|\nabla\bar{\psi}_k|^2\, dx\right) \leqslant C,
	\end{split}
\end{equation*}
for some constant $C>0$ independent of $k$, thanks to~\eqref{vsdvdsv} and the fact that~$2\mu>n$. Consequently, sending~$R\to+\infty$, we conclude that~$ \bar{\psi}\in H^1(\mathbb{R}^n) $, as desired.

\noindent{\bf Step 3.}
We now claim that $\bar{\psi}$ solves, in the weak sense,
\begin{equation}\label{vsd}
	-\Delta \bar{\psi}+(-\Delta)^s\bar{\psi}+\bar{\psi}-pw^{p-1}\bar{\psi}=0 \qquad \text{in }\mathbb{R}^n.
\end{equation}
Indeed, it suffices to show that 
\begin{equation}\label{reg}
	\int_{\mathbb{R}^n} \left(-\Delta \varphi+(-\Delta)^s\varphi+\varphi-pw^{p-1}\varphi\right)\bar{\psi}=0 \qquad \text{in }\mathbb{R}^n,
\end{equation}
for every $\varphi\in C_0^\infty(\mathbb{R}^n)$. To this end, we observe that for each $\varphi\in C_0^\infty(\mathbb{R}^n) $, there exists $k_0\in\mathbb{N}$ such that $\varphi\in C_0^\infty(\bar{\Omega}_{\epsilon_k})$ for any $k>k_0$.

Therefore, we multiply the equation in~\eqref{gergr} by $\varphi$, and integrate over $\mathbb{R}^n$. In this way, we have that
\begin{equation}\label{vhjmi}
	\int_{\mathbb{R}^n}\left(-\Delta\varphi +(-\Delta)^s\varphi+\varphi-pw^{p-1}\varphi\right) \bar{\psi}_k=\int_{\mathbb{R}^n}\bar{g}_k\varphi +\sum_{i=1}^{n}c_i^k\int_{\mathbb{R}^n}\bar{Z}_i\varphi. 
\end{equation}
Furthermore, owing to~\eqref{	dsfsdf},  one has that 
\begin{equation}\label{fbdfvdf}\begin{split}&
	\lim_{k\to+\infty}\left|\int_{\mathbb{R}^n}\bar{g}_k\varphi\right|\leqslant
	\lim_{k\to+\infty} \|\bar{g}_k\|_{\infty,0}\int_{\mathbb{R}^n}(1+|x|)^{-2\mu}|\varphi|\\&\qquad
	\leqslant\lim_{k\to+\infty} \|{g}_k\|_{\infty,\xi_k}\int_{\mathbb{R}^n}|\varphi|\leqslant C\lim_{k\to+\infty}\|{g}_k\|_{\infty,\xi_k}= 0.\end{split}
\end{equation}
Additionally, in virtue of the decay of $\bar{Z}_i$ and~\eqref{sddsf}, we get 
\begin{equation}\label{fwfwe}\lim_{k\to+\infty}
\left|	\sum_{i=1}^{n}c_i^k\int_{\mathbb{R}^n}\bar{Z}_i\varphi\right|\leqslant C\lim_{k\to+\infty}\sum_{i=1}^{n}\left|c_i^k\right|=0.
\end{equation}

Moreover, given $r>0$, we estimate that 
\begin{equation*}
	\begin{split}
	&\left|	\int_{\mathbb{R}^n}\left(-\Delta\varphi +(-\Delta)^s\varphi+\varphi-pw^{p-1}\varphi\right) \bar{\psi}_k\, dx-\int_{\mathbb{R}^n}\left(-\Delta\varphi +(-\Delta)^s\varphi+\varphi-pw^{p-1}\varphi\right)\bar{\psi}\, dx\right|\\
	\leqslant &\int_{B_{r}}\left|-\Delta\varphi +(-\Delta)^s\varphi+\varphi-pw^{p-1}\varphi\right||\bar{\psi}_k-\bar{\psi}|\, dx\\
	&\qquad +\int_{\mathbb{R}\setminus B_{r}}\left|-\Delta\varphi
	 +(-\Delta)^s\varphi+\varphi-pw^{p-1}\varphi\right||\bar{\psi}_k-\bar{\psi}|\, dx\\
	 &=:A_1+A_2.
	\end{split}
\end{equation*}
Since $\varphi\in C_0^\infty(\bar{\Omega}_{\epsilon_{k_0}})$, there exists a constant $C>0$ depending on $\varphi$ such that,
for any~$x\in\mathbb{R}^n$,
\begin{equation*}
	\left|-\Delta\varphi(x)
	+(-\Delta)^s\varphi(x)+\varphi(x)-pw^{p-1}(x)\varphi(x)\right|\leqslant C(1+|x|)^{-n-2s}.
\end{equation*} 
Hence, we have that \begin{equation}\label{vsdfdsfds}
	\lim_{k\to+\infty}A_1\leqslant C\|\bar{\psi}_k-\bar{\psi}\|_{L^\infty(B_{r})} \int_{B_{r}}(1+|x|)^{-n-2s}\, dx\leqslant C\lim_{k\to+\infty} \|\bar{\psi}_k-\bar{\psi}\|_{L^\infty(B_{r})}=0, 
\end{equation}
thanks to the uniform convergence of $ \bar{\psi}_k$ to $\bar{\psi}$ on compact sets.  

Besides, by combining~\eqref{fvsddfs} with~\eqref{dfgge}, we see that 
\begin{equation}\label{fsdfds}
		A_2\leqslant C\int_{\mathbb{R}\setminus B_{r}}(1+|x|)^{-\mu-n-2s}\, dx\leqslant C r^{-2s}, 
\end{equation} up to renaming~$C$.

Let now~$\delta>0$. Thanks to~\eqref{fsdfds}, 
we then pick~$r$ sufficiently large such that~$A_2<\delta/4$. The choice of~$r$ is now fixed once and for all.

Moreover, in light of~\eqref{vsdfdsfds}, we can find $k_1>0$, depending only
on~$\delta$, $ n$, $s$ and~$\varphi$, such that, for any $k>k_1$, 
\begin{eqnarray*}&&
	\left|	\int_{\mathbb{R}^n}\left(-\Delta\varphi +(-\Delta)^s\varphi+\varphi-pw^{p-1}\varphi\right)\bar{\psi}\, dx\right|\\&&\qquad
	<\left|\int_{\mathbb{R}^n}\left(-\Delta\varphi +(-\Delta)^s\varphi+\varphi-pw^{p-1}\varphi\right) \bar{\psi}_k\, dx\right|+\delta/2.
\end{eqnarray*}
Hence,  utilizing~\eqref{vhjmi},~\eqref{fbdfvdf} and~\eqref{fwfwe}, there exists~$ k_2>k_1$ such that, for any $k>k_2$, 
\begin{equation*}
	\left|	\int_{\mathbb{R}^n}\left(-\Delta\varphi +(-\Delta)^s\varphi+\varphi-pw^{p-1}\varphi\right)\bar{\psi}\, dx\right|<\frac{\delta}{2}+\frac{\delta}{2}=\delta.
\end{equation*}
Sending~$\delta\to0$, we obtain the desired result~\eqref{reg},
and thus~$\bar{\psi}\in H^1(\mathbb{R}^n)$ is a weak solution to~\eqref{vsd}.

\noindent{\bf Step 4. }
Now we prove that $\bar{\psi}\equiv0$.

Indeed, given $r>0$, we have that 
\begin{equation}\label{fsdf}
		\int_{\mathbb{R}^n}(\bar{\psi}-\bar{\psi}_k) \bar{Z}_i\, dx= \int_{B_r}(\bar{\psi}-\bar{\psi}_k) \bar{Z}_i\, dx+\int_{\mathbb{R}^n\setminus B_r}(\bar{\psi}-\bar{\psi}_k) \bar{Z}_i\, dx.
\end{equation}
Since $\bar{\psi}_k$   uniformly converges to $ \bar{\psi}$   on  compact sets and $\bar{Z}_i$ is bounded, we obtain that 
\begin{equation}\label{fffd}
	\lim_{k\to+\infty}\int_{B_r}(\bar{\psi}-\bar{\psi}_k) \bar{Z}_i\, dx=0.
\end{equation}
Also, by combining~\eqref{fvsddfs} with~\eqref{dfgge},
and using the decay of $\bar{Z}_i$,
  \begin{equation}\label{dsfsdf}
	\int_{\mathbb{R}^n\setminus B_r}(\bar{\psi}-\bar{\psi}_k) \bar{Z}_i\, dx\leqslant C\int_{\mathbb{R}^n\setminus B_r} \frac{dx}{|x|^{n+2s}}\leqslant Cr^{-2s}
\end{equation}
which tends to $0$ as $r \rightarrow +\infty$.  

Moreover, we observe that,  \begin{equation*}
	\lim_{k\to+\infty}\int_{\mathbb{R}^n\setminus \bar{\Omega}_{\epsilon_k}}  \bar{\psi}_k \bar{Z}_i\, dx\leqslant C\lim_{k\to+\infty}\int_{\mathbb{R}^n\setminus B_{c/\epsilon_k}}\frac{dx}{|x|^{n+2s}}
	\leqslant C\lim_{k\to+\infty}\epsilon_k^{2s}= 0.
\end{equation*}
As a consequence of this and the orthogonality condition in~\eqref{gergr}, using~\eqref{fffd} and~\eqref{dsfsdf},  we obtain from \eqref{fsdf} that
\begin{equation}\label{uirewofvhjer7845trfgaj}
	\int_{\mathbb{R}^n}\bar{\psi} \bar{Z}_i\, dx=0.
\end{equation}

In addition, since $\bar{\psi}\in H^1(\mathbb{R}^n)$ is a weak solution to~\eqref{vsd}, employing the nondegeneracy assumption in~\eqref{NOBS-DEFG}, we have that
\begin{equation*}
	\bar{\psi}=\sum_{j=1}^{n}\beta_j\frac{\partial w}{\partial x_j},
\end{equation*}
for some coefficients $\beta_i\in\mathbb{R}$. 

Thus, from~\eqref{uirewofvhjer7845trfgaj} and Lemma~\ref{lemma Zi}, we conclude that, for all~$i\in\{1,\cdots,n\}$,
\begin{equation*}
	0=\sum_{j=1}^{n}\beta_j \int_{\mathbb{R}^n}\bar{Z}_i\bar{Z}_j\, dx=\alpha\beta_i,
\end{equation*}
which entails that $\bar{\psi}\equiv 0$,
as desired.

Notice now that the fact that~$\bar{\psi}\equiv 0$
contradicts~\eqref{dfgge}, yielding the desired result in~\eqref{dsvdsddf}. Therefore, the proof of Lemma~\ref{lemma cdscsd} is completed.
\end{proof}

Our aim is now to use the Fredholm's alternative to prove Theorem~\ref{th unique about g}.
For this, we now prove an auxiliary result, as follows.

\begin{Lemma}\label{lemma auxiliary result}
	Let $g\in L^2(\mathbb{R}^n)$ with $\|g\|_{\infty,\xi}<+\infty$.  Then, the problem   
	\begin{equation}\label{gfgu}
			-\Delta \psi+(-\Delta)^s\psi+\psi=g+\sum_{i=1}^{n}c_iZ_i \qquad \text{in }\Omega_\epsilon
	\end{equation}
has a unique solution $\psi\in \Psi$, for suitable
coefficients $c_i\in\mathbb{R}$ for any $ i\in\left\{1,\cdots,n\right\}$.

	Moreover, there exists a constant $C>0$ such that 
	\begin{equation}\label{sdfsdvjo}
		\|\psi\|_{\infty,\xi}\leqslant C\|g\|_{\infty,\xi}.
	\end{equation}
\end{Lemma}

\begin{proof}
	\noindent{\bf Step 1.} We first prove the uniqueness of the solution to~\eqref{gfgu}. 
	
	Assume that $\psi_1$, $\psi_2\in\Psi$ solve problem~\eqref{gfgu}
	and set~$\psi_\star:=\psi_1-\psi_2$. Then, 
	\begin{equation}\label{rgerg}
		-\Delta \psi_\star+(-\Delta)^s\psi_\star+\psi_\star=\sum_{i=1}^{n}b_iZ_i \qquad \text{in }\Omega_\epsilon
	\end{equation}
	for some constants $b_i\in\mathbb{R}$. 
	
	Moreover, we multiply~\eqref{rgerg} by $\psi_\star$ and integrate over $\Omega_\epsilon$. In this way, we see that
	\begin{equation*}
		\int_{\Omega_\epsilon} 	|\nabla \psi_\star|^2\, dx+\int_{\mathbb{R}^n} |(-\Delta)^{s/2}\psi_\star|^2\, dx+\int_{\Omega_\epsilon} |\psi_\star|^2\, dx=0
	\end{equation*}
	thanks to the orthogonality condition.  This yields that~$\psi_\star=0$
	a.e., and therefore~$\psi_1\equiv\psi_2$ in $\mathbb{R}^n$, since~$\psi_1$, $\psi_2\in C(\mathbb{R}^n)$ (see e.g. \cite[Theorem~1.3]{SVWZ23}).
	
	\noindent{\bf Step 2.}  We now show the existence of a solution to~\eqref{gfgu}.
	
For this, formula (1.5) in \cite{MR3211861} yields that, for every $ \psi$, $\varphi\in\Psi$,
\begin{equation*}
	\int_{\Omega_\epsilon} (-\Delta)^s\psi\, \varphi \, dx=\int_{\mathbb{R}^n} (-\Delta)^{s/2}\psi\, (-\Delta)^{s/2}\varphi\, dx.
\end{equation*}
Thus, for a given $g\in L^2(\mathbb{R}^n)$, we  look for a function $\psi\in \Psi$ such that, for every $\varphi\in\Psi$,
\begin{equation}\label{dfbjdvn}
	\int_{\Omega_\epsilon} 	\nabla \psi\cdot\nabla \varphi\, dx+\int_{\mathbb{R}^n} (-\Delta)^{s/2}\psi\, (-\Delta)^{s/2}\varphi\, dx+\int_{\Omega_\epsilon} \psi\, \varphi\, dx=\int_{\Omega_\epsilon} g\,\varphi\, dx.
\end{equation}
	
	We notice that 
	\begin{equation*}
		\left\langle\varphi,\psi \right\rangle:= \int_{\Omega_\epsilon} 	\nabla \psi\cdot\nabla \varphi\, dx+\int_{\mathbb{R}^n} (-\Delta)^{s/2}\psi\, (-\Delta)^{s/2}\varphi\, dx+\int_{\Omega_\epsilon} \psi\, \varphi\, dx
	\end{equation*}
	defines an inner product in $\Psi$, and that
	\begin{equation*}
		\mathcal{F}(\varphi):= \int_{\Omega_\epsilon} g\, \varphi\, dx
	\end{equation*}
	is a linear and continuous functional in $\Psi$. Therefore, in virtue of Riesz's Theorem, there exists a unique $\psi\in\Psi$ satisfying~\eqref{dfbjdvn}.
	
	We next prove that $\psi $ is a weak solution to~\eqref{gfgu}, that is, for every $\phi\in H^1(\mathbb{R}^n)$ with $\phi=0$ outside $\Omega_\epsilon$,  
	\begin{equation}\label{dfbjdvnfg}
		\int_{\Omega_\epsilon} 	\nabla \psi\cdot\nabla \phi\, dx+\int_{\mathbb{R}^n} (-\Delta)^{s/2}\psi\, (-\Delta)^{s/2}\phi\, dx+\int_{\Omega_\epsilon} \psi\, \phi\, dx=\int_{\Omega_\epsilon} g\,\phi\, dx+\sum_{i=1}^{n}c_i\int_{\Omega_\epsilon}Z_i\, \phi\, dx.
	\end{equation}
	Indeed,  we take a radial cutoff function $\tau\in C_0^\infty(\Omega_\epsilon)$ of the form~$\tau(x):=\tau_0(|x-\xi|)$, where~$\tau_0\in C_0^\infty([-1,1])$. Furthermore, 
	for every $\phi\in H^1(\mathbb{R}^n)$ with $\phi=0$ outside $\Omega_\epsilon$, 
	we define 
	\begin{equation}\label{fyghdwjqs478365439fgkfdghcsbxakdkjwqlkjhgf}
		\bar{\phi}:=\phi- \sum_{j=1}^{n}\left(\tilde{\alpha}^{-1}\int_{\mathbb{R}^n} {\phi}Z_j\, dx\right)\tilde{Z_j}:=\phi- \sum_{j=1}^{n}\lambda_j\tilde{Z_j}
	\end{equation}
	where $\tilde{Z_i}$ and $\tilde{\alpha}$ are as in Lemma~\ref{lemma Zi}. 
	
	Since $ \tilde{Z_i}\equiv 0$ outside $\Omega_\epsilon$, so does $\bar{\phi}$ and we thus deduce that, for any~$i=1,\cdots,n$,
	\begin{equation*}\begin{split}
			&\int_{\Omega_\epsilon}\bar{\phi}Z_i\, dx=\int_{\mathbb{R}^n} \bar{\phi}Z_i\, dx=\int_{\mathbb{R}^n} {\phi}Z_i\, dx-\sum_{j=1}^{n}\left(\tilde{\alpha}^{-1}\int_{\mathbb{R}^n} {\phi}Z_j\, dx\right)\int_{\mathbb{R}^n} Z_i\tilde{Z_j}\, dx\\
			&\qquad= \int_{\mathbb{R}^n} {\phi}Z_i\, dx-\left(\tilde{\alpha}^{-1}\int_{\mathbb{R}^n} {\phi}Z_i\, dx\right)\tilde{\alpha}=0.
	\end{split}
	\end{equation*}
	This proves that $\bar{\phi}\in\Psi$. 
	
Accordingly, we can use $\bar{\phi}$ as a test function in~\eqref{dfbjdvn} and,
recalling~\eqref{fyghdwjqs478365439fgkfdghcsbxakdkjwqlkjhgf}, we obtain that 
	\begin{equation*}
		\begin{split}
		\int_{\Omega_\epsilon} 	\nabla \psi\cdot\nabla {\phi}\, dx+\int_{\mathbb{R}^n} (-\Delta)^{s/2}\psi\, (-\Delta)^{s/2}{\phi}\, dx+\int_{\Omega_\epsilon} \psi\,{\phi}\, dx
		= \int_{\Omega_\epsilon} g\,{\phi}\, dx+	\sum_{j=1}^{n}c_j\int_{\mathbb{R}^n} {\phi}Z_j\, dx,
	\end{split}
	\end{equation*}
	where  \begin{equation*}
	c_j:=\tilde{\alpha}^{-1}\left(\int_{\Omega_\epsilon} 	\nabla \psi\cdot\nabla \tilde{Z_j}\, dx+\int_{\mathbb{R}^n} (-\Delta)^{s/2}\psi\, (-\Delta)^{s/2}\tilde{Z_j}\, dx+\int_{\Omega_\epsilon} (\psi-g)\tilde{Z_j}\, dx\right),
	\end{equation*}
which proves~\eqref{dfbjdvnfg}, as desired.


\noindent{\bf Step 3.} We now prove~\eqref{sdfsdvjo}.

{F}rom Lemma~\ref{lemma dgfusdhi}, it follows that
\begin{equation}\label{dfbdfg}
	\|\psi\|_{\infty,\xi}\leqslant C\left(\|g\|_{\infty,\xi}+\sum_{i=1}^{n}|c_i|\, \|Z_i\|_{\infty,\xi}\right)\leqslant C\left(\|g\|_{\infty,\xi}+\sum_{i=1}^{n}|c_i|\right)
\end{equation}
up to renaming $C>0$, thanks to the decay of $Z_i$ in Lemma~\ref{lemma decay of w} and the fact that~$\mu<n+2s$.

We claim that there exists a constant $C>0$ such that 
\begin{equation}\label{ndfgfd}
	\sum_{i=1}^{n}|c_i|\leqslant C\left(\|\psi\|_{L^2(\mathbb{R}^n)}+\|g\|_{\infty,\xi}\right).
\end{equation} 
Indeed, 
let $\tilde{Z_j}$ and~$\tilde{\alpha}$ be as in Lemma~\ref{lemma Zi}.
We multiply~\eqref{gfgu} by~$\tilde{Z_j}$
and integrate over $\Omega_\epsilon$ to infer that 
\begin{equation*}
\tilde{\alpha} c_j=	\int_{\mathbb{R}^n} \left(-\Delta \tilde{Z_j}+(-\Delta)^s\tilde{Z_j}+\tilde{Z_j}\right)\psi\, dx-\int_{\mathbb{R}^n} g\tilde{Z_j}\, dx.
\end{equation*} 
As a consequence, in virtue of Lemmata~\ref{lemma decay of w}, \ref{lemma gradient decay} and~\ref{lemma second gradient decay}, and recalling also~\eqref{dgerf}, one sees that, for any~$j=1,\cdots,n$,
\begin{equation*}
 \begin{split}
 	|c_j|&=\tilde{\alpha}^{-1}\left(\|\psi\|_{L^2(\mathbb{R}^n)}\|-\Delta \tilde{Z_j}+(-\Delta)^s\tilde{Z_j}+\tilde{Z_j}\|_{L^2(\mathbb{R}^n)}+\|g\|_{L^2(\mathbb{R}^n)}\|\tilde{Z_j}\|_{L^2(\mathbb{R}^n)}\right)	\\
 	&\leqslant C\left(\|\psi\|_{L^2(\mathbb{R}^n)}\|\tilde{Z_j}\|_{H^2(\mathbb{R}^n)}+\|g\|_{L^2(\mathbb{R}^n)}\|\tilde{Z_j}\|_{L^2(\mathbb{R}^n)}\right)\\
 	&\leqslant C\left(\|\psi\|_{L^2(\mathbb{R}^n)}+\|g\|_{\infty,\xi}\right),
 \end{split}
\end{equation*}
up to renaming  $C>0$, which shows the desired result~\eqref{ndfgfd}. 

Moreover, recalling that $\psi\in\Psi$, we multiply ~\eqref{gfgu} by $\psi$ and integrate over $\Omega_\epsilon$, concluding that 
$$
	\|\psi\|_{L^2(\mathbb{R}^n)}\leqslant \|g\|_{L^2(\mathbb{R}^n)}
.$$
Inserting this and~\eqref{ndfgfd} into~\eqref{dfbdfg}, owing to~\eqref{dgerf}, we obtain~\eqref{sdfsdvjo} and finish the proof of Lemma~\ref{lemma auxiliary result}.
\end{proof}

In light of Lemma~\ref{lemma auxiliary result}, for any given $g\in L^2(\mathbb{R}^n)$ with $\|g\|_{\infty,\xi}<\infty$,
we denote by $\mathcal{T}[g]\in \Psi$  the unique solution to
problem~\eqref{gfgu}. In particular, we have that $\mathcal{T}$ is a linear operator satisfying \begin{equation}\label{vsdvsdf}
 	\|\mathcal{T}[g]\|_{\infty,\xi}\leqslant C\|g\|_{\infty,\xi}. 
 \end{equation}
 
We define the Banach space 
 \begin{equation*}
 	Y_\infty:=\left\{\psi:\mathbb{R}^n\rightarrow \mathbb{R}\,:\, \|\psi\|_{\infty,\xi}<+\infty\right\}
 \end{equation*}
equipped with the norm $\|\cdot\|_{\infty,\xi}$.

With this preparatory work, we can now check the existence result for  the linear problem, as stated in Theorem~\ref{th unique about g}.

\begin{proof}[Proof of Theorem~\ref{th unique about g}]
	We observe that, to find a solution of~\eqref{sdfsde}, it suffices to look for a function $\psi\in\Psi$ such  that 
	\begin{equation}\label{fvs}
		\psi-\mathcal{T}[pw_\xi^{p-1}\psi]=\mathcal{T}[g].
	\end{equation}
For this, we set \begin{equation*}
	\mathcal{A}[\psi]:=\mathcal{T}[pw_\xi^{p-1}\psi].
\end{equation*}
Hence, by~\eqref{vsdvsdf}, and recalling also that $w_\xi$ is bounded,
one infers that, for every $\psi\in \Psi$,
\begin{equation*}
	\|\mathcal{A}[\psi]\|_{\infty,\xi}
	=\|	\mathcal{T}[pw_\xi^{p-1}\psi]\|_{\infty,\xi}\le C\|w_\xi^{p-1}\psi
	\|_{\infty,\xi}
	\leqslant C	\|\psi\|_{\infty,\xi},
\end{equation*}
up to relabelling~$C>0$.

We now claim that
\begin{equation}\label{fdbvdf}
	\mathcal{A}:Y_\infty\rightarrow Y_\infty \text{ is a compact operator with respect to the norm } \|\cdot\|_{\infty,\xi}.
\end{equation}

We point out that, once this claim is established, 
the conclusions of Theorem~\ref{th unique about g} will follow. Indeed,
since~$\mathcal{T}[pw_\xi^{p-1}\psi]$, $\mathcal{T}[g]\in\Psi$ (see Lemma~\ref{lemma auxiliary result}), one has that~$\psi\in\Psi$.

Moreover, since  $\psi=\mathcal{A}[\psi]+\mathcal{T}[g] $ solves~\eqref{sdfsde}, in virtue of Lemma~\ref{lemma cdscsd},  one has that if $g=0$ then $\psi=0$ is the unique solution of~\eqref{sdfsde}. Therefore, by the Fredholm's alternative, we deduce that,
for any~$g\in Y_\infty$, there exists a unique~$\psi\in Y _\infty$ that solves~\eqref{fvs}.

Also, the estimate~\eqref{sdfsdv} follows from Lemma~\ref{lemma cdscsd} and this would complete the proof of Theorem~\ref{th unique about g}.

Hence, we will now focus on the proof of~\eqref{fdbvdf}. For this, we assume that $\psi_k$ is a bounded sequence in $Y_\infty$ with respect to $\|\cdot\|_{\infty,\xi}$. Our goal is to show that there exists a subsequence (still denoted by $\psi_k$) and a function $\phi$ such that 
\begin{equation}\label{0987654mvcnbktryRTRE}
\lim_{k\to+\infty}\|\mathcal{A}[\psi_k]-\phi\|_{\infty,\xi}=0.
\end{equation}

To check this, we observe that, in view of Lemma~\ref{lemma regularity}, by combining the decay of~$Z_i$ in Lemma~\ref{lemma decay of w} with the boundedness of~$w_\xi$, recalling also~\eqref{dgerf} and~\eqref{ndfgfd},
one deduces that
\begin{equation*}
	\begin{split}
		\sup\limits_{x\neq y}\frac{|\mathcal{A}[\psi_k](x)-\mathcal{A}[\psi_k](y)|}{|x-y|}&\leqslant C\left(\left|\left|pw_{\xi}^{p-1}\psi_k+\sum_{i=1}^{n}c_i^kZ_i\right|\right|_{L^\infty(\mathbb{R}^n)}+\left|\left|pw_{\xi}^{p-1}\psi_k+\sum_{i=1}^{n}c_i^kZ_i\right|\right|_{L^2(\mathbb{R}^n)}\right)\\
		&\leqslant C\left(\|\psi_k\|_{L^2(\mathbb{R}^n)}+\|\psi_k\|_{L^\infty(\mathbb{R}^n)}+\sum_{i=1}^{n}|c_i^k|\right)\\
		&\leqslant C\left(\|\psi_k\|_{\infty,\xi}+\sum_{i=1}^{n}|c_i^k|\right)\\&\leqslant C,
	\end{split}
\end{equation*}
up to renaming $C>0$. 

This shows that the sequence $\mathcal{A}[\psi_k]$ is equicontinuous, thus there exists a function $\phi$ such that, for any $R>0$, 
\begin{equation}\label{dscdsvds}
	\lim_{k\to+\infty}\|\mathcal{A}[\psi_k]-\phi\|_{L^\infty(B_R(\xi))}= 0.
\end{equation}

Furthermore, we observe that 
\begin{equation}\label{fhewjhfweuoftuowe6784376473hfkj}
	\begin{cases}
	\left(	-\Delta +(-\Delta)^s+1\right)\mathcal{A}[\psi_k]=pw_\xi^{p-1}\psi_k+\displaystyle
	\sum_{i=1}^{n}c_i^kZ_i \qquad& \text{in }\Omega_\epsilon,\\
	\mathcal{A}[\psi_k]=0 &\text{in }\mathbb{R}^n\setminus\Omega_\epsilon.
	\end{cases}
\end{equation}

Since $\psi_k$ is a uniformly bounded sequence, owing to~\eqref{bvsvs},  \eqref{ndfgfd} and the decay of $Z_i$ in Lemma~\ref{lemma decay of w}, it follows that 
\begin{equation}\label{fwfjkasasdfgh0985}
	\begin{split}
	  \left|pw_\xi^{p-1}\psi_k+\sum_{i=1}^{n}c_i^kZ_i\right|  &\leqslant C\left(\|\psi_k\|_{\infty,\xi}(1+|x-\xi|)^{-(n+2s)(p-1)-\mu}+\sum_{i=1}^{n}|c_i^k|(1+|x-\xi|)^{-(n+2s)}\right)\\
	  &\leqslant \Lambda_0 (1+|x-\xi|)^{-\mu-\theta},
	\end{split}
\end{equation}
where the constant $\Lambda_0>0$ is independent of $k$,  and
\begin{equation}\label{choicethetar4383y}
\theta:=\min\left\{\frac{(n+2s)(p-1)}{\mu},\frac{n+2s-\mu}{\mu}\right\}>0.
\end{equation}

	Moreover, we define
\begin{equation*}
	\eta(x):=\int_{\mathbb{R}^n}\frac{\mathcal{K}(y) }{(1+|x-y|)^{\mu+\theta}}\, dy
\end{equation*} 
where $\mathcal{K}$ is given in~\eqref{mathcal{K}}. 
Hence, we have that
\begin{equation}\label{dsfds1}
	-\Delta\eta +(-\Delta)^s\eta+\eta={(1+|x|)^{-\mu-\theta}} \qquad \text{in }\mathbb{R}^n.
\end{equation}

Now, one can repeat the computation in~\eqref{lkjhgfd09876543mnbvc} with~$\mu$ replaced by~$\mu+\theta$. Notice indeed that the only assumption needed to carry over that
computation is that~$\mu+\theta\le n+2s$, which is guaranteed by~\eqref{choicethetar4383y}.
In this way, one obtains that there exists $C_0>0$ such that
\begin{equation*}
	\eta(x)\leqslant C_0(1+|x|)^{-\mu-\theta}\qquad \text{for any } x\in\mathbb{R}^n.
\end{equation*}

We denote $\eta_\xi:=\eta(x-\xi)$ and $\tilde{\eta}:=\Lambda_0\eta_\xi-\mathcal{A}[\psi_k]$. By combining~\eqref{fhewjhfweuoftuowe6784376473hfkj}
and~\eqref{dsfds1}, one has that, in~$\Omega_\epsilon$,
\begin{eqnarray*}
&&-\Delta\tilde{\eta} +(-\Delta)^s\tilde{\eta}+\tilde{\eta}
=\Lambda_0{(1+|x|)^{-\mu-\theta}} -
pw_\xi^{p-1}\psi_k-\displaystyle \sum_{i=1}^{n}c_i^kZ_i \ge0,
\end{eqnarray*} thanks to~\eqref{fwfjkasasdfgh0985}.
Also, $\tilde{\eta}=\Lambda_0\eta_\xi>0$ in~$\mathbb{R}^n\setminus\Omega_\epsilon$.

Thus, the Maximum Principle gives that $\tilde\eta\geqslant 0$ in $\mathbb{R}^n$. This yields that, for every $x\in\mathbb{R}^n$,
\begin{equation*}
	\mathcal{A}[\psi_k](x)\leqslant \Lambda_0\eta_\xi(x)\leqslant \Lambda_0C_0(1+|x-\xi|)^{-\mu-\theta}.
\end{equation*}
 As a consequence of this, 
we have that
\begin{equation*}
	\begin{split}
		\|\mathcal{A}[\psi_k]-\phi\|_{\infty,\xi}&=\sup\limits_{x\in\mathbb{R}^n} (1+|x-\xi|)^\mu\left(\mathcal{A}[\psi_k]-\phi\right)\\
		&\leqslant \sup\limits_{x\in B_R(\xi)} (1+|x-\xi|)^\mu\left(\mathcal{A}[\psi_k]-\phi\right)+\sup\limits_{x\in\mathbb{R}^n\setminus B_R(\xi)} (1+|x-\xi|)^\mu\left(\mathcal{A}[\psi_k]-\phi\right)\\
		&\leqslant (1+R)^\mu \|\mathcal{A}[\psi_k]-\phi\|_{L^\infty(B_R(\xi))}+ 2\Lambda_0C_0\sup\limits_{x\in\mathbb{R}^n\setminus B_R(\xi)} (1+|x-\xi|)^{-\theta}.
	\end{split}
\end{equation*}
Now, for any $\delta>0$, there exists $R_0>0$ (which is now fixed once and for all) such that 
$$2\Lambda_0C_0\sup\limits_{x\in\mathbb{R}^n\setminus B_{R_0}(\xi)} (1+|x-\xi|)^{-\theta}<\frac{\delta}2.$$
Also, using~\eqref{dscdsvds}, taking $k$ sufficiently large, one has that
\begin{equation*}
	(1+R_0)^\mu \|\mathcal{A}[\psi_k]-\phi\|_{L^\infty(B_{R_0}(\xi))}\leqslant \frac{\delta}2.
\end{equation*} 
Gathering these pieces of information, one concludes that
\begin{equation*}
	\|\mathcal{A}[\psi_k]-\phi\|_{\infty,\xi}\le\delta.
\end{equation*}
Sending~$\delta\to0$, we obtain~\eqref{0987654mvcnbktryRTRE}.
This establishes the claim in~\eqref{fdbvdf}.
\end{proof}

Now, thanks to Theorem~\ref{th unique about g},
we can consider a linear and continuous operator from~$Y_\infty$ into itself with respect to the norm $\|\cdot\|_{\infty,\xi}$. Namely, 
for any $g\in L^2(\mathbb{R}^n)$ with $\|g\|_{\infty,\xi}<+\infty$,  
we denote that
\begin{equation}\label{dfdsvsd}
\psi:=\mathcal{I}_\xi[g] \text{ is the unique solution of~\eqref{sdfsde} in } Y_\infty.
\end{equation}
 
\subsection{The nonlinear project problem}
In this subsection, we aim to solve the nonlinear projected problem
\begin{equation}\label{nonlinear projected problem}
	\begin{cases}
			-\Delta \psi+(-\Delta)^s\psi+\psi-pw_\xi^{p-1}\psi=E(\psi)+\displaystyle
			\sum_{i=1}^{n}c_iZ_i \qquad &\text{in }\Omega_\epsilon,\\
		\psi=0 &\text{in }\mathbb{R}^n\setminus\Omega_\epsilon,\\
		\displaystyle
		\int_{\Omega_\epsilon} \psi Z_i\, dx=0&\text{for any } i=1,\cdots,n,
	\end{cases}
\end{equation} 
where $E(\psi)$ is given by~\eqref{sdvds}.

We will establish the following existence result for
the nonlinear problem~\eqref{nonlinear projected problem} by making use of
the Contraction Mapping Theorem.

\begin{Theorem}\label{th:nonlinear problem }
	If $\epsilon>0$ sufficiently small, there exists a unique solution $\psi\in \Psi$ of~\eqref{nonlinear projected problem} for some coefficients $c_i\in\mathbb{R}$, with~$i=1,\cdots,n$, such that 
	\begin{equation}\label{dsf}
		\|\psi\|_{\infty,\xi}\leqslant C\epsilon^{\gamma_1}
	\end{equation} 
for a suitable constant $C>0$, where 
\begin{equation}\label{hewgamma068jkagamma}
\gamma_1:=\min\left\{n+2s,p(n+2s)-\mu\right\}>0.\end{equation}
\end{Theorem}

We first show the following:

\begin{Lemma}\label{lemma:  E(psi)}
	Let $\psi\in\Psi$. Then, there exists a positive constant $C_0$ such that 
	\begin{equation*}
		\|E(\psi)\|_{\infty,\xi}\leqslant C_0\left(\|\psi\|^2_{\infty,\xi}+\|\psi\|^p_{\infty,\xi}+\epsilon^{\gamma_1}\right),
	\end{equation*}
where $\gamma_1$ is as in~\eqref{hewgamma068jkagamma}.
\end{Lemma}

\begin{proof} We
recall that $ v_\xi(x)=w_\xi(x)-\bar{u}_\xi(x)>0$ in $\mathbb{R}^n$.
By combining Lemmata~\ref{lemma hepsilon} and ~\ref{lemma comparable for h}, 
and recalling also~\eqref{fvfdv} and~\eqref{fdgd},
one finds that
there exists a constant $C_1$ such that,
for any  $x\in\mathbb{R}^n$,  
\begin{equation}\label{nbfdv}
v_\xi(x)\leqslant C_1(1+|x-\xi|)^{-(n+2s)}\quad \text{and}\quad 	v_\xi(x)\leqslant C_1\epsilon^{n+2s}.
\end{equation}

Furthermore, 
	by the definition of $E(\psi)$, we see that
	\begin{equation*}
	E(\psi)=(w_\xi+\psi-v_\xi)^p-w_\xi^p-pw_\xi^{p-1}(\psi-v_\xi)-pw_\xi^{p-1}v_\xi.
	\end{equation*}
Referring to page 122 in \cite{MR2186962} and recalling~\eqref{bvsvs},  we deduce that, for every $x\in\mathbb{R}^n$,
\begin{equation*}
		|E(\psi)(x)|\leqslant C_2\left(|\psi(x)-v_\xi(x)|^p+|\psi(x)-v_\xi(x)|^2+p(1+|x-\xi|)^{-(p-1)(n+2s)}v_\xi(x)\right)
\end{equation*}
 where the positive constant $C_2$ depends on $p$ and the ground state $w$. 

{F}rom this and~\eqref{nbfdv}, using that $n+2s>\mu$, one finds that
\begin{equation*}
		\begin{split}
		\|E(\psi)\|_{\infty,\xi}&\leqslant C_0\left(\|\psi\|_{\infty,\xi}^p+\|\psi\|_{\infty,\xi}^2+(1+|x-\xi|)^{\mu-(p-1)(n+2s)}v_\xi(x)\right)\\
		&\leqslant C_0\left(\|\psi\|_{\infty,\xi}^p+\|\psi\|_{\infty,\xi}^2+\epsilon^{\gamma_1}\right)
		\end{split}
\end{equation*}
for some positive constant $C_0>0$, where~$\gamma_1=n+2s$ if~$\mu\leqslant(p-1)(n+2s)$, and~$\gamma_1=p(n+2s)-\mu$ if~$\mu>(p-1)(n+2s)$.
\end{proof}

\begin{proof}[Proof of Theorem~\ref{th:nonlinear problem }]
	Recalling the definition of $\mathcal{I}_\xi$ in~\eqref{dfdsvsd}, we will show by the Contraction Mapping Theorem that  there exists a unique $\psi$ satisfying~\eqref{dsf} such that 
	\begin{equation*}
		\psi=\mathcal{I}_\xi[E(\psi)].
	\end{equation*}
To this end, we set
\begin{equation*}
	\mathcal{M}_\xi[\psi]:=\mathcal{I}_\xi[E(\psi)].
\end{equation*}
Also, we pick a constant $C_\star>0$ and $\epsilon>0$ sufficiently small (to be specified later on, see formula~\eqref{svsfd} below), and we set
\begin{equation*}
	\mathcal{S}:=\left\{\psi\in Y_\infty\, |\,\, \|\psi\|_{\infty,\xi}<C_\star\epsilon^{\gamma_1}\right\}.
\end{equation*} 

We claim that 
\begin{equation}\label{vdfvdf}
	\mathcal{M}_\xi \text{ is a contraction from } \mathcal{S} \text{ into itself with the norm } \|\cdot\|_{\infty,\xi}.
\end{equation}
We point out that, this claim and the Contraction Mapping Theorem yield the existence of a unique solution $\psi\in \mathcal{S}$ to~\eqref{nonlinear projected problem}, which concludes the proof. 

Thus, it remains to show that~\eqref{vdfvdf} holds true. For this, we first prove that \begin{equation}\label{dvsvds}
\text{if~$\psi\in\mathcal{S}$ then~$\mathcal{M}_\xi[\psi]\in \mathcal{S}$}.
\end{equation}
Indeed, if $\psi\in\mathcal{S},$ from Theorem~\ref{th unique about g} and Lemma~\ref{lemma:  E(psi)} it follows that  
\begin{equation*}
\begin{split}
		&\|\mathcal{M}_\xi[\psi]\|_{\infty,\xi}=\|\mathcal{I}_\xi[E(\psi)]\|_{\infty,\xi}\leqslant C\|E(\psi)\|_{\infty,\xi}\\
		&\qquad\leqslant CC_0\left(C_\star^p\epsilon^{p\gamma_1}+C_\star^2\epsilon^{2\gamma_1}+\epsilon^{\gamma_1}\right)\\
		&\qquad=C_\star\epsilon^{\gamma_1}\left(\frac{CC_0}{C_\star}+C_\star^{p-1}C_0C\epsilon^{(p-1)\gamma_1}+C_\star C_0C\epsilon^{\gamma_1}\right).
\end{split}
\end{equation*}
We take $C_\star>2CC_0$, and 
\begin{equation}\label{epsilon1}
\epsilon<\epsilon_1:=	\begin{cases}
	 \left(2C_0C(C_\star^{p-1}+C_\star)\right)^{-1/\gamma_1}\quad &\text{ if } p\geqslant 2,\\
	 \left(2C_0C(C_\star^{p-1}+C_\star)\right)^{-1/(p-1)\gamma_1} &\text{ if } 1<p<2.
	\end{cases}
\end{equation}
With this choice, we obtain that~$\|\mathcal{M}_\xi[\psi]\|_{\infty,\xi}\le C_\star\epsilon^{\gamma_1}$, which 
entails~\eqref{dvsvds}.

Now, we take $\psi_1$, $\psi_2\in\mathcal{S}$. Recalling that $\bar{u}_\xi\leqslant w_\xi$ in $\mathbb{R}^n$,  and referring again to page 122 in \cite{MR2186962},  one has that 
\begin{equation*}
	\begin{split}
		|E(\psi_1)-E(\psi_2)|&=|(w_\xi+\psi_1-v_\xi)^p-(w_\xi+\psi_2-v_\xi)^p-pw_\xi^{p-1}(\psi_1-\psi_2)|\\
		&\leqslant C_1\left(|\psi_1-v_\xi|^{p-1}+|\psi_2-v_\xi|^{p-1}+|\psi_1-v_\xi|+|\psi_2-v_\xi|\right)|\psi_1-\psi_2|\\
		&\leqslant C_1\left(|\psi_1|^{p-1}+|\psi_2|^{p-1}+|\psi_1|+|\psi_2|+2|v_\xi|^{p-1}+2|v_\xi|\right)|\psi_1-\psi_2|\\
		&\leqslant C_1\left(|\psi_1|^{p-1}+|\psi_2|^{p-1}+|\psi_1|+|\psi_2|+\epsilon^{q(n+2s)}\right)|\psi_1-\psi_2|
	\end{split}
\end{equation*}
for some constant $C_1>0$, where $q=\min\left\{p-1,1\right\}$, thanks to~\eqref{nbfdv}.

{F}rom this and the fact that $\psi_1$, $\psi_2\in\mathcal{S}$,  it follows that
\begin{equation}\label{sdsd}
	\begin{split}&
	\|E(\psi_1)-E(\psi_2)\|_{\infty,\xi}\\&\leqslant C_1\left(\|\psi_1\|^{p-1}_{\infty,\xi}+\|\psi_2\|^{p-1}_{\infty,\xi}+\|\psi_1\|_{\infty,\xi}+\|\psi_2\|_{\infty,\xi}+\epsilon^{q(n+2s)}\right)\|\psi_1-\psi_2\|_{\infty,\xi}\\
	&\leqslant C_1\left(2C_\star^{p-1}\epsilon^{(p-1)\gamma_1}+2C_\star\epsilon^{\gamma_1}+\epsilon^{q(n+2s)}\right)\|\psi_1-\psi_2\|_{\infty,\xi}\\
	&\leqslant C_1\left(2\left(C_\star^{p-1}+C_\star\right)\epsilon^{(p-1)\gamma_1}+\epsilon^{q(n+2s)}\right)\|\psi_1-\psi_2\|_{\infty,\xi}.
	\end{split}
\end{equation}

Now, we set
\begin{equation*}
\epsilon_2:=	\begin{cases}
		\left(2C_1(C_\star^{p-1}+C_\star)+C_1\right)^{-1/q(n+2s)}\quad &\text{ if } p\geqslant 2,\\
		\left(2C_1(C_\star^{p-1}+C_\star)+C_1\right)^{-1/(p-1)\gamma_1} &\text{ if } 1<p<2.
	\end{cases}
\end{equation*}
{F}rom this and~\eqref{epsilon1}, we conclude that if  
\begin{equation}\label{svsfd}
	C_\star>2CC_0\quad \text{and}\quad \epsilon<\min\left\{\epsilon_1,\epsilon_2\right\},
\end{equation}
 then, by~\eqref{sdsd}, 
 \begin{equation*}
 	\|E(\psi_1)-E(\psi_2)\|_{\infty,\xi}<\|\psi_1-\psi_2\|_{\infty,\xi}.
 \end{equation*}
This establishes the desired result in~\eqref{vdfvdf}, thus completing the proof of Theorem~\ref{th:nonlinear problem }.
\end{proof}

Now, in light of Theorem~\ref{th:nonlinear problem },
for any~$\xi\in\Omega_\epsilon$,
we denote that
\begin{equation}\label{dsvsdfx}
\text{$\Psi(\xi)$ is the unique solution to~\eqref{nonlinear projected problem}.}
\end{equation}

\subsection{Reduction to finite dimensions }
This subsection is devoted to finding the solutions to
problem~\eqref{vsdvdscsc} of the form~\eqref{vdsfdsfew}, namely, 
\begin{equation*}
	u=u_\xi=\bar{u}_\xi+\Psi(\xi)
\end{equation*}
where the notation in~\eqref{dsvsdfx} is in use.
 
 We notice that, by combining~\eqref{sdvbsd} with~\eqref{nonlinear projected problem}, the above function $u_\xi$ solves the equation
 \begin{equation}\label{dsfdsdv}
 	-\Delta u_\xi+(-\Delta)^s u_\xi+u_\xi=u_\xi^{p}+\sum_{i=1}^{n}c_iZ_i\qquad \text{in } \Omega_\epsilon.
 \end{equation}
We next aim to find a suitable point $\xi\in\Omega_\epsilon$ such that all the coefficients $c_i$ in~\eqref{dsfdsdv} vanish, which yields a solution to~\eqref{vsdvdscsc} associated with such a point $\xi$.

To this end, we define the functional $J_\epsilon:\Omega_\epsilon\rightarrow \mathbb{R}$ as
\begin{equation}\label{vsdf}
	J_\epsilon(\xi):=I_\epsilon(\bar{u}_\xi+\Psi(\xi))=I_\epsilon(u_\xi)\qquad \text{ for any } \xi\in\Omega_\epsilon,
\end{equation}
where $I_\epsilon$ is given in~\eqref{functional}.  

We will show that the problem of finding a solution to~\eqref{vsdvdscsc} reduces to the one of finding a critical point of the function defined in~\eqref{vsdf}. This is described in the following result.

\begin{Theorem}\label{th: critical thoery}
If $\epsilon>0$ sufficiently small, the coefficients $c_i$ in~\eqref{dsfdsdv} are equal to zero if and only if $\xi$ satisfies 
\begin{equation*}
	\frac{\partial J_\epsilon}{\partial\xi}(\xi)=0.
\end{equation*}
\end{Theorem}

Before proving Theorem~\ref{th: critical thoery}, we need some preliminary lemmata.

\begin{Lemma}\label{lemma vsodv}
		Let $v_\xi$ be as in~\eqref{dbsfbs}. Then, there exists~$C>0$ such that 
		\begin{equation*}
			\left|\frac{\partial v_\xi}{\partial \xi}\right|\leqslant C\epsilon^{n+2s}.
		\end{equation*}
\end{Lemma}

	\begin{proof}
{F}rom~\eqref{vdsvd}, we have that $v_\xi=w_\xi$ outside $\Omega_\epsilon$ and  
\begin{equation*}
		-\Delta v_\xi +(-\Delta)^s v_\xi +v_\xi=0 \qquad  \text{in } \Omega_\epsilon.
\end{equation*}
Therefore, the derivative of $v_\xi$ with respect to $\xi$ satisfies
\begin{equation*}
	-\Delta \frac{\partial v_\xi}{\partial \xi} +(-\Delta)^s \frac{\partial v_\xi}{\partial \xi} +\frac{\partial v_\xi}{\partial \xi}=0 \qquad  \text{in } \Omega_\epsilon\end{equation*}
and $$ \frac{\partial v_\xi}{\partial \xi}=\frac{\partial w_\xi}{\partial \xi} \qquad \text{in } \mathbb{R}^n\setminus \Omega_\epsilon.$$
By Lemma~\ref{lemma decay of w} and~\eqref{fdgd}, we have that, in~$\R^n\setminus\Omega_\epsilon$,
$$
\left|\frac{\partial v_\xi}{\partial \xi}\right|\le C\epsilon^{n+2s} .$$
{F}rom this and the Maximum Principle, we obtain the
desired estimate.
	\end{proof}

\begin{Lemma}\label{equvilant norm}
	Let $\xi\in\mathbb{R}^n$ and $|t|<1$.   Let $\xi_j^t:=\xi+te_j$ for  $j\in\left\{1,\cdots,n\right\}$. 
	
	Then, 
	\begin{equation*}
		2^{-\mu}\|g\|_{\infty,\xi}\leqslant 	\|g\|_{\infty,\xi_j^t}\leqslant 	2^\mu\|g\|_{\infty,\xi}
	,\end{equation*}
	where $\mu $ is as in~\eqref{fiuewfjke768578599}. 
\end{Lemma}

\begin{proof}
	We observe that,  for every $x$, $\xi\in\mathbb{R}^n$ and $|t|<1$,
	\begin{equation*}
		(1+|x-\xi-te_j|)^\mu\leqslant (2+|x-\xi|)^\mu\leqslant 2^\mu(1+|x-\xi|)^\mu
	\end{equation*}
	and 
	\begin{equation*}
		(1+|x-\xi-te_j|)^\mu\geqslant 2^{-\mu}(2+|x-\xi-te_j|)^\mu\geqslant 2^{-\mu}(1+|x-\xi|)^\mu.
	\end{equation*}
	{F}rom the above two formulas, it follows that 
	\begin{equation*}
		2^{-\mu}	\sup\limits_{x\in\mathbb{R}^n}(1+|x-\xi|)^\mu|g(x)|\leqslant 	\sup\limits_{x\in\mathbb{R}^n}(1+|x-\xi-te_j|)^\mu|g(x)|\leqslant 2^\mu	\sup\limits_{x\in\mathbb{R}^n}(1+|x-\xi|)^\mu|g(x)|,
	\end{equation*}
which yields the desired result. 
\end{proof}

\begin{Lemma}\label{vfdsvds}
	Let $\mathcal{I}_\xi $ be defined in~\eqref{dfdsvsd}. Then, the map $\xi\mapsto\mathcal{I}_\xi$ on $\Omega_\epsilon$ is continuously differentiable. 
	
	Moreover, there exists a positive constant $C$ such that 
	\begin{equation}\label{dsvsdcx}
		\left\lVert\frac{\partial \mathcal{I}_\xi[g]}{\partial\xi}\right\rVert_{\infty,\xi}\leqslant C\left(\|g\|_{\infty,\xi}+\left\lVert\frac{\partial g}{\partial\xi}\right\rVert_{\infty,\xi}\right).
	\end{equation}
\end{Lemma}

\begin{proof}
	First we assume that the map $\xi\mapsto\mathcal{I}_\xi$ on $\Omega_\epsilon$ is continuously differentiable and we prove~\eqref{dsvsdcx}. 
	
	Given $\xi\in\Omega_\epsilon$ and $|t|<1$ with $t\neq 0$, we let $\xi_j:=\xi+te_j$, $\psi:=\mathcal{I}_\xi[g]$ and $\psi_{j}^t:=\mathcal{I}_{\xi_j^t}[g_{\xi_j^t}]$.
	 
	 Recalling the definition of $\mathcal{I}_\xi$ in~\eqref{dfdsvsd}, we know that $\psi$ and $\psi_j^t$ solve the following equations
	\begin{equation*}
			-\Delta \psi+(-\Delta)^s\psi+\psi-pw_{\xi}^{p-1} \psi=g+\sum_{i=1}^{n}c_i Z_{i}
	\end{equation*}
	and 
	\begin{equation*}
		-\Delta \psi_j^t+(-\Delta)^s\psi_j^t+\psi_j^t-pw_{\xi_j^t}^{p-1} \psi_j^t=g_{\xi_j^t}+\sum_{i=1}^{n}c_i^{\xi_j^t} Z_{i}^{\xi_j^t}
	\end{equation*}
	where $Z_i=\partial w_\xi/\partial x_i$ and $Z_{i}^{\xi_j^t}=\partial w_{\xi_j^t}/\partial x_i $. 
	
	We define
		\begin{equation*}
		\varphi_j^t:=\frac{\psi_j^t-\psi}{t}\qquad \text{ for any } j=1,\cdots,n
	\end{equation*}
and we observe that $\varphi_j^t$ solves 
\begin{equation}\label{dsfdsfs}
	\begin{split}
		&-\Delta \varphi_j^t+(-\Delta)^s\varphi_j^t+\varphi_j^t-pw_{\xi_j^t}^{p-1} \varphi_j^t\\&\qquad =pD_j^k (w_\xi^{p-1})\psi+D_j^k g+\sum_{i=1}^{n}D_j^k (c_i)Z_{i}^{\xi_j^t}+ \sum_{i=1}^{n}c_iD_j^k(Z_i) \qquad
		\text{ in }\Omega_\epsilon,
	\end{split}
\end{equation}
where $$D_j^k f:=\frac{f(\xi_j^t)-f(\xi)}{t}.$$
Also, $\varphi_j^t\in H^1(\mathbb{R}^n)$ and $\varphi_j^t=0$ outside $\Omega_\epsilon$.

Now, for any $j\in\left\{1,\cdots,n\right\}$, we define 
\begin{equation}\label{dsvsfs}
	\lambda_i:=\tilde{\alpha}^{-1}\int_{\mathbb{R}^n}\varphi_j^tZ_i^{\xi_j^t}\, dx\qquad{\mbox{and}}\qquad
	\bar{\varphi}_j^t:=\varphi_j^t-\sum_{i=1}^{n}\lambda_i\tilde{Z}_i^{\xi_j^t},
\end{equation}
where~$ \tilde{\alpha}$ and~$ \tilde{Z}_i^{\xi_j^t}$ are given in Lemma~\ref{lemma Zi}.

We notice that $\bar{\varphi}_j^t=0 $ in $\mathbb{R}^n\setminus\Omega_\epsilon$ and 
\begin{equation*}
	\int_{\Omega_\epsilon} \bar{\varphi}_j^t {Z}_k^{\xi_j^t} \, dx=0\qquad \text{ for any } k=1,\cdots,n.
\end{equation*}
Moreover, plugging~\eqref{dsvsfs} into~\eqref{dsfdsfs}, one deduces that
\begin{equation*}
		-\Delta \bar{\varphi}_j^t+(-\Delta)^s\bar{\varphi}_j^t+\bar{\varphi}_j^t-pw_{\xi_j^t}^{p-1} \bar{\varphi}_j^t=\bar{f}_j^t+\sum_{i=1}^{n}D_j^k(c_i)Z_{i}^{\xi_j^t}\quad \text{in }\Omega_\epsilon,
\end{equation*}
where 
\begin{equation*}
	\bar{f}_j^t:=\sum_{i=1}^{n}\lambda_i\left(-\Delta +(-\Delta)^s+1-pw_{\xi_j^t}^{p-1} \right)\tilde{Z}_i^{\xi_j^t}+pD_j^k (w_\xi^{p-1})\psi+D_j^k g+ \sum_{i=1}^{n}c_iD_j^k (Z_i).
\end{equation*}

We now claim that 
\begin{equation}\label{sdfsdvdss}
	\|\bar{f}_j^t\|_{\infty,\xi_j^t}\leqslant C\left(\|g\|_{\infty,\xi}+\|D_j^t g\|_{\infty,\xi}\right)
\end{equation}
for some constant $C>0$ independent of $\xi_j^t$.

Indeed, we observe that, from~\eqref{dsvsfs} and the orthogonality condition, it follows that 
\begin{equation*}
\begin{split}&
		\lambda_i=\tilde{\alpha}^{-1}\int_{\mathbb{R}^n}{-\psi}D_j^k(Z_i)\, dx=-\tilde{\alpha}^{-1}t^{-1}\int_{\mathbb{R}^n}\psi\left(\frac{\partial w}{\partial x_i}(x-\xi-te_j)-\frac{\partial w}{\partial x_i}(x-\xi)\right)\, dx\\
		&\qquad=-\tilde{\alpha}^{-1}
		\int_{\mathbb{R}^n}\frac{\psi}{t}\left(
		\int_{0}^{t}\frac{d}{d\tau}\left(\frac{\partial w}{\partial x_i}(x-\xi-\tau e_j)\right)\,d\tau\right)\, dx\\
		&\qquad=-\tilde{\alpha}^{-1}\int_{\mathbb{R}^n}\frac{\psi}{t}
		\left(\int_{0}^{t}\nabla\left(\frac{\partial w}{\partial x_i}\right)(x-\xi-\tau e_j)\cdot e_j\, d\tau\right)\, dx.
\end{split}
\end{equation*}
Thus, owing to Lemma~\ref{lemma gradient decay},  
\begin{equation*}
\begin{split}
	&|\lambda_i|\leqslant \tilde{\alpha}^{-1}\|\psi\|_{\infty,\xi}\int_{\mathbb{R}^n}\frac{1}{t}\left(
	\int_{0}^{t}(1+|x-\xi-\tau e_j|)^{-(n+2s)}\, d\tau\right)\, dx\\
	&\qquad\leqslant 2^{n+2s}\tilde{\alpha}^{-1}\|\psi\|_{\infty,\xi}\int_{\mathbb{R}^n}(1+|x-\xi|)^{-(n+2s)}\, dx\leqslant C\|\psi\|_{\infty,\xi}.
\end{split}
\end{equation*}

Similarly, using Lemma~\ref{lemma decay of w}, we infer that
\begin{equation*}
	\frac{1}{t}\big|w_{\xi_j^t}^{p-1}-w_{\xi}^{p-1}\big|\leqslant \frac{p-1}{t}\int_{0}^{t}w^{p-2}(x-\xi-\tau e_j)|\nabla w(x-\xi-\tau e_j)|\, d\tau
	\leqslant C.
\end{equation*}
Moreover, by combining Lemmata~\ref{lemma decay of w},~\ref{lemma gradient decay} and~\ref{lemma second gradient decay},
and recalling also that~$ w_{\xi_j^t}$ is bounded,
\begin{equation*}
	\sup\limits_{x\in\mathbb{R}^n}(1+|x-\xi-te_j|)^\mu\big|-\Delta \tilde{Z}_i^{\xi_j^t} +(-\Delta)^s \tilde{Z}_i^{\xi_j^t}+\tilde{Z}_i^{\xi_j^t}-pw_{\xi_j^t}^{p-1}\tilde{Z}_i^{\xi_j^t} \big|\leqslant C.
\end{equation*}
Therefore, from the above formulas and Lemma~\ref{equvilant norm}, we conclude that 
\begin{equation*}
	\|\bar{f}_j^t\|_{\infty,\xi_j^t}\leqslant C\left(\|\psi\|_{\infty,\xi}+\|D_j^t g\|_{\infty,\xi_j^t}+\sum_{i=1}^{n}|c_i|\right)\leqslant C
	\left(\|\psi\|_{\infty,\xi}+\|D_j^t g\|_{\infty,\xi}+\sum_{i=1}^{n}|c_i|\right)
\end{equation*}
up to renaming $C>0$.  

Recalling the estimates in~\eqref{dgerf} and~\eqref{ndfgfd}
and utilizing Lemma~\ref{lemma cdscsd}, we thereby find that
\begin{equation*}
	\begin{split}
		&\|\bar{f}_j^t\|_{\infty,\xi_j^t}\leqslant C\left(\|\psi\|_{\infty,\xi}+\|D_j^t g\|_{\infty,\xi_j^t}+\|\psi\|_{L^2(\mathbb{R}^n)}+\|g\|_{\infty,\xi}\right)\\
		&\qquad\leqslant C\left(\|\psi\|_{\infty,\xi}+\|D_j^t g\|_{\infty,\xi}+\| g\|_{\infty,\xi}\right)\leqslant C\left(\|g\|_{\infty,\xi}+\|D_j^t g\|_{\infty,\xi}\right),
	\end{split}
\end{equation*}
which leads to~\eqref{sdfsdvdss}.

As a consequence of~\eqref{sdfsdvdss}, using again Lemma~\ref{lemma cdscsd}, \begin{equation}\label{fdgdf}
	\|\bar{\varphi}_j^t\|_{\infty,\xi_j^t}\leqslant C\|\bar{f}_j^t\|_{\infty,\xi_j^t}\leqslant C\left(\|g\|_{\infty,\xi}+\|D_j^t g\|_{\infty,\xi}\right),
\end{equation}
for some constant $C>0$ independent of $\xi_j^t$.

In virtue of Lemmata~\ref{lemma cdscsd} and~\ref{equvilant norm}, by combining~\eqref{dsvsfs} with~\eqref{fdgdf}, we obtain that 
\begin{equation*}
		\|{\varphi}_j^t\|_{\infty,\xi}\leqslant 2^\mu \|{\varphi}_j^t\|_{\infty,\xi_j^t}\leqslant 2^\mu \left(	\|\bar{\varphi}_j^t\|_{\infty,\xi_j^t}+\sum_{i=1}^{n}|\lambda_i|\|\tilde{Z}_i^{\xi_j^t}\|_{\infty,\xi_j^t}\right)\leqslant C\left(\|g\|_{\infty,\xi}+\|D_j^t g\|_{\infty,\xi}\right),
\end{equation*}
for some constant $C>0$ independent of $t$. 

Hence,  taking the limit as~$t\rightarrow 0$, we obtain that
\begin{equation*}
	\left\lVert\frac{\partial \psi}{\partial\xi}\right\rVert_{\infty,\xi}\leqslant C\left(\|g\|_{\infty,\xi}+\left\lVert\frac{\partial g}{\partial\xi}\right\rVert_{\infty,\xi}\right),
\end{equation*} 
which proves the desired result in~\eqref{dsvsdcx}.

By combining the previous computation and the Implicit Function Theorem, a standard argument shows that $\xi\rightarrow \mathcal{I}_\xi$ is continuously differentiable (see e.g. \cite{MR3121716} below formula~(4.20)).
\end{proof}

\begin{Lemma}\label{lemma hvisdvs}
Let $\psi\in\Psi$ be a solution to~\eqref{nonlinear projected problem} with $\|\psi\|_{\infty,\xi}<C_*\epsilon^{\gamma_1}$, where~$\gamma_1$ is
as in~\eqref{hewgamma068jkagamma}.

Then, there exists a positive constant $C>0$ such that 
\begin{equation*}
	\left\lVert\frac{\partial \psi}{\partial\xi}\right\rVert_{\infty,\xi}\leqslant C\epsilon^{\gamma},
\end{equation*}
where~$q:=\min\left\{1,p-1\right\}$ and
$$\gamma:=\min\big\{{n+2s-\mu(2-p)_+},q\gamma_1\big\}>0.$$
\end{Lemma}

\begin{proof}
We notice that, in the light of Lemma~\ref{vfdsvds} (applied here with $g:=E(\psi)$)
and Lemma~\ref{lemma:  E(psi)},
\begin{equation}\label{vdsfs}
	\left\lVert\frac{\partial \psi}{\partial\xi}\right\rVert_{\infty,\xi}\leqslant C\left(\|E(\psi)\|_{\infty,\xi}+\left\lVert\frac{\partial E(\psi)}{\partial\xi}\right\rVert_{\infty,\xi}\right)\leqslant C\left(\epsilon^{\gamma_1}+\left\lVert\frac{\partial E(\psi)}{\partial\xi}\right\rVert_{\infty,\xi}\right),
\end{equation}
up to renaming $C>0$.
 
We now claim that
 \begin{equation}\label{dfvdfb}
 	\left\lVert\frac{\partial E(\psi)}{\partial\xi}\right\rVert_{\infty,\xi}\leqslant C\left(\epsilon^{q\gamma_1}\left\lVert\frac{\partial \psi}{\partial\xi}\right\rVert_{\infty,\xi}+\epsilon^\gamma\right).
 \end{equation}
Indeed, from~\eqref{sdvds} it follows that 
 \begin{equation*}
	\begin{split}
	\frac{\partial E(\psi)}{\partial\xi}&=p(\bar{u}_\xi+\psi)^{p-1}\left(\frac{\partial \bar{u}_\xi}{\partial \xi}+\frac{\partial \psi}{\partial \xi}\right)-pw_\xi^{p-1}\left(\frac{\partial w_\xi}{\partial \xi}+\frac{\partial \psi}{\partial \xi}\right) -p(p-1)w_\xi^{p-2}\frac{\partial w_\xi}{\partial \xi}\psi\\
	&=p(\bar{u}_\xi+\psi)^{p-1}\left(\frac{\partial \bar{u}_\xi}{\partial \xi}+\frac{\partial \psi}{\partial \xi}\right)-p(w_\xi+\psi)^{p-1}\left(\frac{\partial w_\xi}{\partial \xi}+\frac{\partial \psi}{\partial \xi}\right)\\
	&\qquad +p[(w_\xi+\psi)^{p-1}-w_\xi^{p-1}]\left(\frac{\partial w_\xi}{\partial \xi}+\frac{\partial \psi}{\partial \xi}\right) -p(p-1)w_\xi^{p-2}\frac{\partial w_\xi}{\partial \xi}\psi
	\\
	&=p\left(\frac{\partial \psi}{\partial \xi}+\frac{\partial w_\xi}{\partial \xi}\right)[(\bar{u}_\xi+\psi)^{p-1}-(w_\xi+\psi)^{p-1}]+p(\bar{u}_\xi+\psi)^{p-1}\left(\frac{\partial \bar{u}_\xi}{\partial \xi}-\frac{\partial w_\xi}{\partial \xi}\right)\\
	&\qquad +p[(w_\xi+\psi)^{p-1}-w_\xi^{p-1}]\left(\frac{\partial w_\xi}{\partial \xi}+\frac{\partial \psi}{\partial \xi}\right) -p(p-1)w_\xi^{p-2}\frac{\partial w_\xi}{\partial \xi}\psi.
	\end{split}
\end{equation*}
Thus, recalling the definition of $v_\xi$  in~\eqref{dbsfbs}, and
 referring to page 122 in \cite{MR2186962}, 
\begin{equation*}
	\begin{split}
	\left|\frac{\partial E(\psi)}{\partial\xi}\right|&\leqslant C_{p,\|w_\xi\|_{L^\infty(\mathbb{R}^n)}} [|v_\xi|^{p-1}+|v_\xi|]\left(\left|\frac{\partial w_\xi}{\partial \xi}\right|+\left|\frac{\partial \psi}{\partial \xi}\right|\right)+c_p\left(|\bar{u}_\xi|^{p-1}+|\psi|^{p-1}\right)\left|\frac{\partial v_\xi}{\partial \xi}\right|\\
	&\qquad +C_{p,\|w_\xi\|_{L^\infty(\mathbb{R}^n)}} [|\psi|^{p-1}+|\psi|]\left(\left|\frac{\partial w_\xi}{\partial \xi}\right|+\left|\frac{\partial \psi}{\partial \xi}\right|\right)+p(p-1)\left|w_\xi^{p-2}\frac{\partial w_\xi}{\partial \xi}\right|\, |\psi|.
	\end{split}
\end{equation*}
 {F}rom the decay of $w_\xi$ and Lemma~\ref{lemma decay of w}, we get
 \begin{equation*}
 	\left|w_\xi^{p-2}\frac{\partial w_\xi}{\partial \xi}\right|\leqslant C(1+|x-\xi|)^{-(p-1)(n+2s)}\leqslant C.
 \end{equation*}

Moreover, since $\bar{u}_\xi=0$ and $\psi=0$ outside $\Omega_\epsilon$,  by taking into account that~$\mu<n+2s$ and~$w_\xi>\bar{u}_\xi$ in~$\mathbb{R}^n$,  
and owing to Lemmata~\ref{lemma decay of w} and~\ref{lemma vsodv}, we see that
\begin{equation*}
	\begin{split}
		&\sup\limits_{x\in\mathbb{R}^n}(1+|x-\xi|)^\mu\left(|\bar{u}_\xi(x)|^{p-1}+|\psi(x)|^{p-1}\right)\left|\frac{\partial v_\xi}{\partial \xi}\right|\\
		&\qquad =
		\sup\limits_{x\in\Omega_\epsilon}(1+|x-\xi|)^{\mu(p-1)}\left(|\bar{u}_\xi(x)|^{p-1}+|\psi(x)|^{p-1}\right)(1+|x-\xi|)^{\mu(2-p)}\left|\frac{\partial v_\xi}{\partial \xi}\right|\\
		&\qquad \leqslant C\sup\limits_{x\in\Omega_\epsilon}(1+|x-\xi|)^{\mu(2-p)}\left(\|w_\xi\|^{p-1}_{\infty,\xi}+\|\psi\|^{p-1}_{\infty,\xi}\right)	\epsilon^{n+2s}\\
		&\qquad  \leqslant C\left(1+\epsilon^{(p-1)\gamma_1}\right)\epsilon^{n+2s-\mu(2-p)_+}
	\end{split}
\end{equation*}
up to renaming $C>0$, where $(2-p)_+:=\max\left\{0,2-p\right\}<1$.

Hence, from the above two formulas, 
 recalling~\eqref{nbfdv},  by combining Lemmata~\ref{lemma decay of w} and~\ref{lemma vsodv}, one concludes that
 \begin{equation*}
 	\begin{split}
 			\left\lVert\frac{\partial E(\psi)}{\partial\xi}\right\rVert_{\infty,\xi}&\leqslant C\left(\epsilon^{q\gamma_1}\left(\left\lVert\frac{\partial \psi}{\partial\xi}\right\rVert_{\infty,\xi}+1\right)+\epsilon^{n+2s-\mu(2-p)_+}+\epsilon^{\gamma_1}\right)\leqslant C\left(\epsilon^{q\gamma_1}\left\lVert\frac{\partial \psi}{\partial\xi}\right\rVert_{\infty,\xi}+\epsilon^\gamma\right).
 	\end{split}
 \end{equation*}
This yields the claim in~\eqref{dfvdfb}.

As a consequence of~\eqref{vdsfs} and~\eqref{dfvdfb}, we obtain that 
\begin{equation*}
	\left\lVert\frac{\partial \psi}{\partial\xi}\right\rVert_{\infty,\xi}\leqslant C\left(\epsilon^{\gamma}+\epsilon^{q\gamma_1}\left\lVert\frac{\partial \psi}{\partial\xi}\right\rVert_{\infty,\xi}\right).
\end{equation*}
Now, taking $\epsilon >0$ sufficiently small, by reabsorbing one term into the left-hand side, we obtain the desired result.
\end{proof}


With this preparatory work, we now complete the proof of Theorem~\ref{th: critical thoery}.

\begin{proof}[Proof of Theorem~\ref{th: critical thoery}]
Recall that, from Theorem~\ref{th:nonlinear problem }, 
$	\Psi(\xi)$  is the unique solution to~\eqref{nonlinear projected problem} with $\|\Psi(\xi)\|_{\infty,\xi}\leqslant C\epsilon^{\gamma_1}$ and $u_\xi=\bar{u}_\xi+\Psi(\xi)$. 

We denote by~$\xi:=(\xi_1,\cdots,\xi_n)$ and, for any $j\in\left\{1,\cdots,n\right\}$, we
 take the derivative of $u_\xi$ with respect to $\xi_j$, finding that
\begin{equation*}
	\frac{\partial u_\xi}{\partial \xi_j}=\frac{\partial \bar{u}_\xi}{\partial \xi_j}+\frac{\partial \Psi(\xi)}{\partial \xi_j}.
\end{equation*}
 {F}rom Lemma~\ref{lemma vsodv}, it follows that
 \begin{equation*}
 	\left|	\frac{\partial w_\xi}{\partial \xi_j}-	\frac{\partial \bar{u}_\xi}{\partial \xi_j}\right|=\left|	\frac{\partial v_\xi}{\partial \xi_j}\right|\leqslant C\epsilon^{n+2s}.
 \end{equation*}
Also, Lemma~\ref{lemma hvisdvs} entails that
 \begin{equation*}
 \left|	\frac{\partial \Psi(\xi)}{\partial \xi_j}\right|\leqslant C\epsilon^{\gamma}.
 \end{equation*}
Hence, putting the above three formulas together, we have that
 \begin{equation}\label{nk}
 		\frac{\partial u_\xi}{\partial \xi_j}=\frac{\partial w_\xi}{\partial \xi_j}+O(\epsilon^\gamma)=-\frac{\partial w_\xi}{\partial x_j}+O(\epsilon^\gamma)=-Z_j+O(\epsilon^\gamma).
 \end{equation} 

We now multiply the equation in~\eqref{dsfdsdv} by ${\partial u_\xi}/{\partial \xi_j} $, finding that
 \begin{equation*}
	\left(-\Delta u_\xi+(-\Delta)^s u_\xi+u_\xi-u_\xi^{p}\right)	\frac{\partial u_\xi}{\partial \xi_j}=\sum_{i=1}^{n}c_iZ_i \frac{\partial u_\xi}{\partial \xi_j}\qquad \text{in } \Omega_\epsilon.
\end{equation*}
Thus, owing to Lemma~\ref{lemma decay of w} and~\eqref{nk}, we deduce that
  \begin{equation*}
 	\left|\left(-\Delta u_\xi+(-\Delta)^s u_\xi+u_\xi-u_\xi^{p}\right)	\frac{\partial u_\xi}{\partial \xi_j}\right|\leqslant \sum_{i=1}^{n}|c_i|\, |Z_i|\,\left|\frac{\partial u_\xi}{\partial \xi_j}\right|\leqslant C
 \end{equation*}
for some constant $C>0$ independent of $\xi$. 

This gives that the functions $\big(-\Delta u_\xi+(-\Delta)^s u_\xi+u_\xi-u_\xi^{p} \big)	{\partial u_\xi}/{\partial \xi_j} $ are in $L^1(\Omega_\epsilon)$ uniformly with respect to $\xi$. Thus, we are in the position of computing the derivative of~$J_\epsilon$ with respect to~$\xi_j$, obtaining that
\begin{equation}\label{sdfe}
	\begin{split}
	\frac{\partial J_\epsilon}{\partial \xi_j}(\xi)&=\frac{\partial I_\epsilon(u_\xi)}{\partial \xi_j}\\&=\frac{\partial}{\partial \xi_j}\left(\int_{\Omega_\epsilon}\frac{1}{2}\left(u_\xi(-\Delta) u_\xi+u_\xi(-\Delta)^s u_\xi+u_\xi^2\right)-\frac{1}{p+1}u_\xi^{p+1}(x)\, dx\right)\\
&= \int_{\Omega_\epsilon}\frac{1}{2}\frac{\partial u_\xi}{\partial \xi_j}
\big((-\Delta) +(-\Delta)^s \big)u_\xi+\frac{1}{2}u_\xi\big((-\Delta) +(-\Delta)^s \big)
\frac{\partial u_\xi}{\partial \xi_j}
+u_\xi\frac{\partial u_\xi}{\partial \xi_j}
-u_\xi^{p}(x)\frac{\partial u_\xi}{\partial \xi_j}\, dx\\
	&=\int_{\Omega_\epsilon}\frac{\partial u_\xi}{\partial \xi_j}\left(-\Delta u_\xi +(-\Delta)^su_\xi+u_\xi-u_\xi^{p}(x) \right)\, dx\\
&
	=\sum_{i=1}^{n}c_iM_{ji}
	\end{split}
\end{equation}
where $$ M_{ji}:=\int_{\Omega_\epsilon}Z_i \frac{\partial u_\xi}{\partial \xi_j}\, dx.$$

Now, utilizing~\eqref{nk} and the decay of $Z_i$ in Lemma~\ref{lemma comparable for h}, one has that
\begin{equation*}
	M_{ji}=\int_{\Omega_\epsilon}-Z_iZ_j\, dx+O(\epsilon^\gamma)=-\alpha\delta_{ij}+O(\epsilon^\gamma)
\end{equation*}
where $\alpha>0$ was introduced in Lemma~\ref{lemma Zi}. 
Therefore, the matrix $-\alpha^{-1}M_{ji}$ is a perturbation of the identity and thus it is invertible for $\epsilon$ sufficiently small. Accordingly,
$M_{ji}$ is also invertible. 

As a consequence of this fact and~\eqref{sdfe}, we conclude that 
\begin{equation*}
	\frac{\partial J_\epsilon}{\partial \xi}(\xi)=\left(\frac{\partial J_\epsilon}{\partial \xi_1}(\xi), \cdots,\frac{\partial J_\epsilon}{\partial \xi_n}(\xi)\right)=0 \quad \text{ if and only if }\quad c:=\left(c_1,\cdots,c_n\right)=0,
\end{equation*} which is the claim of
of Theorem~\ref{th: critical thoery}.
\end{proof}

Based on Theorem~\ref{th: critical thoery}, we give an expansion for the function~$J_\epsilon$, as follows.

\begin{Theorem}\label{th expansion}
	We have the following expansion
	\begin{equation*}
		J_\epsilon(\xi)=I_\epsilon(\bar{u}_\xi)+o(\epsilon^{n+4s}).
	\end{equation*}
\end{Theorem}

\begin{proof}	
	{F}rom the definition of $J_\epsilon$, we know that
	\begin{equation*}
			J_\epsilon(\xi)=I_\epsilon(\bar{u}_\xi+\Psi(\xi)).
	\end{equation*}
Hence, 
\begin{equation*}
	\begin{split}
			J_\epsilon(\xi)&=I_\epsilon(\bar{u}_\epsilon)+I'_\epsilon(\bar{u}_\epsilon)[\Psi(\xi)]+I''_\epsilon(\bar{u}_\epsilon)[\Psi(\xi),\Psi(\xi)]+O([\Psi(\xi)]^3)\\
			&=I_\epsilon(\bar{u}_\epsilon)+\int_{\Omega_\epsilon}\left((-\Delta )\bar{u}_\xi \Psi(\xi) +(-\Delta)^s\bar{u}_\xi \Psi(\xi)+\bar{u}_\xi\Psi(\xi)-\bar{u}_\xi^{p}\Psi(\xi) \right)\, dx\\
			&\;+\int_{\Omega_\epsilon}\left((-\Delta) \Psi(\xi) \Psi(\xi) +(-\Delta)^s\Psi(\xi)\Psi(\xi)+\Psi(\xi)^2-p\bar{u}_\xi^{p-1}\Psi(\xi)^2 \right)\, dx+O([\Psi(\xi)]^3).
	\end{split}
\end{equation*}
Using the fact that $u_\xi=\bar{u}_\xi+\Psi(\xi)$, one has that 
\begin{equation}\label{fdas}
		\begin{split}
		J_\epsilon(\xi)&=I_\epsilon(\bar{u}_\epsilon)+\int_{\Omega_\epsilon}\left(-\Delta u_\xi +(-\Delta)^su_\xi+u_\xi-u_\xi^{p} \right)\Psi(\xi)\, dx\\
		&\qquad \qquad +\int_{\Omega_\epsilon} \left(u_\xi^p-\bar{u}_\xi^p-p\bar{u}_\xi^{p-1}\Psi(\xi)\right)\Psi(\xi)\, dx +O([\Psi(\xi)]^3).
		\end{split}
\end{equation}
Moreover, by combining~\eqref{dsfdsdv} and the orthogonality condition, one has that
\begin{equation*}
	\int_{\Omega_\epsilon}\left(-\Delta u_\xi +(-\Delta)^su_\xi+u_\xi-u_\xi^{p} \right)\Psi(\xi)\, dx=0.
\end{equation*}
Thus, we can rewrite~\eqref{fdas} as
\begin{equation}\label{dsfdsc}
	J_\epsilon(\xi)	=I_\epsilon(\bar{u}_\epsilon)+\int_{\Omega_\epsilon} \left(u_\xi^p-\bar{u}_\xi^p-p\bar{u}_\xi^{p-1}\Psi(\xi)\right)\Psi(\xi)\, dx +O([\Psi(\xi)]^3).
\end{equation}

We recall that $\bar{u}_\xi<w_\xi$, and, referring to page~122 in~\cite{MR2186962}, we conclude that 
\[  \big|u_\xi^p-\bar{u}_\xi^p-p\bar{u}_\xi^{p-1}\Psi(\xi)\big|\leqslant c\left(|\Psi(\xi)|^2+|\Psi(\xi)|^p\right),\]
where the positive constant $c$ only depends on $p$ and $\|w_\xi\|_{L^\infty(\mathbb{R}^n)}$.

Accordingly, since $\mu>n/2$,
\begin{equation}\label{sdfa}
	\begin{split}
		&\left|\int_{\Omega_\epsilon} \left(u_\xi^p-\bar{u}_\xi^p-p\bar{u}_\xi^{p-1}\Psi(\xi)\right)\Psi(\xi)\, dx\right|\leqslant c \int_{\Omega_\epsilon}\left(|\Psi(\xi)|^3+|\Psi(\xi)|^{p+1}\right)\, dx\\
		&\qquad\leqslant c\left(\|\Psi(\xi)\|^3_{\infty,\xi}\int_{\Omega_\epsilon}(1+|x-\xi|)^{-3\mu}\, dx+ \|\Psi(\xi)\|^{p+1}_{\infty,\xi}\int_{\Omega_\epsilon}(1+|x-\xi|)^{-\mu(p+1)}\, dx\right)\\
		&\qquad \leqslant c\epsilon^{2\gamma_1}
	\end{split}
\end{equation}
up to renaming $c>0$.

Now, on the one hand we notice that, 
for any $p\geqslant 2$ and $\mu<n+2s$,
$$2\gamma_1=\min\left\{2n+4s,2p(n+2s)-2\mu\right\}>n+4s .$$ On the other hand, for each $1<p<2$, since~$2\mu<2p(n+2s)-{n}-4s$, 
$$ 2\gamma_1=\min\left\{2n+4s,2p(n+2s)-2\mu\right\}>n+4s.$$

As a consequence of this, by combining~\eqref{dsfdsc} and~\eqref{sdfa}, we obtain the desired claim in Theorem~\ref{th expansion}.
\end{proof}

\section{Proof of Theorem~\ref{th main theorem}}\label{sec:Proof of Theorem 1.1}
 In this section, we complete the proof of Theorem~\ref{th main theorem}. To this end, we observe that, from Theorems~\ref{th energy estimates} and~\ref{th expansion}, it follows that, for any  $\xi\in\Omega_\epsilon$  with $d:=$ dist$(\xi,\partial\Omega_\epsilon)\geqslant \delta/\epsilon>2 $ (for some  $\delta\in(0,1)$), 
 	\begin{equation}\label{dbfd}
 		J_\epsilon(\xi)=I(w)+\frac{1}{2}\mathcal{H}_\epsilon(\xi)+o(\epsilon^{n+4s})
 	\end{equation}
where $J_\epsilon$ and $I$ are as defined in \eqref{vsdf} and~\eqref{dsvsbs}, respectively, and $\mathcal{H}_\epsilon(\xi)$ is as introduced in~\eqref{bjbca}.
 
Moreover, we recall the definition of the set~$\Omega_{\epsilon,\delta}$ given in~\eqref{ngdjs}. Since $J_\epsilon$ is a continuous functional, it admits a minimum~$\xi_0\in\overline{\Omega}_{\epsilon,\delta}$.
 
We now show that 
\begin{equation}\label{weknow78gdshoe}\xi_0\in\Omega_{\epsilon,\delta}.
\end{equation}
Indeed, suppose by contradiction that $\xi_0\in\partial \Omega_{\epsilon,\delta}$. Then, by~\eqref{dbfd},
\begin{equation}\label{gdfgfd}
	J_\epsilon(\xi_0)=I(w)+\frac{1}{2}\mathcal{H}_\epsilon(\xi_0)+o(\epsilon^{n+4s})\geqslant I(w)+\frac{1}{2}\min\limits_{\partial\Omega_{\epsilon,\delta}}\mathcal{H}_\epsilon+o(\epsilon^{n+4s}).
\end{equation} 

Additionally, from Lemma~\ref{lemma compact set }, we know that~$\mathcal{H}_\epsilon$ attains an interior minimum in~$\Omega_{\epsilon,\delta}$. More precisely, there exists~$\bar{\xi}\in \Omega_{\epsilon,\delta}$ such that 
\begin{equation}\label{sfsdg}
\mathcal{H}_\epsilon(\bar{\xi})=	\min\limits_{\Omega_{\epsilon,\delta}}\mathcal{H}_\epsilon\leqslant c_1\epsilon^{n+4s}< c_2\left(\frac{\epsilon}{\delta}\right)^{n+4s}\leqslant \min\limits_{\partial\Omega_{\epsilon,\delta}}\mathcal{H}_\epsilon,
\end{equation}
for suitable constants $c_1$, $c_2>0$.

Furthermore, the minimality of $\xi_0$ gives that
\begin{equation}\label{gdfgf}
	J_\epsilon(\xi_0)\leqslant J_\epsilon(\bar\xi)= I(w)+\frac{1}{2}\mathcal{H}_\epsilon(\bar\xi)+o(\epsilon^{n+4s})
\end{equation} 
Thus, by combining~\eqref{sfsdg} with~\eqref{gdfgf}, one evinces that
\begin{equation*}
	J_\epsilon(\xi_0)-I(w)\leqslant \frac{1}{2}\mathcal{H}_\epsilon(\bar\xi)+o(\epsilon^{n+4s})\leqslant \frac{c_1}{2}\epsilon^{n+4s}+o(\epsilon^{n+4s}).
\end{equation*}

Moreover, utilizing~\eqref{gdfgfd} and~\eqref{sfsdg}, we see that
\begin{equation*}
	J_\epsilon(\xi_0)-I(w)\geqslant \frac{1}{2}\min\limits_{\partial\Omega_{\epsilon,\delta}}\mathcal{H}_\epsilon+o(\epsilon^{n+4s})\geqslant \frac{c_2}{2\delta^{n+4s}}\epsilon^{n+4s}+o(\epsilon^{n+4s}).
\end{equation*}
{F}rom the above two formulas, one concludes that
\begin{equation*}
	\frac{c_2}{2\delta^{n+4s}}\epsilon^{n+4s}+o(\epsilon^{n+4s})\leqslant \frac{c_1}{2}\epsilon^{n+4s}+o(\epsilon^{n+4s}).
\end{equation*}
Hence, after a division by $\epsilon^{n+4s}$, taking the limit as~$\epsilon\rightarrow 0$, we obtain that
\begin{equation*}
	\frac{c_2}{2\delta^{n+4s}}\leqslant \frac{c_1}{2},
\end{equation*}
which is a contradiction for $\delta$ sufficiently small. The claim in~\eqref{weknow78gdshoe} is thereby established.

Thanks to~\eqref{weknow78gdshoe}, we have that 
\begin{equation*}
	\frac{\partial J_\epsilon}{\partial \xi}(\xi_0)=0.
\end{equation*}

As a consequence of this and Theorem~\ref{th: critical thoery}, we obtain the existence of a solution to~\eqref{vfveffd}, which satisfies~\eqref{fvvsf} for $\epsilon$ sufficiently small. The proof of Theorem~\ref{th main theorem} is complete.

\begin{appendix}

\section{Uniform $L^\infty$ estimates}\label{sec:vfdvdfv}

\begin{Lemma}\label{lemma regularity boundedness}
	Let $g\in L^2(\mathbb{R}^n)\cap L^\infty(\mathbb{R}^n)$ and~$\psi\in H^1(\mathbb{R}^n)$ be a solution of~\eqref{vsdjvb}.
	
	Then, there exists a positive constant $C$,
	only depending on $n$ and~$s$, such that 
	\begin{equation}\label{dvdss}
		\|\psi\|_{L^\infty(\mathbb{R}^n)}\leqslant C\left(\|g\|_{L^\infty(\mathbb{R}^n)}+\|\psi\|_{L^2(\mathbb{R}^n)}\right).
	\end{equation}
\end{Lemma}

\begin{proof}
		Let $\delta>0$ to be conveniently chosen later on (see formula~\eqref{dscsdc}  below). We define, for every~$x\in \mathbb{R}^n$,
	\begin{equation}\label{dvscsd}
		\tilde{\psi}(x):=\frac{\delta \psi(x)}{\|\psi\|_{L^2(\mathbb{R}^n)}+\|g\|_{L^\infty(\mathbb{R}^n)}}\qquad \text{and}\qquad \tilde{g}(x):=\frac{\delta g(x)}{\|\psi\|_{L^2(\mathbb{R}^n)}+\|g\|_{L^\infty(\mathbb{R}^n)}} .
	\end{equation}
	In this way, we have that 
	\begin{equation}\label{v jds sd}
		\begin{cases}
			-\Delta \tilde\psi+ (-\Delta )^s\tilde\psi+\tilde\psi =\tilde{g}\qquad &\text{in } \Omega_\epsilon,\\
			\tilde\psi=0 &\text{in } \mathbb{R}^n\setminus\Omega_\epsilon.
		\end{cases}
	\end{equation}
	To prove~\eqref{dvdss}, it suffices to show that 
	\begin{equation}\label{vsvsdv}
		\|\tilde\psi\|_{L^\infty(\mathbb{R}^n)}\leqslant 1
	\end{equation}since this would give that 
	$$\|\psi\|_{L^\infty(\mathbb{R}^n)}\leqslant \frac{\|g\|_{L^\infty(\mathbb{R}^n)}+\|\psi\|_{L^2(\mathbb{R}^n)}}{\delta},$$
	which is the desired estimate.
	
Hence, we now focus on the proof of~\eqref{vsvsdv}. For any integer $k\in\mathbb{N}$, we set $$v_k:=\tilde\psi-(1-2^{-k}),\quad w_k:={(v_k)}_+:=\max\left\{v_k, 0\right\},\quad U_k:=\|w_k\|^2_{L^2(\mathbb{R}^n)}.$$
	It is immediate to check that 
	\begin{itemize}
		\item[(i)] $w_k\in H^1(\mathbb{R}^n)$, $w_k\geqslant 0$ in $\mathbb{R}^n$ and $w_k=0$ in $\mathbb{R}^n\setminus \Omega_\epsilon$.
		\item[(ii)] $w_k>\frac{1}{2^{k+1}}$ and $\tilde\psi<2^{k+1}w_k$ if $w_{k+1}>0$.
	\end{itemize} 
	Moreover, plugging $w_{k+1}$ as a test function in~\eqref{v jds sd}, one has that
	\begin{equation}\label{vfdvds}
		\begin{split}
			\int_{\mathbb{R}^n} \nabla \tilde{\psi}\cdot \nabla w_{k+1}+\int_{\mathbb{R}^n}\int_{\mathbb{R}^n}\frac{(\tilde{\psi}(x)-\tilde{\psi}(y))(w_{k+1}(x)-w_{k+1}(y))}{|x-y|^{n+2s}}+\int_{\mathbb{R}^n}\tilde{\psi}w_{k+1}=\int_{\mathbb{R}^n} \tilde{g}w_{k+1}.
		\end{split}
	\end{equation}
	
	We observe that, since $\nabla\tilde{\psi}(x)=\nabla w_{k+1}(x) $
	in~$\{w_{k+1}>0\}$,  
	\begin{equation}\label{vfdvd}
		\int_{\mathbb{R}^n} \nabla \tilde{\psi}\cdot \nabla w_{k+1}=\int_{\mathbb{R}^n}|\nabla w_{k+1}|^2>0.
	\end{equation}
	Furthermore,  by taking into account the fact that $$|w_k(x)-w_k(y)|^2\leqslant(w_k(x)-w_k(y))(v_k(x)-v_k(y)) ,$$ we obtain that 
	\begin{equation}\label{vsdvd}
		\int_{\mathbb{R}^n}\int_{\mathbb{R}^n}\frac{|w_{k+1}(x)-w_{k+1}(y)|^2}{|x-y|^{n+2s}}\, dx\,dy	\leqslant\int_{\mathbb{R}^n}\int_{\mathbb{R}^n}\frac{(\tilde{\psi}(x)-\tilde{\psi}(y))(w_{k+1}(x)-w_{k+1}(y))}{|x-y|^{n+2s}}\, dx\,dy.
	\end{equation}
	Thus, by combining~\eqref{vfdvds},~\eqref{vfdvd} and~\eqref{vsdvd}, we see that 
	\begin{equation}\label{fdgdfgdf}
		\begin{split}&
			[w_{k+1}]_s^2:=	\int_{\mathbb{R}^n}\int_{\mathbb{R}^n}\frac{|w_{k+1}(x)-w_{k+1}(y)|^2}{|x-y|^{n+2s}}\, dx\,dy\leqslant \int_{\mathbb{R}^n}(|\tilde{\psi}|+\tilde{g})w_{k+1}\, dx\\
			&\qquad \leqslant \int_{\left\{w_{k+1}>0\right\}}(2^{k+1}w_k+\delta)w_{k+1}\, dx\leqslant \int_{\left\{w_{k+1}>0\right\}}2^{k+1}w^2_{k}\, dx+\delta \int_{\left\{w_{k+1}>0\right\}}w_{k+1}\, dx\\
			&\qquad\leqslant 2^{k+1}U_k+\delta U_k^{\frac{1}{2}}\left|\left\{w_{k+1}>0\right\}\right|^{\frac{1}{2}}\leqslant 2^{k+1}U_k+\delta U_k^{\frac{1}{2}}\left|\left\{w_{k}>\frac{1}{2^{k+1}}\right\}\right|^{\frac{1}{2}}\\
			&\qquad
			\leqslant 2^{k+1}U_k(1+\delta).
		\end{split}
	\end{equation}
	
	Moreover, from \cite[Lemma~8.1]{MR3259559}, it follows that 
	\begin{equation*}
		U_k\leqslant C(n,s)[w_k]^2_{s} \left|\left\{w_k>0\right\}\right|^{\frac{2s}{n+2s}}.
	\end{equation*}
	As a consequence of this and~\eqref{fdgdfgdf}, we have that 
	\begin{equation*}
		\begin{split}
			U_{k+1}&\leqslant  C(n,s)2^{k+1}U_k(1+\delta)\left|\left\{w_{k+1}>0\right\}\right|^{\frac{2s}{n+2s}}\\
			&\leqslant  C(n,s)2^{k+1}U_k(1+\delta)(2^{2k+2}U_k)^{\frac{2s}{n+2s}}\\
			&\leqslant  \left(C(n,s)2^{1+\frac{4s}{n+2s}}(1+\delta)\right)^{k+1}U_k^{1+\frac{2s}{n+2s}}.
		\end{split}
	\end{equation*}
	
	We set $$\tilde{C}:=\max\left\{C(n,s)2^{1+\frac{4s}{n+2s}}(1+\delta),1\right\}$$ and we point out that 
	\begin{equation}\label{dscvadsc}
		U_k\leqslant\tilde{C}^{k+1}U_{k}^{1+\frac{2s}{n+2s}}.
	\end{equation}  
	Let us pick    \begin{equation}\label{dscsdc}
		\eta:=\tilde{C}^{-\frac{n+2s}{2s}}<1\quad \text{and}\quad \delta:=\tilde{C}^{-\frac{n^2+2ns}{8s^2}}.
	\end{equation}
	We claim that 
	\begin{equation}\label{dsvsd}
		U_k\leqslant\delta^2\eta^k\qquad \text{for any } k\in\mathbb{N}.
	\end{equation}
	
	We show~\eqref{dsvsd} by induction. Indeed, for $k=0$,  recalling~\eqref{dvscsd}, one has that
	\begin{equation*}
		U_0=\|(\tilde{\psi})_+\|^2_{L^2(\mathbb{R}^n)}\leqslant \delta^2.
	\end{equation*}
	Now,  we suppose that~\eqref{dsvsd} is true for $k$ and we prove it
	for $k+1$ by combining~\eqref{dscvadsc} with~\eqref{dscsdc}. For this, we observe that
	\begin{equation*}
		\begin{split}
			U_{k+1}\leqslant \tilde{C}^{k+1}U_k^{1+\frac{2s}{n+2s}}\leqslant \tilde{C}^{k+1}(\delta^2\eta^k)(\delta^2\eta^k)^{\frac{2s}{n+2s}}=\delta^2\eta^k(\tilde{C}\eta^{\frac{2s}{n+2s}})^k \tilde{C}\delta^{\frac{4s}{n+2s}}=\delta^2\eta^{k+1}.
		\end{split}
	\end{equation*}
	This proves the claim~\eqref{dsvsd}.
	
	As a result, employing~\eqref{dscsdc}, one has that 
	\begin{equation*}
		\lim\limits_{k\rightarrow\infty}\|w_k\|^2_{L^2(\mathbb{R}^n)}=	\lim\limits_{k\rightarrow\infty}U_k=0.
	\end{equation*}
	
	Moreover, we notice that 
	\begin{equation*}
		0\leqslant w_k\leqslant |\psi|\qquad \text{for any } k\in\mathbb{N}.
	\end{equation*}
	By the Dominated Convergence Theorem, we deduce that
	\begin{equation*}
		\lim_{k\to+\infty}w_k(x)= (\psi(x)-1)_+=0\qquad \text{a.e. } x\in\mathbb{R}^n,
	\end{equation*}
	thus, $\psi\leqslant 1$ a.e. in $\mathbb{R}^n$. 
	
	By replacing $\psi$ with $-\psi$, we obtain the desired bound in~\eqref{vsvsdv}, which concludes the proof.
\end{proof}

\end{appendix}

\bibliographystyle{is-abbrv}

\bibliography{manuscript}
\vfill
\end{document}